\documentclass[12pt]{amsart}

\newif\ifdebug
\debugfalse

\newif \iffig
\figtrue

\newif \iftable
\tablefalse

\usepackage{sfilip_package_settings}
\usepackage{sfilip_abbreviations}
\usepackage{sfilip_thm_style_long}

\usepackage[left=1in, right=1in, top = 1in, bottom = 1in]{geometry}

\title{Canonical currents and heights for K3 surfaces}

\hypersetup{
	pdftitle={Canonical currents and heights for K3 surfaces},	
	pdfauthor={Simion Filip, Valentino Tosatti},	
	pdfsubject={K3 surfaces, currents, heights},		
	pdfkeywords={MSC Subject Classification System: 37F80, 37P30, 32U40, 32Q20, 14J28, 14J50},	
	pdfnewwindow=true,		
	linkcolor=darkblue,		
	citecolor=darkred,		
	filecolor=darkblue,		
	urlcolor=darkblue,		
	pdfborder={0 0 0},
	breaklinks=true
}

\thanks{\noindent
Revised \textsc{\today}\\
\indent \textsc{MSC2020} Subject Classification System: 37F80, 37P30, 32U40, 32Q20, 14J28, 14J50}

\date{January 2021}

\author{
	Simion Filip
}
\address{
	\parbox{0.7\textwidth}{
		Department of Mathematics\\
		University of Chicago\\
		5734 S University Ave\\
		Chicago, IL 60637}
}
\email{{sfilip@math.uchicago.edu}}

\author{
	Valentino Tosatti
}
\address{
	\parbox{0.7\textwidth}{Courant Institute of Mathematical Sciences\\
New York University\\
251 Mercer St\\
New York, NY 10012}
}
\email{{tosatti@cims.nyu.edu}}

\begin{document}

\begin{abstract}
	We construct canonical positive currents and heights on the boundary of the ample cone of a K3 surface.
	These are equivariant for the automorphism group and fit together into a continuous family, defined over an enlarged boundary of the ample cone.
	Along the way, we construct preferred representatives for certain height functions and currents on elliptically fibered surfaces.
\end{abstract}

%
\maketitle
%
\tableofcontents
\ifdebug
   \listoffixmes
\fi


\section{Introduction}
	\label{sec:introduction}

Consider an algebraic variety $X$ equipped with an automorphism or endomorphism $f$.
Suppose also that $f$ has some hyperbolicity, e.g. the action of $f^*$ on cohomology $H^\bullet(X)$ has spectral radius larger than $1$.
When $X$ is defined over $\bC$, a basic first step in complex dynamics is to construct closed positive currents which are expanded or contracted by $f$.
When $X$ is defined over a global field, such as a number field, one is led to arithmetic dynamics considerations where it is often possible to construct height functions that have analogous equivariance properties under $f$.
The closed positive currents, or rather their potentials, can be viewed as the local heights at the archimedean place associated to the global height function.

This paper studies the case when there is a large group of automorphisms.
While it is possible to apply the standard constructions to each individual group element, taking into consideration the whole group leads to a substantially larger set of currents and heights, which also assemble into a well-behaved object.


\subsection{Main results}
	\label{ssec:main_results_intro}

We consider K3 surfaces since these have some of the most interesting automorphisms groups which are ``large'' in an appropriate sense.
Our assumptions are spelled out in \autoref{sssec:standing_assumptions_main}, but in short, we assume that the Picard rank $\rho$ is at least $3$ and the automorphism group is a lattice in the corresponding $\SO_{1,\rho-1}(\bR)$.
Because we consider automorphisms and not endomorphisms (i.e. our map $f$ is invertible), the canonical currents and heights are defined only for cohomology classes of self-intersection zero and on the boundary of the ample cone.
Furthermore, to assemble them into a well-behaved family, we must introduce a slightly larger space denoted $\partial^\circ \Amp_c(X)$, see \autoref{ssec:boundaries} for precise definitions.
With these preparations, here are our main two results.
In both cases, $X$ is a K3 surface satisfying the assumptions in \autoref{sssec:standing_assumptions_main}.

\begin{theoremintro}[Canonical currents]
	\label{thm:canonical_currents_intro}
	There exists a unique $\Aut(X)$-equivariant map
	\[
		\eta\colon \partial^\circ\Amp_c(X)\to \cZ_{1,1}(X)
	\]
	which is continuous and takes values in closed positive currents.
	It is compatible with taking cohomology classes, the currents in the image have $C^0$ potentials and the map is also continuous for the $C^0$-topology of potentials on currents.
\end{theoremintro}
\noindent See \autoref{thm:continuous_family_of_boundary_currents} for further details and the proof.
Furthermore the currents in the image also have laminarity properties, see \autoref{sssec:laminarity}.

\begin{theoremintro}[Canonical heights]
	\label{thm:canonical_heights_intro}
	Suppose that $X$ is defined over a number field $k$ and let $\ov k$ be an algebraic closure.
	There exists a unique $\Aut(X)$-equivariant map
	\[
		h^{can}\colon \partial^\circ\Amp_c(X)\to \cH(X)
	\]
	taking values in height functions and compatible with the map to the Picard group, such that for any $p\in X\left(\ov k\right)$ the function
	\[
	 	h^{can}(-;p)\colon \partial^\circ \Amp_c(X)\to \bR
	\]
	is non-negative, continuous, and equivariant for the scaling action of $\bR_{>0}$ on both sides.
\end{theoremintro}
\noindent See \autoref{thm:canonical_heights_on_the_boundary} for further details and the proof.
To a given point $p\in X\left(\ov k\right)$ we can associate the star-shaped set in $\partial^\circ \Amp(X)$ where the height is bounded above by $1$.
\autoref{fig:star_shaped_sets} presents some numerical simulations.
In \autoref{thm:invariance_of_total_height} we associate further invariants to this canonical height, which are in particular constant along $\Aut(X)$-orbits.

\begin{figure}[htbp!]
	\centering
	\includegraphics[width=0.32\linewidth]{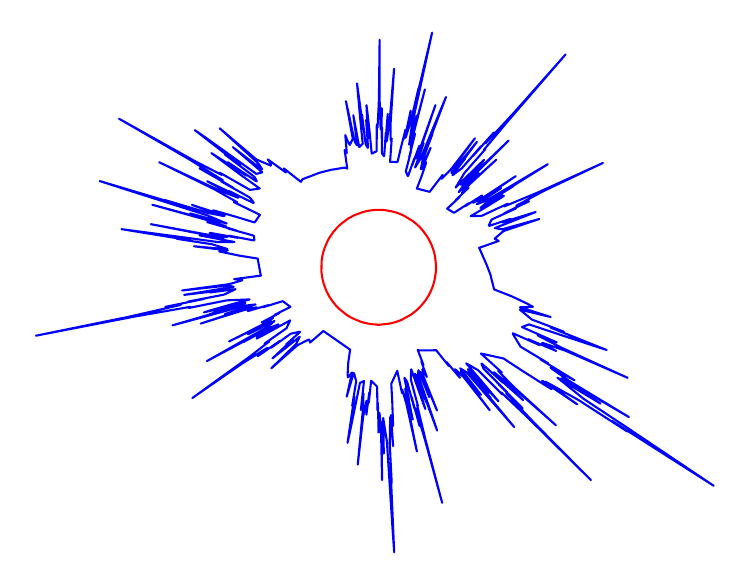}
	\includegraphics[width=0.32\linewidth]{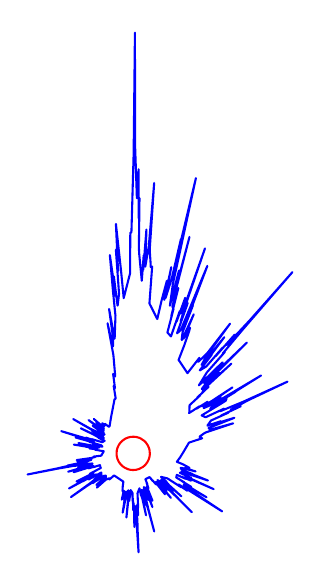}
	\includegraphics[width=0.32\linewidth]{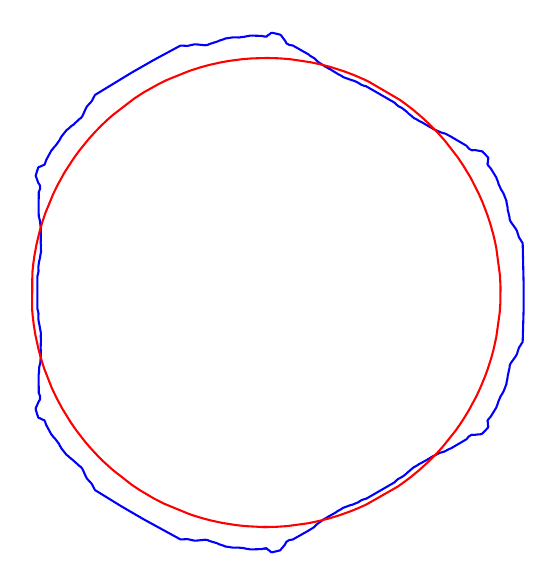}
	\caption{The star-shaped sets associated to canonical heights (in blue).
	The red circles are cohomology classes of normalized mass $1$.
	See \autoref{thm:invariance_of_total_height}.}
	\label{fig:star_shaped_sets}
\end{figure}

\subsubsection*{Uniqueness}
The currents constructed in \autoref{thm:canonical_currents_intro} are the unique closed positive representatives of their cohomology classes, when the cohomology class is not proportional to a rational one.
This is a result of Sibony--Verbitsky, included here as \autoref{thm:uniqueness_in_irrational_classes_verbitsky_sibony}.
This then implies that the map $\eta$ in \autoref{thm:canonical_currents_intro} is unique with its stated properties.
For rational classes, there is always the current that comes from degenerating the Ricci flat metrics.
However, as soon as the Picard rank $\rho$ is at least $4$, one needs to ``fill-in'' a positive-dimensional space of currents for the rational classes, and these are in general distinct from the currents obtained via Ricci-flat metrics.
See \autoref{sssec:an_example_with_elliptic_fibrations} for some examples, and \autoref{ssec:boundary_currents_and_ricci_flat_metrics} for a further discussion.

The heights constructed in \autoref{thm:canonical_heights_intro} are also uniquely characterized by equivariance and continuity for fixed $p\in X(\ov k)$.
We expect that a different characterization, in the spirit of the Sibony--Verbitsky argument (\autoref{thm:uniqueness_in_irrational_classes_verbitsky_sibony}) is possible.
It would involve proving the same uniqueness theorem for positive currents at the non-archimedean places of $k$.

\subsubsection*{Elliptic fibrations}
To establish the above results, we need to obtain accurate estimates for the dynamics of parabolic automorphisms of the K3 surface, i.e. those that preserve an elliptic fibration.
The needed estimates are established with the help of the following intermediary results, which might be of independent interest.

Suppose that $X\xrightarrow{\pi}B$ is a genus one fibration, where $B$ is a curve and $X$ is a surface, all defined over the field $k$.
For currents we take $k=\bC$ and for heights $k$ is a number field.

\begin{theoremintro}[Preferred currents]
	\label{thm:preferred_currents_intro}
	Under the assumptions of \autoref{sssec:standing_assumptions_elliptic_currents} on $X\to B$, suppose that $\alpha$ is a smooth closed $(1,1)$-form on $X$ such that $\int_{X_b}\alpha=0$ i.e. $\alpha$ integrates to zero on some (hence every) fiber.

	Then there exists $\phi\in L^{\infty}(X)$ such that $\alpha+dd^c\phi$ restricts to zero on every smooth fiber of the fibration.
\end{theoremintro}
\noindent For the proof and a more refined statement, see \autoref{thm:preferred_currents}.
In fact $\phi$ can be arranged smooth over the open locus where the fibers are smooth, but the key point is global boundedness of $\phi$ on $X$.
We also obtain a ``current-valued pairing'', see \autoref{thm:current_valued_pairing}.
As a consequence of this construction, we obtain a ``current-valued pairing'' analogous to the results of \cite[\S6]{DeMarcoMyrto-Mavraki2020_Elliptic-surfaces-and-arithmetic-equidistribution-for-R-divisors-on-curves}.

Here is also the height version of \autoref{thm:preferred_currents_intro}:

\begin{theoremintro}[Preferred heights]
	\label{thm:preferred_heights_intro}
	Under the assumptions of \autoref{sssec:standing_assumptions_elliptic_heights}, suppose that $L^0$ is a line bundle on $X$ that restricts to have degree $0$ on every irreducible component of any fiber of the elliptic fibration.

	Then there exists a height function $h^{pf}_{L^0}$ associated to $L^0$ such that $h^{pf}_{L^0}$ restricts on each fiber $X_b$ as a canonical height associated to $L^0\vert_{X_b}$.
\end{theoremintro}
\noindent Note that we say ``a'' canonical height, because a canonical height depends on a choice of basepoint in each fiber, so the assertion is that $h^{pf}_{L^0}\vert_{X_b}$ equals such a canonical height up to some constant which is allowed to depend on $b\in B(\ov k)$.
For the proof and a more refined version, see \autoref{thm:preferred_heights_on_elliptic_fibrations}.
Just like for the preferred currents in \autoref{thm:preferred_currents_intro}, the main difficulty is to obtain a height function defined on all of $X$.




\subsection{Applications of the main results}
	\label{ssec:applications_of_the_main_results}

\subsubsection*{Regularity of boundary currents}
It is well-known that cohomology $(1,1)$-classes on the boundary of the K\"ahler cone of a compact K\"ahler manifold always contain closed positive currents, whose potentials are quasi-plurisubharmonic functions, and an important question is to understand when these currents can be chosen to be sufficiently regular.

This question is in general very hard, even in the case of K3 surfaces. For example, while one might naively expect that classes on the boundary of the K\"ahler (or even ample) cone of a K3 surface should contain a smooth semipositive form, this was recently shown to be false by the authors \cite{FT} using a Kummer rigidity result in dynamics \cite{CD, FT2}. In these examples the class is an eigenclass for a hyperbolic automorphism of the surface.

However, in these dynamical counterexamples the currents have H\"older continuous (but not smooth) potentials.  A direct consequence of \autoref{thm:canonical_currents_intro} is then that every K3 surface that satisfies the assumptions in \autoref{sssec:standing_assumptions_main} has the
property that every class on the boundary of the ample cone contains a closed positive current with continuous potentials.
This confirms the following conjecture of the second-named author \cite[Conjecture 3.7]{Tos_survey} for these surfaces:
\begin{conjectureintro}
Every class $[\alpha]\in H^{1,1}(X,\mathbb{R})$ on the boundary of the K\"ahler cone of a compact Calabi-Yau manifold $X$ contains a closed positive current with bounded potentials.
\end{conjectureintro}
\noindent On K3 surfaces, this conjecture is known when the selfintersection of the class is positive  \cite[Thm. 1.3]{FT}, or when it vanishes and the class is rational (up to scaling) \cite[Prop. 1.4]{FT}. For the remaining irrational classes with vanishing selfintersection, the conjecture was only previously known for the aforementioned eigenclasses of hyperbolic automorphisms.

\subsubsection*{Ricci-flat metrics}
Let $X$ be a K3 surface that satisfies the assumptions in \autoref{sssec:standing_assumptions_main}, equipped with a K\"ahler metric $\omega$ and a closed real $(1,1)$-form $\alpha$ whose class $[\alpha]\in \NS(X)_{\mathbb{R}}$ is on the boundary of the ample cone and has vanishing selfintersection. Thanks to Yau's Theorem \cite{Yau_Ricci}, for any $t>0$ the cohomology class $[\alpha]+t[\omega]$ contains a unique Ricci-flat K\"ahler metric $\omega_t$, which is of the form $\omega_t=\alpha+t\omega+i\partial\overline{\partial}\vp_t$ for smooth potentials $\vp_t$, which can be normalized to have average zero. These metrics are degenerating as $t\to 0$, and understanding their behavior is a problem that has received much attention recently, see for example the recent survey \cite{Tos_survey}.

While the picture is by now quite clear when $[\alpha]$ is a rational class (up to scaling, or equivalently it is pulled back from the base of an elliptic fibration), the behavior of the metrics $\omega_t$ remains mysterious when the class is irrational.
As an application of our results, we are able to shed some light (see \autoref{ssec:boundary_currents_and_ricci_flat_metrics} and in particular \autoref{thm:Linfty_bound_potentials}):
\begin{theoremintro}\label{ricciflat}
Suppose $X$ is a K3 surface that satisfies the assumptions in \autoref{sssec:standing_assumptions_main}, let $\alpha,\omega,\omega_t$ be as above, and suppose $[\alpha]$ is irrational. Then as $t\to 0$ the Ricci-flat metrics $\omega_t$ weakly converge to the canonical current in $[\alpha]$ given by  \autoref{thm:canonical_currents_intro}, and its normalized potentials $\vp_t$ satisfy a uniform $L^\infty$ estimate independent of $t$. Furthermore, the Gromov-Hausdorff limit of $(X,\omega_t)$ is a point.
\end{theoremintro}


\noindent This settles the second author's \cite[Conjecture 3.25]{Tos_survey} for these surfaces. The $L^\infty$ bound also holds when $[\alpha]$ is rational, but in this case it was already known. On the other hand, when $[\alpha]$ is rational (and nontrivial), the Gromov-Hausdorff limit is a certain metric on $\mathbb{P}^1$, and the weak limit as currents need not coincide with the canonical current from \autoref{thm:canonical_currents_intro}, as we shall explain in \autoref{sssec:rational_boundary_points} and \autoref{sssec:an_example_with_elliptic_fibrations}.

\subsubsection*{Small points and bounded orbits}
In \autoref{ssec:small_points_and_bounded_orbits} we include some simple applications of the canonical heights from \autoref{thm:canonical_heights_intro}.
Here is one, see \autoref{thm:zero_height_implies_bounded_orbit_irrational_case}:
\begin{theoremintro}[Bounded orbit equivalent to vanishing height]
	\label{thm:bounded_orbit_equivalent_to_vanishing_height_intro}
	Suppose that $p\in X(\ov k)$ and $[\alpha]\in \partial^\circ \Amp_c(X)$ is an irrational direction (see \autoref{eqn:boundary_stratification}).
	Let $\{\gamma_i\}_{i\geq 1}$ be the sequence of automorphisms associated to $[\alpha]$ and the generators, by the construction in \autoref{sssec:fixed_finite_set_of_generators}.

	Then $h^{can}([\alpha];p)=0$ if and only if the set $\{\gamma_n\cdots \gamma_1 p\}_{n\geq 1}$ is finite.
\end{theoremintro}
\noindent We also include in \autoref{thm:kernel_of_quadratic_form_and_torsion} some results about finiteness of orbits along elliptic fibrations, with some mild adaptations to our case where the existence of a section is not assumed.
We will return to further arithmetic applications of these facts in another text.

In \autoref{thm:invariance_of_total_height} we associate an invariant to $\Aut(X)$-orbits of rational points:
\begin{theoremintro}[Total height]
	\label{thm:total_height_intro}
	For any $p\in X(\ov k)$ the natural volume of
	\[
	 	\{[\alpha]\in \partial^\circ \Amp_c(X)\colon h^{can}([\alpha];p)\leq 1\}
	\]
	 is constant on the $\Aut(X)$-orbit of $p$.
\end{theoremintro}
\noindent In \autoref{fig:star_shaped_sets} the set in question is bounded by the blue curve.

\subsubsection*{Suspension space}
Associated to an algebraic K3 surface is a natural finite-volume hyperbolic manifold.
We introduce in \autoref{sec:suspension_space_construction} a ``suspended space'' over this hyperbolic manifold and consider associated dynamical systems.
Some elementary results are included in \autoref{ssec:g_dynamics} and \autoref{ssec:a_dynamics}.
Some more challenging questions are formulated there as well; the one we find most interesting concerns unipotent flows, see \autoref{eg:u_invariant_implies_p_invariant}.
Let us note that the results of Cantat--Dujardin \cite{CantatDujardin2020_Stationary-measures-on-complex-surfaces} are very closely related to this suspended space, although they do not seem to translate directly to this setting (see however \autoref{sssec:relation_to_p_dynamics}).

The currents provided by \autoref{thm:canonical_currents_intro} can be used to construct \emph{the} measure of maximal entropy associated to the K3 surface.
This will be treated in another text.



\subsection{Context and related works}
	\label{ssec:context_and_related_works}

\subsubsection*{Relation to the work of Cantat--Dujardin}
The two recent papers of Cantat and Dujardin are concerned with closely related topics.
In \cite[Thm.D]{CantatDujardin2020_Stationary-measures-on-complex-surfaces} they construct closed positive currents on the boundary of the ample cone, for some irrational classes, obtained as limit points of random walks on $\Aut(X)$.
They also obtain uniqueness, by a method similar to that of Sibony--Verbitsky, as well as \Holder potentials.

It is convenient to compare our results using the language of hyperbolic geometry and the word metric on the group $\Aut(X)$.
See \autoref{sec:hyperbolic_geometry_background} for some of the notation.
It is shown in \cite{GadreMaherTiozzo2015_Word-length-statistics-and-Lyapunov-exponents-for-Fuchsian-groups} that, provided there is at least one parabolic automorphism, for a Lebesgue-typical geodesic the cusp excursions are quite long.
Specifically they prove that
\[
	\lim \frac{\dist_{word}(1,\gamma)}{\dist_{hyp}([\omega_0],\gamma[\omega_0])} = +\infty
\]
where $[\omega_0]$ is a basepoint in hyperbolic space, and $\dist_{word},\dist_{hyp}$ denote the word length distance on $\Aut(X)$ and hyperbolic distance respectively.
On the other hand, for random walks the above limit is finite.
This implies that on the boundary, harmonic measure for random walks is singular for Lebesgue measure.
In particular, the set of cohomology classes covered by \cite[Thm.~D]{CantatDujardin2020_Stationary-measures-on-complex-surfaces} is of zero Lebesgue measure.

In order to access the Lebesgue-generic points on the boundary of the ample cone, one needs to separately treat the parabolic automorphism and obtain sharper estimates.
This is in fact the bulk of the work in our case, both for currents (\autoref{sec:elliptic_fibrations_and_currents}) as well as for heights (\autoref{sec:elliptic_fibrations_and_heights}).

\subsubsection*{Relation to the work of Baragar}
	\label{sssec:work_of_baragar}
Let us additionally remark that Baragar \cite{Baragar1996_Rational-points-on-K3-surfaces-in-bf-P1bf-P1bf-P1} has considered some related constructions related to heights on K3 surfaces.
With van Luijk \cite{BaragarLuijk2007_K3-surfaces-with-Picard-number-three-and-canonical-vector} they provide numerical evidence that a much stronger notion of canonical height does not exist.
This was established by Kawaguchi \cite{Kawaguchi2013_Canonical-vector-heights-on-K3-surfaces---a-nonexistence-result} for certain K3 surfaces and in greater generality by Cantat and Dujardin in \cite{CantatDujardin2020_Finite-orbits-for-large-groups-of-automorphisms-of-projective-surfaces} for a large class of surfaces.
Specifically, their definition of canonical heights would require an $\Aut(X)$-equivariant assignment of heights to any class in $\NS(X)$, not just the ones in the isotropic cone, and which is furthermore linear for the additive structure on $\NS(X)$.

Let us note that because of the necessity to consider the blown-up boundary $\partial^\circ\Amp_c(X)$, i.e. the non-uniqueness of canonical heights in rational boundary directions if $\rk \NS(X)\geq 4$, under this rank assumption and provided the existence of at least one elliptic fibration it follows from the construction in this text that canonical heights in this strong sense cannot exist.

\subsubsection*{Work of Sibony--Verbitsky}
On K3 surfaces (and \hyperkahler manifolds) Sibony and Verbitsky \cite{VerbitskySibony} have announced a quite general uniqueness result for positive currents of self-intersection zero.
We include a special case which is relevant for our work in \autoref{thm:uniqueness_in_irrational_classes_verbitsky_sibony}, proving it using the methods described to the first-named author by Verbitsky.

\subsubsection*{Relation to work of DeMarco--Mavraki}
The basic results on how the canonical height moves in a family of elliptic curves were obtained by Tate \cite{Tate} and Silverman \cite{Silverman1994_Variation-of-the-canonical-height-in-algebraic-families,Silverman1992_Variation-of-the-canonical-height-on-elliptic-surfaces.-I.-Three-examples}.
In their recent work, DeMarco and Mavraki \cite{DeMarcoMyrto-Mavraki2020_Elliptic-surfaces-and-arithmetic-equidistribution-for-R-divisors-on-curves} further enhanced those results, showing in particular that the relevant height functions come from adelically metrized line bundles and enjoy equidistribution properties.

In the present text, we adapted the classical results of Tate and Silverman to our situation of a family of genus one curves with automorphisms, lacking in general a section.
One distinction is our \autoref{thm:preferred_heights_on_elliptic_fibrations} yielding preferred heights on the total space of the fibration.
These objects could be of independent interest and we expect they also come from adelically metrized line bundles.

\subsubsection*{Betti map}
Recall that given an elliptic fibration $X\to B$ over $\bC$ (or more generally a family of abelian varieties), one has a canonical ``Betti form'' $\omega_{Betti}$ on the locus of smooth fibers $X^{\circ}$ which in particular has the property that it restricts to the flat metric on each fiber.
This object, which apparently originated in \cite[Lemma IV.8.5]{Satake}, was later rediscovered in \cite[pp.24-25]{GSVY} in the context of noncompact Calabi-Yau manifolds, and has since been extensively used in the study of degenerating Calabi-Yau metrics where it is dubbed a ``semi-flat form'' (see e.g. \cite{GW,GTZ,HT}, and also \cite[Theorem 3.1]{TZ} which gives a characterization of this form that also shows that it coincides with the Betti form), as well as in arithmetic considerations, see e.g. \cite{AndreCorvajaZannier2020_The-Betti-map-associated-to-a-section-of-an-abelian-scheme}.

However, it is typically not possible to extend the Betti form to an object with useful regularity on all of $X$.
Analogously to the Betti form, one also has the \Neron--Tate cannonical heights on each fiber of the fibration, which again typically do not extend to a height function on all of $X$.

Our constructions of preferred currents and heights in \autoref{thm:preferred_currents_intro} and \autoref{thm:preferred_heights_intro} can be viewed as one way to address this issue.
If we consider forms, or heights, that have degree $0$ on the fibers, then it is possible to obtain objects global on $X$ with some regularity.



\subsection{Assumptions}
	\label{ssec:assumptions}

\subsubsection{Standing assumptions}
	\label{sssec:standing_assumptions_main}
Throughout this paper, we make the following simplifying assumptions.
Our K3 surface $X$ is algebraic, of Picard rank $\rho\geq 3$, and contains no $(-2)$ curves.
A result of Sterk \cite{Sterk_Finiteness-results-for-algebraic-K3-surfaces} implies that in this case $\Aut(X)$ maps (with finite kernel) to a lattice in the orthogonal group $\Orthog_{1,\rho-1}(\bR)$.
This also implies that all singular fibers of elliptic fibrations on $X$ are of Kodaira type $I_1$ or $II$.

In \autoref{sec:elliptic_fibrations_and_currents}, which treats currents in elliptic fibrations, these assumptions are relaxed to allow any elliptically fibered K\"ahler surface with singular fibers of Kodaira type $I_1$ or $II$ (see \autoref{sssec:standing_assumptions_elliptic_currents}).
In \autoref{sec:elliptic_fibrations_and_heights}, which treats heights in elliptic fibrations, we allow quite general elliptic fibrations (see \autoref{sssec:standing_assumptions_elliptic_heights}).
In particular, the results in these two sections do not assume that $X$ is a K3 surface.

In the results on heights we assume that the K3 surface is defined over a finite extension of $\bQ$.

\subsubsection{Relaxing the standing assumptions}
	\label{sssec:relaxing_the_standing_assumptions}
As it will become apparent from the method of proof, there is no need for most results to make the assumption on the absence of $(-2)$ curves on the K3 surface.
We expect that the main results on the existence of equivariant currents and heights continue to hold for any projective K3 surface, and will consider the general case in subsequent work.

Let us note also that there are plenty of examples satisfying our standing assumptions, for instance Wehler surfaces \cite{Wehler1988_K3-surfaces-with-Picard-number-2} defined by an equation of bidegree $(2,2,2)$ in $\bP^1\times \bP^1\times \bP^1$.


\subsubsection*{Acknowledgments}
We are grateful to Misha Verbitsky for telling us about his proof of uniqueness of closed positive currents.
Ian Agol informed us about the notion of truncated hyperbolic space, appearing in \autoref{sec:hyperbolic_geometry_background}, whose boundary is the relevant object for many of our constructions.
We have also benefited from conversations with Serge Cantat, Romain Dujardin, Bertrand Deroin, and Hans-Joachim Hein.
We are grateful to Serge Cantat for extensive feedback that significantly improved the exposition, and to Jeff Diller, Curt McMullen and the referee for further feedback.

This research was partially conducted during the period SF served as a Clay Research Fellow.
Both authors gratefully acknowledges support from the Institute for Advanced Study for an excellent research environment and for hosting SF and VT when a major part of this work was completed. VT would also like to thank the Department of Mathematics and the Center for Mathematical Sciences and Applications at Harvard University for their hospitality.
This material is based upon work supported by the National Science Foundation under Grant No. DMS-2005470, DMS-1638352 (SF) and DMS-2231783, DMS-1903147, DMS-1610278 (VT).



\section{Hyperbolic geometry background}
	\label{sec:hyperbolic_geometry_background}

\subsubsection*{Outline of section}
We describe some of the geometric properties of the ample cone in the language of hyperbolic geometry, since we find it more convenient.
Specifically, in \autoref{ssec:groups_and_spaces} we introduce the groups and several flavors of hyperbolic spaces that we consider.
In \autoref{ssec:boundaries} we introduce the several flavors of boundaries that will appear in our results.
Finally, in \autoref{ssec:linear_algebra_and_hyperbolic_geometry} we include some of the linear algebra estimates that are equivalent to geometric properties in hyperbolic space.
For a general introduction to hyperbolic geometry, the reader can consult the introductory notes \cite{CannonFloydKenyon1997_Hyperbolic-geometry} and the more comprehensive survey \cite{Vinberg1993_Geometry.-II}.


\subsection{Groups and spaces}
	\label{ssec:groups_and_spaces}

\begin{wrapfigure}{L}{0.5\textwidth}
	\centering
	\includegraphics[width=\linewidth]{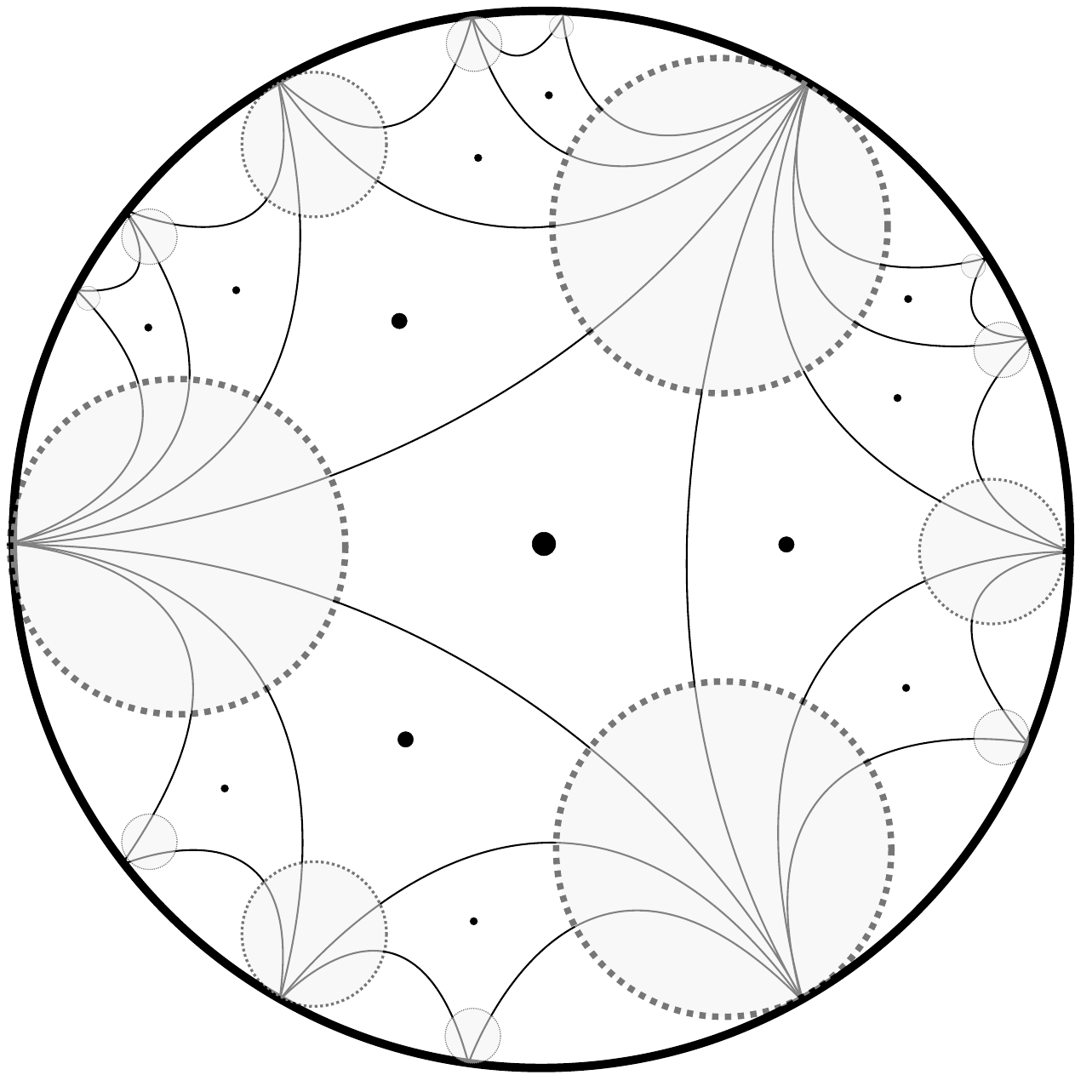}
	\caption{A fundamental domain with horoballs at the cusps, and a reference point in the fundamental domain.}
	\label{fig:tesselation}
\end{wrapfigure}

\subsubsection{Setup}
	\label{sssec:setup_groups_and_spaces}
Denote by $N:=\NS(X)$ the \Neron--Severi group of the K3 surface $X$, with $N_{\bZ}$ the corresponding lattice and $N_{\bR}$ the extension of scalars to $\bR$.
The intersection pairing on it has signature $(1,\rho-1)$ where $\rho=\rk N$.
By our standing assumptions on $X$ in \autoref{sssec:standing_assumptions_main}, the group $\Aut(X)$ is a lattice in the real orthogonal group $\Orthog_{1,\rho-1}(N_{\bR})$.
We denote by $\Amp(X)\subset N_{\bR}$ the ample cone of $X$, which again by our standing assumptions is one of the connected components of the set of vectors $v\in N_{\bR}$ satisfying $v.v>0$.
We will further denote by $\Amp^1(X)$ the subset of ample classes satisfying $v.v=1$.
Then equipped with the (negative of) the induced metric, $\Amp^1(X)$ is isometric to hyperbolic $(\rho-1)$-space.

Define next $	\cM = \leftquot{\Aut(X)}{\Amp^1(X)}$ which is a finite-volume hyperbolic orbifold of dimension at least $2$.
The arguments in the rest of the text are insensitive to passing to a finite-index subgroup, so we assume from now on that $\Aut(X)$ is torsion-free.
In particular, cusp stabilizers (see below) can be assumed to be free abelian groups of rank $\bZ^{\rho-2}$.

\subsubsection{Cusps of $\cM$}
	\label{sssec:cusps_of_cm}
Under our standing assumptions, by \cite[Prop.~11.1.3, Rmk.~8.2.13]{Huy}, see also \cite[Prop. 2.1.12]{Filip2020_Counting-special-Lagrangian-fibrations-in-twistor-families-of-K3-surfaces}, given a vector $[E]\in N_{\bZ}$ which is primitive and satisfies $[E].[E]=0$, precisely one of $\pm [E]$ can be represented by a smooth genus one curve $E$ which furthermore determines a genus one fibration on $X$ with base $\bP^1$.
We always take $[E]$ to denote the class represented by a genus one curve.

The hyperbolic manifold $\cM$ need not be compact and its cusps are in bijection with $\Aut(X)$-orbits of vectors $[E]$ representing genus one fibrations.
Call any such class $[E]$ a \emph{parabolic boundary point}.
It determines for each $c>0$ an (open) horoball in $\Amp^1(X)$ defined by
\[
	H_{[E],c}:=\{[\omega]\in \Amp^1(X)\colon [\omega].[E]<c\}.
\]
Taking the complements of these horoballs yields
\begin{align}
	\label{eqn:amp_c_defn}
	\begin{split}
	\Amp^1_c(X) & :=\{[\omega]\in \Amp^1(X)\colon [\omega].[E]\geq c\\
	& \qquad \qquad \forall [E]\text{ parabolic boundary point}\}\\
	\cM_c & = \leftquot{\Aut(X)}{\Amp^1_c(X)}
	\end{split}
\end{align}
Take $c>0$ sufficiently small such that the inclusion $\cM_c\into \cM$ induces an isomorphism $\pi_1(\cM_c)\toisom \pi_1(\cM)$ and the complement $\cM\setminus \cM_c$ is a finite disjoint union of ``standard hyperbolic cusps'', i.e. quotients of a horoball by an abelian subgroup of cofinite volume.
We regard $c>0$ as fixed from now on, and picked sufficiently small for later parts of the argument.

For each parabolic class $[E]$ let $\Gamma_{[E]}\subset \Aut(X)$ be its stabilizer.
It is a group commensurable to $\bZ^{\rho-2}$.
Since we passed to a finite index subgroup of $\Aut(X)$ to remove torsion, we can assume that $\Gamma_{[E]}\isom \bZ^{\rho-2}$, which we now do.

\subsubsection{Fundamental domain, basepoint}
	\label{sssec:fundamental_domain_basepoint}
We fix a basepoint $[\omega_0]\in \Amp^1_c(X)$ and its associated Dirichlet fundamental domain $F_{[\omega_0]}$.
This fixes also representatives of $\Aut(X)$-equivalences class of cusps, given by finitely many parabolic points $[E_1],\ldots, [E_p]$, and their corresponding parabolic automorphism groups.
Denote by $F_{[\omega_0],c}:=F_{[\omega_0]}\cap \Amp^1_c(X)$ the compact part of the fundamental domain.

\subsubsection{Recurrent and divergent rays}
	\label{sssec:recurrent_and_divergent_rays}
Let $s$ be any hyperbolic ray starting in the fundamental domain $F_{[\omega_0],c}$.
When projected to the finite volume hyperbolic manifold $\cM$, either $s$ returns infinitely often to the compact set $\cM_c$, or not.
Rays that return to compact sets are called recurrent, and divergent otherwise.
If the endpoint of the ray $s$ is a parabolic boundary point, i.e. proportional to a rational vector, then the ray is divergent, and otherwise the ray is recurrent.

\begin{figure}[htbp!]
	\centering
	\includegraphics[width=1.0\linewidth]{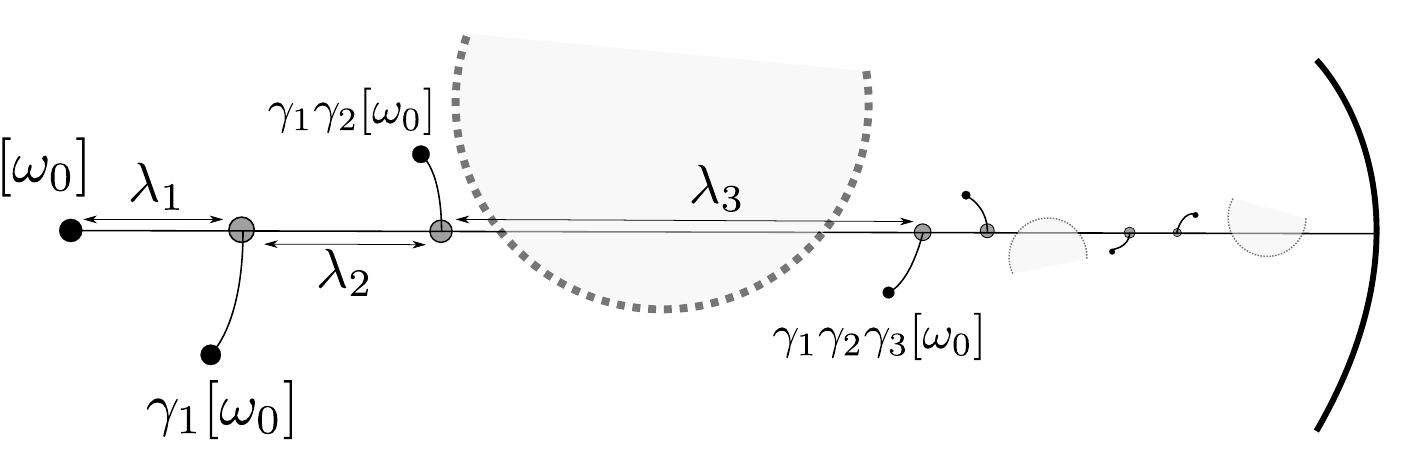}
	\caption{A recurrent geodesic.
	The centers of fundamental domains along the way, and cusp excursions.
	The parameters $\lambda_i$ as the distances between the projections of the centers of fundamental domains to the geodesic.}
	\label{fig:geodesic}
\end{figure}

We will be interested in parametrizing the list of translates of the compact part of the fundamental domain $\gamma F_{[\omega_0],c}$, that are encountered by geodesic rays.

\subsubsection{Fixed finite set of generators}
	\label{sssec:fixed_finite_set_of_generators}
Recall that we have the fixed representatives for each cusp and parabolic groups $\Gamma_{[E_i]}$ with $i=1\ldots k$.
We fix a finite set of generators $S\subset \Aut(X)$ (or rather, its finite index subgroup that we are considering) such that the fundamental domain $F_{[\omega_0],c}$ can be taken to any of its adjacent translates by one of the generators.
Then we have the following geometric construction, which writes the elements of the fundamental group as products of generators and parabolic elements in $\Gamma_{[E_i]}$.

Let $s$ be any hyperbolic ray starting in the fundamental domain $F_{[\omega_0],c}$.
Any point on the geodesic $s$ is contained in some translate of the fundamental domains $F_{[\omega_0]}$, and this divides up $s$ into segments.
Given one such segment, one possibility is that it, as well as the next one, are contained in translates of the compact part $F_{[\omega_0],c}$.
Then we will use a generator from $S$ to relate the two adjacent fundamental domains.
The other possibility is that there is a sequence of fundamental domains that $s$ encounters, such that the corresponding segments in $s$ are contained in translates of $F_{[\omega_0]}\setminus F_{[\omega_0],c}$.
The union of those segments is entirely contained in a horoball associated to a specific parabolic boundary point $[E]$.
If the ray $s$ eventually exits the horoball, then we can connect the first and the last fundamental domains by an element of $\Gamma_{[E]}$.

To summarize, attached to $s$ there exists a sequence $\{\gamma_i\}$ with $\gamma_{i}\in S\bigcup_{i=1}^p \Gamma_{[E_i]}$, where $[E_1],\ldots,[E_p]$ are representatives of the $\Aut(X)$-orbits on parabolic points, with the following properties:
\begin{enumerate}
	\item The sequence of points $\pi_s(\gamma_1\cdots \gamma_n[\omega_0])\in s$ on $s$ are a bounded hyperbolic distance from $\gamma_1\cdots \gamma_n[\omega_0]$.
	Denote by $\lambda_i$ the (oriented) distance between the successive projections (see \autoref{fig:geodesic}).
	Specifically
	\[
		\lambda_i:=\pm\dist(\pi_s(\gamma_1\cdots \gamma_i[\omega_0]), \pi_s(\gamma_1\cdots \gamma_{i+1}[\omega_0]))
	\]
	where the sign is $+$ if the points come in the written order as seen from the origin of the ray $s$ and $-$ otherwise.
	The displacements $\lambda_i$ are typically positive (see \autoref{prop:displacement_estimate}).
	\item If the endpoint of $s$ is rational (the geodesic is divergent), then the sequence $\{\gamma_i\}$ is finite (with $n$ elements) and there is an index $1\leq k(s)\leq p$ such that the geodesic $(\gamma_1\cdots \gamma_n)^{-1}s$ goes into the rational endpoint corresponding to $[E_{k(s)}]$.
	\item Furthermore, for any other geodesic ray $s'$, given any $n\geq 1$ there exists $\ve>0$ such that if the distance between the endpoints of $s,s'$, as seen in the visual metric from $[\omega_0]$ is less than $\ve$, then we can choose the corresponding sequence $\{\gamma_i(s')\}$ such that $\gamma_i(s)=\gamma_i(s')$ for all $i\leq n$.
\end{enumerate}
In the case when $s$ is rational, say with sequence $\gamma_1,\cdots, \gamma_n$, the last requirement is to be understood as saying that $\gamma_{n+1}(s')\in \Gamma_{[E_{i(s)}]}$ and given any $R>0$, we can pick $\ve>0$ such that $\norm{\gamma_{n+1}(s')}_{\NS}\geq R$ (see \autoref{sssec:summary_elliptic_fibrations} for $\norm{\bullet}_{\NS}$).
In other words, by picking $\ve>0$ sufficiently small, we can arrange that the cusp excursion is sufficiently long.
This last assertion can be verified in the upper half-space model $\bR^{\rho-2}\times \bR_{+}$ of hyperbolic space, by placing the parabolic point at infinity.
Then if $s$ is going into the parabolic point, it must be a vertical geodesic, and a neighborhood of the parabolic boundary point is given by the set of points $p\in \bR^{\rho-2}$ satisfying $\norm{p}\geq A$.
By taking $A$ sufficiently large, i.e. the neighborhood sufficiently small, we can ensure that the geodesic excursion into the cusp is sufficiently long.

\begin{proposition}[Displacement estimate]
	\label{prop:displacement_estimate}
	There exist constants $C_0, \delta>0$ such that for any geodesic ray $s$ and sequence of displacements $\lambda_1,\ldots, \lambda_n,\ldots$ as above, we have the estimate for any $a,N\geq 1$:
	\[
		\delta \cdot N - C_0 \leq \lambda_a + \cdots + \lambda_{a+(N-1)}
	\]
\end{proposition}
\begin{proof}
	Let $\gamma_1,\ldots, \gamma_n,\ldots$ be the generators corresponding to the fundamental domains visited by the ray $s$.
	On the automorphism group $\Aut(X)$, let $d^{orb}$ denote the distance function obtained from embedding it to $\Amp^1$ via the orbit of $[\omega_0]$, and restricting the hyperbolic distance function.
	Then we have
	\[
		d^{orb}(\gamma_1\cdots \gamma_a , \gamma_1\cdots \gamma_{a+N}) \leq \lambda_a+\cdots \lambda_{a+(N-1)} - C'
	\]
	for some fixed constant $C'$, since the centers of the fundamental domain and their projections to $s$ are a bounded distance away.

	On the other hand we have the distance function $d^{orb,c}$ on $\Aut(X)$ obtained from the embedding into $\Amp_1^c$ with its path metric, so we clearly have $d^{orb}\geq d^{orb,c}$.
	Finally, we have the estimate
	\[
		d^{orb,c}(\gamma_1\cdots \gamma_a , \gamma_1\cdots \gamma_{a+N})
		\geq
		\frac{1}{D}N - C''
	\]
	where $D,C''$ depends on the diameter of the compact part of the fundamental domain.
	Concatenating the bounds yields the result.
\end{proof}



\subsection{Boundaries}
	\label{ssec:boundaries}

\subsubsection{Truncated hyperbolic space}
	\label{sssec:truncated_hyperbolic_space}
The space $\Amp^1_c(X)$ defined in \autoref{eqn:amp_c_defn} is a \emph{truncated hyperbolic space}, see e.g. \cite[II.11]{BridsonHaefliger1999_Metric-spaces-of-non-positive-curvature} and \cite{Ruane2005_CAT0-boundaries-of-truncated-hyperbolic-space}.
We will quote some useful results from \cite[p.362]{BridsonHaefliger1999_Metric-spaces-of-non-positive-curvature}.
First, it is a complete, $\CAT(0)$, uniquely geodesic metric space.
The geodesics are described in \cite[Cor.~11.34]{BridsonHaefliger1999_Metric-spaces-of-non-positive-curvature} as a concatenation of hyperbolic geodesics in the interior of $\Amp^1_c(X)$, and flat geodesics in the boundary horospheres, such that the successive pieces meet tangentially.

Observe that the horoball associated to a class $[E]$ is defined by $[\omega].[E]<c$ and its boundary is naturally an affine space over $[E]^{\perp}/[E]$, where $[E]^{\perp}\subset N$ denotes the orthogonal to $[E]$ for the intersection pairing, and hence contains $[E]$ since $[E]$ is isotropic.
In particular the quotient $[E]^{\perp}/[E]$ is of dimension $\rho-2$ and carries a non-degenerate negative-definite inner product (for which we flip the sign to make it positive-definite).
In particular, the visual boundary of the horosphere is naturally the unit sphere $\bS([E]^{\perp}/[E])$.

\subsubsection{The boundary of truncated hyperbolic space}
	\label{sssec:the_boundary_of_truncated_hyperbolic_space}
A description of geodesics also gives a description of the visual $\CAT(0)$-boundary, see \cite[II.8]{BridsonHaefliger1999_Metric-spaces-of-non-positive-curvature} for the definition of visual boundary.
According to \cite[Thm.~3.4]{Ruane2005_CAT0-boundaries-of-truncated-hyperbolic-space} the boundary $\partial \Amp^1_c(X)$ is homeomorphic to a sphere $\bS^{\rho-2}$ with a countable collection of dense pairwise disjoint open discs removed, one disc for each cusp.
Recall that under our assumptions on the K3 surface, each cusp is identified with a rational point on the boundary of the full hyperbolic space $\Amp^1(X)$.
Thus set-theoretically we have that
\[
	\partial \Amp^1_c(X) = \partial \Amp^1(X)^{irr} \coprod_{[E]}\bS\left([E]^{\perp}/[E]\right)
\]
where $\partial \Amp^1(X)^{irr}$ denotes the irrational points on the boundary of the full hyperbolic space, and $\bS\left([E]^\perp/[E]\right)$ denotes the unit sphere associated to the quotient space with its natural nondegenerate (negative) definite inner product.
This set-theoretical decomposition is equivariant for the action of $\Aut(X)$.

Let us finally note that for each point $p\in\partial \Amp^1_c(X)$, we have a sequence of translates $\gamma_1\cdots \gamma_n[\omega_0]$ of the basepoint which approach it, such that the elements $\gamma_i$ belong to a fixed finite set $S$ or to one of finitely many parabolic groups $\Gamma_{[E]}$, as in \autoref{sssec:fixed_finite_set_of_generators}.
If $p$ is not on one of the spheres, then the index $n$ goes to infinity.
If $p$ is on one of the spheres, then the index eventually remains constant equal to $n$ and only the last element changes, say $\gamma_{n,i}$ with $n$ fixed and $i\to +\infty$.
Furthermore, for each parabolic boundary point $[E]$ we will construct in \autoref{sssec:linearizing_automorphisms} an injection $\xi\colon \Gamma_{[E]}\into [E]^{\perp}/[E]$ and denote by $\norm{-}_{\NS}$ the positive-definite norm on $[E]^\perp/[E]$ induced from the intersection pairing and flipping the sign.
With this identification we have that $\frac{\xi(\gamma_{n,i})}{\norm{\xi(\gamma_{n,i})}_{\NS}}$ approaches $\left(\gamma_1\cdots \gamma_{n-1}\right)^{-1}p$ on the sphere of $[E]^{\perp}/[E]$, and furthermore $\norm{\xi(\gamma_{n,i})}_{\NS}\to +\infty$.

\subsubsection{The projectivized visual boundary}
	\label{sssec:the_projectivized_visual_boundary}
We will need a variant of the above boundary, where we identify antipodal points on spheres, namely at the set-theoretical level we have:
\begin{align}
	\label{eqn:boundary_stratification_projectivized}
	\partial^{\circ}\Amp^{1}_c(X):= \partial \Amp^1(X)^{irr} \coprod_{[E]}\bP\left([E]^{\perp}/[E]\right)
\end{align}
To give this space a topology, we describe it alternatively as a projective limit of blow-ups of the rational points on the boundary:
\[
	\partial^{\circ}\Amp^{1}_c(X) = \varprojlim_{R\to +\infty}  \Bl_{\cC_R} \partial \Amp^1(X)
\]
where $\cC_R$ denotes the cusp points whose representatives fiber classes $[E]$ satisfy $[\omega_0].[E]\leq R$.
This projective limit is independent of the choice of basepoint $[\omega_0]$ since the function on $\partial \Amp(X)$ given by $v\mapsto [\omega_0].v$ is comparable, up to multiplicative constants, with one given by another class $[\omega_0']$.
This description of the boundary is also closer in spirit to the one from \cite[II.8.5]{BridsonHaefliger1999_Metric-spaces-of-non-positive-curvature}.
Note that to give a similar description of the visual boundary $\partial \Amp^1_c(X)$, we must simply replace blow-ups by real-oriented blowups.

\subsubsection{Blown-up boundary of the ample cone}
	\label{sssec:blown_up_boundary_of_the_ample_cone}
In order to state our main result, we must introduce a version of the space $\partial^\circ \Amp^1_{c}(X)$ from \autoref{sssec:the_projectivized_visual_boundary} that also includes scalings.
Denote by $\partial \Amp(X)\subset \NS(X)_{\bR}$ the boundary of the ample cone (with the origin removed).
Under our assumptions on the K3 surface, this is equal to one of the components of $\{v\in \NS(X)_{\bR}\setminus 0\colon v.v=0\}$, i.e. a cone formed by null rays.

Let us define then
\[
	\partial^\circ \Amp_c(X):= \varprojlim_{R\to +\infty} \Bl_{\cE_R}  \partial \Amp(X)
\]
where $\cE_R$ denotes the (finitely many) rays spanned by a primitive integral vector $[E]$ satisfying $[\omega_0].[E]\leq R$ and $\Bl_{\cE_R}$ denotes the blowup of the boundary of the ample cone along the given rays.
Note that the projective limit and induced topology are independent of the choice of $[\omega_0]$ and the induced map
\begin{align}
	\label{eqn:fullbdy_to_projectivized}
	\partial^\circ \Amp_c(X) \xrightarrow{\bullet/\bR_{>0}}
	\partial^\circ \Amp^1_c(X)
\end{align}
obtained by quotienting by positive rescaling is continuous.
Finally, note that every point in $\partial^\circ\Amp_c(X)$ determines a cohomology class via the natural projection to $\partial\Amp(X)\subset \NS(X)_\bR$.

We also have the stratification analogous to \autoref{eqn:boundary_stratification_projectivized}
\begin{align}
	\label{eqn:boundary_stratification}
	\partial^{\circ}\Amp_c(X):= \partial \Amp(X)^{irr} \coprod_{[E]}
	\Big(\bR_{>0}[E]\Big)\times\bP\Big([E]^{\perp}/[E]\Big)
\end{align}
where $\bR_{>0}[E]$ denotes the ray spanned by $[E]$.
To specify what it means for a function on $\partial^{\circ}\Amp_c(X)$ to be continuous, it suffices to specify what it means for a sequence of points $p_i$ in it to converge to another one $p$.
Let us denote by $[p_i],[p]$ their projections under $\partial^\circ\Amp_c(X)\to \partial \Amp(X)$, i.e. the corresponding cohomology classes.
If $p$ is in the irrational part, for convergence it suffices to have that the projections $[p_i]$ converge to the projection $[p]$.
If the point is in one of the rational strata, of the form say $p=[E]\times [\xi]$ (or some rescaling of it), then we must have that $[p_i]$ converge to $p$ in the blow-up $\Bl_{\bR_{>0}[E]}\partial \Amp(X)$ of the ray spanned by $[E]$.
Note that $p$ is by definition a point on the exceptional divisor of the blow-up.

To prove continuity of various quantities on the boundary, we will fix an ample class $[\omega_0]$ and use it to identify $\partial^\circ \Amp_c^1(X)$ with the subset of $\partial^\circ \Amp_c(X)$ formed of classes that have cup-product $1$ with $[\omega_0]$.
In that case, convergence of points on the boundary translates directly to geometric statements about geodesic rays in hyperbolic space, see \autoref{sssec:fixed_finite_set_of_generators}.

\subsubsection{Some examples}
	\label{sssec:some_examples}
We describe the boundaries discussed above in some special cases, namely $\rho=3,4$.
Since the boundaries for the ample cone (without a normalization) are topologically just $\bR$ times the boundaries associated to hyperbolic space, we will only describe the latter.
We always assume that there are at least some parabolic points, otherwise all constructions reduce to usual boundaries.

When $\rho=3$, so $\Amp^1(X)$ is hyperbolic $2$-space, the usual boundary $\partial \Amp^1(X)$ is just a circle.
The oriented blown-up boundary $\partial \Amp^1_c(X)$ is a Cantor set, obtained from the circle by inserting an open set at the countable many rational points on the circle.
The projectively blown-up boundary $\partial^\circ \Amp^1_c(X)$ identifies pairs of points separated by a gap in the Cantor set, recovering the original circle boundary.

When $\rho=4$, so $\Amp^1(X)$ is hyperbolic $3$-space, the usual boundary $\partial \Amp^1(X)$ is a $2$-sphere.
The oriented blown-up boundary $\partial \Amp^1_c(X)$ is a \Sierpinski carpet, obtained from the $2$-sphere by inserting an open disk at the countable many rational points on the sphere.
The projectively blown-up boundary $\partial^\circ \Amp^1_c(X)$ identifies opposite points on the boundary circles to yield $\bR\bP^1$'s, so $\partial^\circ \Amp^1_c(X)$ can be viewed as obtain from the \Sierpinski carpet by gluing in countably many \Mobius bands.
Now $\partial \Amp^1_c(X)$ and $\partial^\circ \Amp^1_c(X)$ are no longer manifolds.



\subsection{Linear algebra and hyperbolic geometry}
	\label{ssec:linear_algebra_and_hyperbolic_geometry}

\subsubsection{Setup}
	\label{sssec:setup_linear_algebra_and_hyperbolic_geometry}

We now interpret some of the geometric constructions in \autoref{ssec:groups_and_spaces} in terms of linear-algebraic properties of the matrices corresponding to the group elements $\gamma_i$.
First, because the compact part of the fundamental domain $F_{[\omega_0],c}$ has bounded diameter, all projected basepoints are a uniform bounded distance from their projections in \autoref{sssec:fixed_finite_set_of_generators} above.

Recall that $\NS(X)=:N$ is equipped with an intersection pairing of signature $(1,\rho-1)$ and we will denote by $\SO(N_{\bR})$ the \emph{indefinite} orthogonal group preserving it, which is isomorphic to $\SO_{1,\rho-1}(\bR)$.
The fixed basepoint $[\omega_0]\in \Amp^1(X)$ determines a positive-definite metric on $N_{\bR}$, which we use also use and denote by $\norm{\bullet}$.
It is explicitly given as follows: every $[\alpha]\in N_{\bR}$ can be written uniquely as $[\alpha]=a[\omega_0]+[\beta]$ where $a\in\mathbb{R}$ and $[\beta]\cdot[\omega_0]=0$, and we let $\norm{[\alpha]}^2=a^2-[\beta]^2$.
Note that an element $g\in \SO(N_{\bR})$ does not preserve this norm, and its operator norm for this metric is related to the hyperbolic distance by
\begin{align}
	\label{eqn:operator_norm_hypb_dist}
	\norm{g}_{op} = e^{\dist([\omega_0],g[\omega_0])}
\end{align}
Using the notation from \autoref{sssec:fixed_finite_set_of_generators} we have that
\[
	\frac 1C \leq
	\frac{\norm{\gamma_i\gamma_{i+1}\cdots \gamma_{i+j}}_{op}}
	{e^{\lambda_i + \phantom{\cdots} \cdots \phantom{\cdots} + \lambda_{i+j}}}
	\leq C
\]
for some fixed constant $C$ and for all $i\geq 1, j\geq 0$.

\subsubsection{Mass in cohomology}
	\label{sssec:mass_in_cohomology}
For $[\omega]\in \Amp(X)$ denote by $M([\omega]):=[\omega_0].[\omega]$, the \emph{mass with respect to the reference class $[\omega_0]$}.
This will be a useful normalization factor to bring all cohomology classes to a bounded set.
Additionally, there exists a constant $C>0$ (depending on $[\omega_0]$ or equivalently the fixed inner product) such that
\[
	\frac{1}{C}\norm{[\omega]} \leq M([\omega]) \leq C\norm{[\omega]} \text{ for all }[\omega] \text{ s.t. }[\omega]^2\geq 0 \text{ and }[\omega].[\omega_0]\geq 0
\]
Indeed, the comparison follows immediately on the compact set of classes with $M([\omega])=1$ and inside the ample cone, and the inequalities are invariant by rescaling.

\subsubsection{Transversality and growth}
	\label{sssec:transversality_and_growth}
Any $g\in \SO(N_{\bR})$ has a KAK (or Cartan, or polar), decomposition as $g=k_1ak_2$ where $k_i$ preserves the fixed norm $\norm{\bullet}$ on $N_{\bR}$ (equivalently, the basepoint $[\omega_0]$) and the element $a$ expands a null-vector by $\norm{g}_{op}$ and contracts another null-vector by $\norm{g}_{op}^{-1}$.
The span of the two null-vectors contains $[\omega_0]$.

\begin{definition}[Transversality of a vector relative to a transformation]
	\label{def:transversality_of_a_vector_relative_to_a_transformation}
	We will say that $v\in N_\bR$ is \emph{$C$-transverse} for $g\in \SO(N_{\bR})$ if we have that
	\[
		\norm{gv} \geq \frac{1}{C}\norm{g}_{op}\norm{v}
	\]
	If the constant $C$ only depends on some parameters fixed in terms of the hyperbolic manifold $\cM$, we will say that $v$ is \emph{transverse} for $g$ without specifying the parameter $C$.
\end{definition}
Note that being $C$-transverse is invariant under scaling $v$.
Given a KAK decomposition $g=k_1a k_2$, the condition that the unit vector $v$ is transverse for $g$ means that $k_2v$ is not too close to the unit eigenvector contracted by $a$, with distance depending on the amount of transversality.
Relative to the reference basepoint $[\omega_0]$, a vector $v$ is $C$-transverse for $g$ if $v$ and $g^{-1}[\omega_0]$ are not too close, projectively.
If $v$ satisfies $v.v=1$ (so it determines a point in hyperbolic space), then $C$-transversality is equivalent to $v$ not being contained in a neighborhood of the endpoint at infinity of the geodesic ray starting at $[\omega_0]$ and going to $g^{-1}[\omega_0]$.

Because the generators $\gamma_i$ in \autoref{sssec:adapted_generators} are chosen in terms of the visits of a hyperbolic geodesic to fundamental domains, it follows that we have the following properties:
\begin{itemize}
	\item The vector $[\omega_0]$ is $C$-transverse for $\gamma_i\cdots \gamma_j$ for any $j\geq i\geq 1$.
	\item The vectors $\gamma_{j+1}\cdots \gamma_{j+k}[\omega_0]$ are $C$-transverse for $\gamma_i\cdots \gamma_j$ for any $k\geq 1$ and $j\geq i\geq 1$
\end{itemize}
for some constant $C$ that only depends on $\cM$ and $[\omega_0]$.
These estimates are simply restatements of the geometric properties of the geodesics, translated via the identification of operator norms and hyperbolic distances in \autoref{eqn:operator_norm_hypb_dist}.

From now on, to avoid making constants explicit, we will use the notation $A\leqapprox B$ to mean that there exists some constant $c>0$, depending on $\cM$ and $[\omega_0]$, such that $A\leq c\cdot B$.

\begin{proposition}[Contraction from transversality]
	\label{prop:contraction_from_transversality}
	Suppose that $x,y\in N_\bR$ are two vectors that are transverse for an element $g\in \SO\left(N_\bR\right)$.
	Then we have the estimate
	\[
		\norm{
		\frac{gx}{M(gx)}
		-
		\frac{gy}{M(gy)}
		} \leqapprox \frac{1}{\norm{g}_{op}}
	\]
\end{proposition}
\begin{proof}
	Let us write the KAK decomposition $g=k_1ak_2$ and let $v_-,v_+$ be the vectors contracted by $a$, each of mass $1$, so that $2[\omega_0]=v_- + v_+$.
	The transformation $k_1$ does not affect the estimates so we can assume that $k_1=\id$.
	We then have that
	\[
		k_2 x = c_-(x)v_- + c_+(x)v_+ + v_0(x)
	\]
	with $v_0(x)$ orthogonal to both $v_-$ and $v_+$.
	By the transversality of $x$ for $g$ we know that $c_+\geq c>0$ where $c$ is some constant determined by the amount of transversality.
	By a direct calculation it follows that $\frac{ak_2x}{M(ak_2x)}=v_+ + \frac{O(1)}{\norm{g}_{op}}$.
	A similar argument applies to $y$ and the result follows.
\end{proof}




\section{Elliptic fibrations and currents}
	\label{sec:elliptic_fibrations_and_currents}

\subsubsection*{Outline}
The main result of this section is \autoref{thm:current_valued_pairing}, which establishes the existence of canonical currents associated to parabolic automorphisms of an elliptic fibration. In this section we can relax considerably our standing assumptions in \autoref{sssec:standing_assumptions_main} to the following more general setup.

\subsubsection{Standing assumptions for this section}
	\label{sssec:standing_assumptions_elliptic_currents}
Throughout this section, $X$ denotes a compact \Kahler surface with a genus one fibration
\[
	X \xrightarrow{\pi} B
\]
where $B$ is a compact complex curve, such that the following holds:
\begin{center}
The singular fibers of $\pi$ are reduced and irreducible, i.e. the singularities have Kodaira type $I_1$ or $II$.
\end{center}

Denote by $\Aut_\pi(X)$ the automorphisms of $X$ that preserve the fibers of $\pi$.
In most cases, including K3 surfaces, this group is of finite index in the group of all automorphisms of $X$ taking fibers to (possibly other) fibers.
Indeed, in the K3 case (where $B=\bP^1$), the induced group of automorphisms of $\bP^1$ must preserve $\pi_*\dVol$ and hence be of finite order.

Observe also that when $X$ is a K3 surface, the standing assumption of this section is implied by our earlier standing assumptions in \autoref{sssec:standing_assumptions_main}, since by \cite[Theorem 11.1.9]{Huy} any singular fiber which is not of type $I_1$ or $II$ contains a $(-2)$-curve.


\subsection{Calculations in cohomology}
	\label{ssec:calculations_in_cohomology_elliptic_fibrations}

\subsubsection{Setup}
	\label{sssec:setup_calculations_in_cohomology_elliptic_fibrations}
For a point $b\in B$ denote by $X_b:=\pi^{-1}(b)$ the corresponding fiber and let $[E]\in \NS(X)$ denote the class of a general fiber of the fibration.
Importantly, in this section we consider the action of $\Aut(X)$ on $\NS(X)$ by \emph{pushforward} (induced by the pushforward operation on divisors and denoted $\gamma_*$ for $\gamma\in \Aut(X)$); this is relevant for keeping track of some signs.
The following filtration of the \Neron--Severi group is preserved by $\Aut_\pi(X)$:
\begin{align}
	\label{eqn:cohomology_filtration_elliptic_fibration}
	0\subsetneq [E] \subsetneq [E]^{\perp} \subsetneq \NS(X)
\end{align}
where $[E]^{\perp}$ denotes the orthogonal complement for the intersection pairing.
The induced intersection pairing on $[E]^\perp/[E]$ is strictly negative-definite and we let $\Aut_\pi^\circ(X)$ denote the subgroup of $\Aut_\pi(X)$ whose induced action on $[E]^\perp/[E]$ is trivial.
This is a finite-index subgroup of $\Aut_\pi(X)$.

\subsubsection{Linearizing automorphisms}
	\label{sssec:linearizing_automorphisms}
Because $\Aut_\pi^\circ(X)$ preserves the filtration in \autoref{eqn:cohomology_filtration_elliptic_fibration} and acts as the identity on subquotients, there is a group homomorphism
\begin{align}
	\label{eqn:aut_pi_linearization_cohomology}
	\begin{split}
	\Aut_\pi^\circ(X) & \xrightarrow{\xi} \Hom\left(\rightquot{[E]^\perp}{[E]},[E]\right)\\
	\gamma & \mapsto \xi(\gamma)(v)= \gamma_*(v)-v
	\end{split}
\end{align}
Note that for $v\in [E]^{\perp}$ and $\gamma\in \Aut_\pi^\circ(X)$, we have by definition that $\gamma_*(v)$ and $v$ agree modulo $[E]$, so their difference is a multiple of $[E]$.
Via duality and identifications of duals using the intersection pairing, we have equivalently
\[
	\xi(\gamma)^t \in \Hom\left(\NS(X)/[E]^{\perp},\rightquot{[E]^\perp}{[E]}	\right)
\]
Indeed, the intersection pairing establishes a nondegenerate duality pairing between $[E]$ and $\NS(X)/[E]^{\perp}$ and induces a perfect pairing on $[E]^\perp/[E]$.
Alternatively, one can define the transformation by
\[
	\xi(\gamma)^t([v]) := \gamma_* [v] - [v] \in [E]^{\perp}/[E] \quad \text{ for }[v]\in \NS(X)/[E]^{\perp}
\]
since for any $v\in \NS(X)$ and $\gamma\in \Aut_\pi^{\circ}(X)$, we have that $\gamma_* v$ and $v$ agree modulo $[E]^{\perp}$, so their difference $\gamma_* v-v$ is in $[E]^{\perp}$, and moreover changing $v$ to $v+v'$ with $v\in [E^{\perp}]$ changes the difference to $\gamma_*v'-v'$, which is an element of $[E]$.

Since the class of the fiber $[E]$ is canonically defined, both spaces on the right-hand side above can be canonically identified with $[E]^\perp/[E]$.
By abuse of notation, we occasionally identify here and below a free $\bZ$-module of rank $1$ with its generator (and in our situation we have a distinguished choice of one of the two generators).

\begin{proposition}[Formulas for fibration automorphism]
	\label{prop:formulas_for_fibration_automorphism}
	\leavevmode
	\begin{enumerate}
	\item
	The map from \autoref{eqn:aut_pi_linearization_cohomology} injects $\Aut_\pi^\circ(X)$ into $[E]^\perp/[E]$.
	\item
	Given $\xi\in [E]^\perp/[E]$, define
	\[
	n_{\xi}(v) := (v.[E])\xi - (v.\xi)[E] \quad \forall v\in \NS(X)
	\]
	This is a nilpotent transformation, well-defined and independent of the lift of $\xi$ from $[E]^\perp/[E]$ to $[E]^\perp$, and strictly decreasing the filtration from \autoref{eqn:cohomology_filtration_elliptic_fibration}.
	Furthermore, the image of $\NS(X)/[E]^\perp$ under $n_{\xi}$ is the span of $\xi$ inside $[E]^\perp/[E]$, as can be seen directly from the formula for $n_\xi$, and we also have $n_\xi(\xi)=-(\xi.\xi)[E]$.
	In addition $n_\xi$ is an element of the indefinite orthogonal Lie algebra, i.e. $n_\xi(v).w + v.n_\xi(w)=0$.

	\item
	The action of $\gamma\in \Aut_\pi^{\circ}(X)$ on $\NS(X)$ is given by
	\[
		1 + n_{\xi} + \tfrac12 {n_{\xi}^2} \quad \text{ with } n_\xi:=n_{\xi(\gamma)}.
	\]
	\end{enumerate}
\end{proposition}
\noindent In particular, we have for a vector $v\in \NS(X)$ and $\xi=\xi(\gamma)$ associated to $\gamma\in \Aut_\pi^{\circ}(X)$ that the action is given by
\begin{align}
	\label{eqn:cohomology_action_explicit}
	\gamma(v) = v + \Big[(v.[E])\xi - (v.\xi)[E]\Big] + (v.[E])\frac{-\xi.\xi}{2}[E],
\end{align}
a well-known formula, see e.g. \cite[(11)]{Cantat_Dolgachev}.
Note that $\xi.\xi<0$ since the intersection pairing on $[E]^\perp/[E]$ is strictly negative-definite.
\begin{proof}[Proof of \autoref{prop:formulas_for_fibration_automorphism}]
	The map from \autoref{eqn:aut_pi_linearization_cohomology} is injective, because more generally the kernel of the map $\Aut(X)\to \Orthog(\NS(X))$, when $h^{1,0}=0$, is a group of automorphisms of $X$ preserving an ample line bundle (so they are naturally a subgroup of some $\PGL_n$).
	The subgroup of elements in $\Orthog(\NS(X))$ which preserve $[E]$ and act as the identity on $[E]^{\perp}/[E]$ is then naturally identified with $[E]^\perp/[E]$, as can be seen for instance by choosing an explicit basis.

	For (ii) the stated properties of $n_{\xi}$ are checked directly from the formula defining it.

	To establish (iii) observe that the properties of $n_\xi$ match those that define the logarithm of the action of $\gamma$ on $\NS(X)$.
	Indeed $\log \gamma$ is a nilpotent transformation in the orthogonal Lie algebra that can be defined as $(\gamma-1) + \tfrac 12(\gamma-1)^2$.
	The term $(\gamma-1)^2$ maps $\NS(X)/[E]^{\perp}$ to $[E]$, so does not contribute to the expression $(\gamma-1) v \mod [E]$ for a vector $v$.
	Therefore $\log \gamma$ and $n_{\xi}$ are nilpotent transformations in the orthogonal Lie algebra which agree on $\NS(X)/[E]$, and hence agree.
	Therefore $\gamma=\exp(n_{\xi})$ as claimed.
\end{proof}

\subsubsection{A remark on signs}
	\label{sssec:a_remark_on_signs}
The cup-product on $[E]^\perp/[E]$ is naturally negative-definite.
Because of this, in order to have convenient algebraic expressions, and to work with positive currents, we will have to introduce some minus signs when defining certain currents.
One such example is in \autoref{thm:current_valued_pairing}, where the pairing $\eta(\gamma,[\alpha])$ gives a current in the cohomology class $-[\xi(\gamma)].[\alpha]$ on $B$.



\subsection{Analytic properties of parabolic automorphisms}
	\label{ssec:analytic_properties_of_parabolic_automorphisms}

\subsubsection{Setup}
	\label{sssec:setup_analytic_properties_of_parabolic_automorphisms}
For the fibration $X\xrightarrow{\pi}B$ let $B^\circ\subseteq B$ denote the image of the smooth fibers and $X^\circ:=\pi^{-1}\left(B^{\circ}\right)$.
The fiber over $b\in B$ is denoted $X_b$ as before.
For this section, fix a \Kahler metric $\omega$ on $X$.

Recall that an usc function $\phi:X\to\mathbb{R}\cup\{-\infty\}$ is called quasi-psh if locally it equals the sum of a psh function plus a smooth function.
Quasi-psh functions thus satisfy $dd^c\phi\geq \gamma$ weakly on $X$, for some smooth real $(1,1)$-form $\gamma$. More generally, a quasi-positive $(1,1)$ current is defined as $T=\alpha+dd^c\phi\geq \gamma$, where $\phi$ is quasi-psh and $\alpha,\gamma$ are smooth $(1,1)$-forms.
We can define the restriction of $T$ to any fiber $X_b,b\in B$ for which $\phi|_{X_b}\not\equiv -\infty$ simply by setting $T|_{X_b}=\alpha|_{X_b}+dd^c(\phi|_{X_b})$, where in the case when $X_b$ is singular we refer e.g. to \cite[\S 1]{Dema_MA} for details about smooth forms and currents on singular complex analytic spaces.

For a compact complex manifold $M^n$ we will use the notation $\cA_{k,k}(M)$ to denote the space of currents dual to smooth $(k,k)$-forms on $M$ (i.e. currents of bidegree $(n-k,n-k)$). Thus, the space of $(1,1)$ currents on the base $B$ is equal to $\cA_{0,0}(B)$.

If $T$ is a $(1,1)$ current on $M$ and $\omega$ is a K\"ahler metric, we will use the standard notation $\mathrm{tr}_\omega T$ for the distribution defined by
$$\mathrm{tr}_\omega T=\frac{n \omega^{n-1}\wedge T}{\omega^n}.$$
In particular, for a smooth function $u$ we define its $\omega$-Laplacian by $\Delta_\omega u=\mathrm{tr}_\omega(dd^c u)$.

\begin{definition}[$\pi$-trivial forms and currents]
	\label{def:pi_trivial_form}
	A smooth $(1,1)$-form $\alpha$ on $X$ is called \emph{$\pi$-trivial} if for all $b\in B$ we have $\alpha\vert_{X_b}\equiv 0$.
	Similarly, a quasi-positive $(1,1)$-current $T$ on $X$ is $\pi$-trivial if for all $b\in B$ the restriction $T|_{X_b}$ is well-defined and is zero.
    We also define $\pi$-triviality on $X^\circ$ if the above condition holds only for all $b\in B^\circ$.
\end{definition}

\begin{remark}[Pulling back from $B$]
	\label{rmk:pulling_back_from_b}
	\leavevmode
	\begin{enumerate}
		\item
		Note that $\pi$-trivial forms need not be pulled back from $B$.
		However, if a $\pi$-trivial form $\alpha$ is of the form $\alpha=dd^c\phi$ for a function $\phi$ then there exists some $\phi'$ on $B$ with $\pi^{*}\phi'=\phi$ since the condition $dd^c\left(\phi\vert_{X_b}\right)\equiv 0$ implies that $\phi$ is constant on fibers.
		\item However, if $\eta$ is a $\pi$-trivial closed $(1,1)$-current which is in the class of a fiber, then $\eta=\pi^*\eta_B$ for a closed $(1,1)$-current $\eta_B$ on $B$.
		Indeed, pick some smooth $\omega_B$ on $B$ such that $[\eta]=\pi^*[\omega_B]$ in cohomology and write $\eta = \pi^*\omega_B + dd^c \phi$ for some $\phi$.
		For every fiber we have $\eta\vert_{X_b}= 0$ (since $\eta$ is $\pi$-trivial) and $\pi^*\omega_B\vert_{X_b}\equiv 0$ so it follows that $dd^c\left(\phi\vert_{X_b}\right)\equiv 0$, which in turn implies $\phi$ is constant on fibers.
		\item Additionally, if $\eta$ is a closed \emph{positive} $(1,1)$-current in the class of the fiber, then it is necessarily $\pi$-trivial and pulled back from $B$.
		Indeed, writing $\eta=\pi^*\omega_B + dd^c\phi$ it follows that $dd^c\phi\geq 0$ on each fiber and hence $\phi$ is constant on fibers.
	\end{enumerate}
\end{remark}

\begin{theorem}[Preferred currents]
	\label{thm:preferred_currents}
	Assume all singular fibers of $\pi$ are reduced and irreducible (i.e. of Kodaira types $I_1$ and $II$). Let $\alpha$ be a closed smooth real $(1,1)$-form on $X$, which satisfies
	\begin{equation}
	\label{eqn:alpha_zero_integral}
	\int_{X_b}\alpha=0,\quad \text{for all }b\in B
	\end{equation}
	or equivalently $[\alpha].[E]=0$.

	Then there exists $\phi\in L^\infty(X)\cap C^{\infty}\left(X^\circ\right)$, unique up to a function pulled back from $B$, such that
	\[
		\alpha + dd^c \phi \text{ is }\pi\text{-trivial on }X,
	\]
and
	\[
		\norm{\phi}_{L^\infty(X)} \leq C(X,\omega)\norm{\alpha}_{L^\infty(X,\omega)}.
	\]
\end{theorem}

\begin{remark}
Using the arguments in the proof of \autoref{prop:continuous_potentials} and more work, it is possible to show that the function $\phi$ in \autoref{thm:preferred_currents} is smooth on $X$ minus the (finitely many) singular points of the singular fibers.
Since this information is not needed here, we will not belabor this point.
\end{remark}

\begin{proof}[Proof of \autoref{thm:preferred_currents}] We first construct the desired function $\phi$ on $X^\circ$. For this, given any $b\in B^\circ$, the condition that $\alpha+dd^c\phi$ is $\pi$-trivial restricts to $X_b$ to
$$dd^c(\phi|_{X_b})=-\alpha|_{X_b},$$
and this equation on $X_b$ is certainly solvable since \autoref{eqn:alpha_zero_integral} gives that $\int_{X_b}\alpha=0$ and so $\alpha|_{X_b}$ is null-cohomologous and hence $dd^c$-exact. Since the fiberwise solution is unique up to a constant, we see that
for every $b\in B^{\circ}$ there is a unique function $\phi_b\in C^\infty(X_b)$ solving
\begin{equation}\label{fiberwise}
dd^c\phi_b=-\alpha|_{X_b},\quad \int_{X_b}\phi_b\omega_b=0,
\end{equation}
on $X_b$.
It follows from standard Hodge theory with parameters (i.e. elliptic estimates on the fixed smooth manifold underlying $X_b$ with respect to the smoothly varying family of Riemannian metrics defined by $\omega_b$) that the functions $\phi_b$ vary smoothly in $b\in B^\circ$.
Thus, taken together, they define a smooth function $\phi$ on $X^\circ$ (i.e. $\phi|_{X_b}=\phi_b$) which satisfies
$$(\alpha+dd^c\phi)|_{X_b}=0,\quad \text{for all }b\in B^{\circ},$$
and so $\alpha+dd^c\phi$ is $\pi$-trivial on $X^\circ$.

Next let us discuss what happens at the singular fibers, so let $0\in B\setminus B^\circ$ be a critical value on the base, with singular fiber $X_0\xrightarrow{\iota} X$ and singular point $x_0\in X_0$.	Let $\wtilde{X_0}\xrightarrow{\nu}X_0$ be its normalization, which maps a finite set $S\subset\wtilde{X_0}\isom \bP^1$ to $x_0 \in X_0$ and is biholomorphic on the complement, and denote by $\mu=\iota\circ\nu$. The set $S$ consists of two points when $X_0$ is of type $I_1$ and of one point when it is of type $II$.

Assumption \autoref{eqn:alpha_zero_integral} implies that $\int_{\wtilde{X_0}}\mu^*\alpha=0,$ and so $\mu^*\alpha$ is null-cohomologous on $\wtilde{X_0}$ and hence $dd^c$-exact. Since $\mu^*\omega$ is a semipositive $(1,1)$-form on $\wtilde{X_0}$ which is K\"ahler away from a finite set (in particular $\int_{\wtilde{X_0}}\mu^*\omega>0$), we can find a  unique function $\ti{\phi}_0\in C^\infty(\wtilde{X_0})$ satisfying
	\begin{equation}\label{hodge}
		\int_{\wtilde{X_0}}\ti{\phi}_0\mu^*\omega = 0 \qquad \text{and} \qquad dd^c\ti{\phi}_0=-\mu^*\alpha,
	\end{equation}
Then $\phi_0:=\nu_*\ti{\phi}_0$ is a bounded function on $X_0$ which satisfies $\alpha|_{X_0}+dd^c \phi_0=0$, and so extending $\phi$ to $X_0$ by $\phi|_{X_0}=\phi_0$, and repeating this procedure at all singular fibers, we obtain a function $\phi$ on $X$, smooth on $X^\circ$, such that $\alpha+dd^c\phi$ is $\pi$-trivial on $X$.

Every other function $\phi'$ with this property will satisfy $dd^c(\phi'-\phi)|_{X_b}=0$ and so will differ from $\phi$ by the pullback of a function from $B$.

It remains to show that $\phi$ is bounded on $X$, and since by construction $\phi$ is bounded on the singular fibers we only need to prove boundedness on $X^\circ$, or in other words that there is a constant $C$ such that for all $b\in B^{\circ}$ we have
\begin{equation}
	\label{eqn:phi_is_bounded}
	\sup_{X_b}|\phi_b|\leq C.
\end{equation}
To see this, let $G_b$ be the Green's function of the Laplacian of $(X_b,\omega_b)$, which is symmetric in its two variables and is normalized by $\int_{X_b}G_b(\cdot,x)\omega_b=0$ for all $x\in X_b$, and which is defined by the property that for every smooth function $u$ on $X_b$ with $\int_{X_b}u\omega_b=0$ we have
$$u(x)=-\int_{X_b}(\Delta_{\omega_b}u) G_b(\cdot,x)\omega_b \quad \text{for all $x\in X_b$.}$$

The key claim then is that we can find $C_0>0$ such that for all $b\in B^{\circ}$ we have
\begin{equation}
	\label{eqn:lower_bound_Green}
	G_b(\cdot,\cdot)\geq -C_0,
\end{equation}
i.e. we have a uniform lower bound for the Green's function of the Laplacian of the degenerating Riemann surfaces $(X_b,\omega_b)$.
To establish this, we employ a classical argument of Cheng--Li \cite{CL}, which is clearly explained in \cite[Chapter 3, Appendix A, pp.137-140]{Si} (see also the recent exposition in \cite[\S 3]{DGG}), which shows that to prove \autoref{eqn:lower_bound_Green} it suffices to obtain a uniform upper bound for the Poincar\'e constant of $(X_b,\omega_b)$, which is uniform for all $b\in B^{\circ}$ (using here that $X_b$ has real dimension $2$ and that  the volume of $(X_b,\omega_b)$ is constant as we vary $b$).

More precisely, the argument of Cheng--Li gives a uniform bound lower bound for the Green's function of $(X_b,\omega_b)$ in terms of a positive lower bound the volume $\int_{X_b}\omega_b$ (which holds trivially in the present case) and of an upper bound for the Poincar\'e constant $C_{P,b}$ of $(X_b,\omega_b)$, which is characterized by
$$\int_{X_b} u^2 \omega_b \leq C_{P,b} \int_{X_b}|d u|^2_{\omega_b}\omega_b,$$
for all functions $u\in C^\infty(X_b)$ with $\int_{X_b}u\omega_b=0$. In other words, $C_{P,b}$ is the reciprocal of the lowest positive eigenvalue of the Laplacian of $\omega_b$.

Thus, to complete the proof of the claim \autoref{eqn:lower_bound_Green} it remains to obtain a uniform upper bound for $C_{P,b}$ independent of $b\in B^\circ$.
This bound is proved by  Yoshikawa \cite{Yo} (see also the shorter and clear exposition in \cite[Proposition 3.2]{RZ}), using the assumption that all singular fibers are reduced and irreducible (and in general the Poincar\'e constant blows up near singular fibers which do not satisfy this). Combining all these results we obtain \autoref{eqn:lower_bound_Green}.


Next, observe that by definition we have
$$-\norm{\alpha}_{L^\infty(X,\omega)}\omega\leq \alpha\leq \norm{\alpha}_{L^\infty(X,\omega)}\omega,$$
holds on all of $X$. Restricting to $X_b$, tracing with respect to $\omega_b$, and using \autoref{fiberwise} we get that
$$-\norm{\alpha}_{L^\infty(X,\omega)}\leq \mathrm{tr}_{\omega_b}(\alpha|_{X_b})=-\Delta_{\omega_b}\phi_b\leq \norm{\alpha}_{L^\infty(X,\omega)},$$
for all $b\in B^{\circ}$.
We can then use the Green's formula for $\phi_b$ on $X_b$, which has fiber integral zero, to get that for all $x\in X_b$ and all $b\in B^{\circ}$ we have
$$\phi_b(x)=-\int_{X_b}\Delta_{\omega_b}\phi_b(G_b(\cdot,x)+C_0)\omega|_{X_b},$$
and using the $L^\infty$ bound for $\Delta_{\omega_b}\phi_b$ and the facts that $G_b(\cdot,x)+C_0\geq 0$ and $G_b(\cdot,x)$ integrates to zero on $X_b$, we immediately obtain that \autoref{eqn:phi_is_bounded} holds.
\end{proof}

\begin{proposition}[Continuous potentials]
	\label{prop:continuous_potentials}
	Suppose $\alpha$ is a closed smooth $(1,1)$-form on $X$ which satisfies \autoref{eqn:alpha_zero_integral} and $\phi\in L^\infty(X)\cap C^\infty(X^\circ)$ is provided by \autoref{thm:preferred_currents}, i.e. so that $\alpha + dd^c\phi$ is $\pi$-trivial on $X$.
	For any $\gamma \in \Aut_\pi^{\circ}(X)$ there exists a closed $(1,1)$ current $\eta=\eta(\gamma,\alpha)\in \cA_{0,0}(B)$ such that
	\[
		\gamma_*\left(\alpha+dd^c\phi\right) = \left(\alpha+dd^c\phi\right) + \pi^* \eta.
	\]
Furthermore $\eta$ has continuous potentials on $B$.
\end{proposition}
\noindent Later on in \autoref{thm:current_valued_pairing} we will establish that $\eta$ depends only on the classes $\xi(\gamma)$ and $[\alpha]$ in $[E]^\perp/[E]$.

\begin{proof}
Thanks to \autoref{eqn:cohomology_action_explicit}, we have
$$[\gamma_*\alpha]-[\alpha]=-([\alpha].\xi)[E].$$
Fix a smooth form $\beta$ on $B$ with the correct total integral so that $\pi^*\beta$ is a smooth representative of $-([\alpha].\xi)[E]$. We can then write $\gamma_*\alpha-\alpha=\pi^*\beta+dd^cf$ for some smooth function $f$ on $X$ (observe here that $\gamma_*\alpha-\alpha$ is smooth on $X$ since $\gamma_*\alpha=(\gamma^{-1})^*\alpha$). Let $\phi\in L^\infty(X)\cap C^\infty(X^\circ)$ be as in \autoref{thm:preferred_currents}, then
the current $\gamma_*(\alpha+dd^c\phi)- (\alpha+dd^c\phi)$ is $\pi$-trivial on $X$ and in the class of $-([\alpha].\xi)[E]$,
so by \autoref{rmk:pulling_back_from_b} it follows that this current is of the form $\pi^*\eta$ for some $\eta\in \cA_{0,0}(B)$. On $X$ we thus have
\begin{equation}\label{grz}
\pi^*\eta=(\gamma_*\alpha-\alpha) + dd^c(\gamma_*\phi-\phi) = \pi^*\beta+dd^c(f+\gamma_*\phi-\phi),
\end{equation}
so that $f+\gamma_*\phi-\phi$ is constant on all fibers and hence it is the pullback of a function $u$ on $B$, with $\eta=\beta+dd^c u$. Since $\phi\in L^\infty(X)\cap C^\infty(X^\circ)$ and $f\in C^\infty(X)$, we see that $u\in L^\infty(B)\cap C^\infty(B^\circ)$. We are thus claiming that $u$ is a continuous function on $B$.

Let thus $0\in B\backslash B^\circ$ with singular fiber $X_0=\pi^{-1}(0)$. 	Using the same notation as in the construction of $\phi$ in \autoref{thm:preferred_currents}, given any sufficiently small open set $U\subseteq \left(X_0\setminus x_0\right)$ the elliptic fibration map $\pi$ can be holomorphically trivialized near $U$ as $U\times D\to D$ where $D\subset B$ is an open disc containing $0$, and the map $\mu$ is an isomorphism over $U$.

The key claim is then that on some nonempty open set $U$ as above, in the local trivialization we have that $(f+\gamma_*\phi-\phi)|_{U\times\{b\}}$ converges uniformly on $U$ to $(f+\gamma_*\phi-\phi)|_{U\times\{0\}}$ as $b\to 0$. Since these functions are all fiberwise constants (equal to $\pi^*u$), the desired continuity of $u$ follows from this, and since $f$ is smooth on all of $X$, it suffices to prove the claim for $\gamma_*\phi-\phi$.

To do this,
fix a K\"ahler metric $\ti{\omega}$ on $\wtilde{X_0}$, and take any sufficiently small $U\subseteq \left(X_0\setminus x_0\right)$ with trivialization as above. For $b\in D\backslash\{0\},$ we get smooth functions $\phi_b=\phi|_{U\times\{b\}}$ on $U$ which satisfy $dd^c\phi_b=-\alpha_b$ where $\alpha_b=\alpha|_{U\times\{b\}}$, and which vary smoothly in $b\in D\backslash\{0\}$. Therefore they also satisfy
$$\Delta_{\mu_*\ti{\omega}}\phi_b=-\mathrm{tr}_{\mu_*\ti{\omega}}\alpha_b,\quad \text{on }U,$$
where the right hand side here is a family of smooth functions that varies smoothly in $b\in D$,
and the functions $\phi_b$ have a uniform $L^\infty$ bound by \autoref{thm:preferred_currents}. We can then apply standard Schauder estimates for the Laplacian of $\mu_*\ti{\omega}$, which show that given any $U'\Subset U$ and $k\in\mathbb{N}$ and $0<a<1$ there is a constant $C$ (independent of $b$) such that
$$\|\phi_b\|_{C^{k+2,a}(U')}\leq C(\|\Delta_{\mu_*\ti{\omega}}\phi_b\|_{C^{k,a}(U)}+\|\phi_b\|_{L^\infty(U)})= C(\|\mathrm{tr}_{\mu_*\ti{\omega}}\alpha_b\|_{C^{k,a}(U)}+\|\phi_b\|_{L^\infty(U)})\leq C',$$
where again $C'$ is independent of $b$.
Thus, shrinking $U$ slightly, Ascoli-Arzel\`a then implies that for some sequence $b_i\to 0$ the corresponding functions $\phi_{b_i}$ converge in the smooth topology on $U$ to a smooth function $\hat{\phi}_{U,0}$ on $U$ which satisfies
$$\Delta_{\mu_*\ti{\omega}}\hat{\phi}_{U,0}=-\mathrm{tr}_{\mu_*\ti{\omega}}\alpha|_U,\quad \text{on }U.$$
We can repeat this procedure on a different open set $U'$, taking a subsequence of $b_i$, and obtain smooth convergence on $U\cup U'$. Continuing this way, taking a countable collection of such $U_j$'s which cover $X_0\setminus x_0$ (with discs $0\in D_j\subset B$ with shrinking radii), and passing to a diagonal sequence, we obtain some new sequence $b_i\to 0$ and a bounded smooth function $\hat{\phi}_0$ on $X_0\setminus x_0$ such that $\phi_{b_i}$ converge to $\hat{\phi}_0$ locally smoothly on $X_0\setminus x_0$ (in the obvious sense) and so we have
$$\Delta_{\mu_*\ti{\omega}}\hat{\phi}_{0}=-\mathrm{tr}_{\mu_*\ti{\omega}}\alpha|_{X_0\setminus x_0},\quad \text{on }X_0\setminus x_0.$$
We also have that as $b\to 0$ the part of $X_b$ which is not covered by the union of the $U_j\times\{b\}$'s with $b\in D_j$ has measure that is going to zero.
Thanks to this, and to the uniform $L^\infty$ bound for $\phi_b$ from \autoref{thm:preferred_currents}, the relation $\int_{X_b}\phi_b\omega_b=0$ passes to the limit to
$$\int_{X_0\setminus x_0}\hat{\phi}_0 \omega=0.$$
Pulling back via $\mu$, which is an isomorphism over $X_0\setminus x_0$, we see that
$$\Delta_{\ti{\omega}}\mu^*\hat{\phi}_{0}=-\mathrm{tr}_{\ti{\omega}}\mu^*\alpha,\quad \text{on }\wtilde{X_0}\setminus S,\quad \int_{\wtilde{X_0}\setminus S}\mu^*\hat{\phi}_0 \mu^*\omega=0.$$
On the other hand the function $\ti{\phi}_0$ defined by \eqref{hodge} satisfies
$$\Delta_{\ti{\omega}}\ti{\phi}_{0}=-\mathrm{tr}_{\ti{\omega}}\mu^*\alpha,\quad \text{on }\wtilde{X_0},\quad \int_{\wtilde{X_0}}\ti{\phi}_0 \mu^*\omega=0,$$
so the difference $\mu^*\hat{\phi}_0-\ti{\phi}_0$ is a smooth function on $\wtilde{X_0}\setminus S$ which is bounded, belongs to the kernel of $\Delta_{\ti{\omega}}$ and it also integrates trivially against $\mu^*\omega$. By a classical Riemann-type extension theorem (``bounded harmonic functions extend across point singularities'', see e.g. \cite[Thm 2.3]{ABR}) $\mu^*\hat{\phi}_0-\ti{\phi}_0$ extends smoothly across $S$ and the extension is then forced to be zero by the integral condition. Thus $\hat{\phi}_0=\phi_0$, and in particular if we picked a different subsequence $b_i\to 0$, the limit is forced to be the same by the uniqueness of $\phi_0$, and so the functions $\phi_b$ converge locally smoothly in our trivialization to $\phi_0$ without having to pass to any subsequence.

The same argument applies to $\gamma_*\phi_b:=(\gamma_*\phi)|_{X_b}$, which satisfies
$$dd^c\gamma_*\phi_b=-(\gamma_*\alpha)|_{X_b},\quad \int_{X_b}\gamma_*\phi_b\gamma_*\omega=0,$$
and so we see that  $\gamma_*\phi_b$ converges locally smoothly on $U$ to $\gamma_*\phi_0$, and this completes the proof of the key claim.
\end{proof}

\begin{remark}[More general singular fibers]\label{rmk:assumptions_for_pi_trivializing}
When $\pi$ has general singular fibers, it is clear from the proof of \autoref{thm:preferred_currents} that we can still find $\phi\in C^\infty(X^\circ)$ such that $\alpha+dd^c\phi$ is $\pi$-trivial on $X^\circ$.
However, in general we see no reason for $\phi$ to be bounded on $X^\circ$.

Nevertheless, we do believe that the result in \autoref{prop:continuous_potentials} does hold without the restriction on the singular fibers, which boils down to showing that while $\phi$ need not be bounded on $X^\circ$ anymore, $\gamma_*\phi-\phi$ should extend to a continuous function on $X$.
Indeed the following simple argument (essentially the same as Tate's \cite{Tate}) proves that $\gamma_*\phi-\phi$ is bounded on $X^\circ$ in general.
This does not quite yield \autoref{prop:continuous_potentials}, which asserts the \emph{continuity} of the potentials of $\eta$.

Recall that $f=\pi^*u+\phi-\gamma_*\phi$ where $f$ is smooth on all of $X$. Then for $k\geq 1$ the Birkhoff sum $S_k(\gamma,f):=\sum_{i=0}^{k-1}\gamma^i_*f$ satisfies
\begin{equation*}S_k(\gamma,f)=k\pi^*u+\phi-\gamma^k_*\phi,\end{equation*}
so for every $b\in B^\circ$ and for every $k\geq 1$ we have
\begin{equation*}|u(b)|\leq \frac{1}{k}\sup_{X_b}|S_k(\gamma,f)|+\frac{1}{k}\sup_{X_b}|\phi-\gamma^k_*\phi|\leq \sup_X|f|+\frac{2}{k}\sup_{X_b}|\phi|,\end{equation*}
and letting $k\to\infty$ proves that $\|u\|_{L^\infty(B^\circ)}\leq \sup_X|f|$ as desired.
\end{remark}

\begin{theorem}[Current-valued pairing]
	\label{thm:current_valued_pairing}
	There exists a map
	\[
		\eta_B\colon
		\Aut^\circ_\pi(X)
		\times
		\left(\rightquot{[E]^\perp}{[E]}\right)
		\to
		\cA_{0,0}(B)
	\]
	which is linear in each coordinate and with the following properties.
	\begin{enumerate}
		\item If $\alpha$ is a closed smooth $(1,1)$-form with $[\alpha].[E]=0$, $\phi$ is provided by \autoref{thm:preferred_currents}, and $\gamma\in \Aut^{\circ}_\pi(X)$ then
		\[
			\gamma_{*}(\alpha + dd^c\phi)
			=
			(\alpha + dd^c \phi) + \pi^{*}\eta_B(\gamma,[\alpha])
		\]
		\item The map is compatible with integration and cup-product, i.e.
		\[
			\int_B \eta_B(\gamma,[\alpha]) = -[\xi(\gamma)].[\alpha]
		\]
		\item For an automorphism $\gamma\in \Aut^\circ_\pi(X)$ we have
		\[
			\eta_B(\gamma,\xi(\gamma))\geq 0 \quad \text{ as a current on }B.
		\]

		\item The map is symmetric, i.e. if $\xi_i=\xi(\gamma_i)$
		\[
			\eta_B(\gamma_1,\xi_2) = \eta_B(\gamma_2,\xi_1)
		\]
	\end{enumerate}
\end{theorem}
\begin{proof}
	By \autoref{prop:continuous_potentials}, there exists a current $\eta_B(\gamma,\alpha)$ such that
	\[
		\gamma_*(\alpha + dd^c\phi) = (\alpha + dd^c \phi) + \pi^*\eta_B(\gamma,\alpha)
	\]
	If we select a different representative in the same cohomology class, i.e. $\alpha'=\alpha + dd^c\psi$, then the current $\alpha'+dd^c\phi'$ provided by \autoref{thm:preferred_currents} differs from $\alpha+dd^c\phi$ by $dd^c$ of a function pulled back from $B$.
	Thus the expression $\gamma_*(\alpha+dd^c\phi)-(\alpha+dd^c\phi)$ depends only on the cohomology class $[\alpha]$.
	
	The definition of $\eta_B(\gamma,[\alpha])$ immediately yields the relations
	\begin{align*}
		\eta_B(\gamma,[\alpha_1]+[\alpha_2])
		& = \eta_B(\gamma,[\alpha_1]) + \eta_B(\gamma,[\alpha_2])\\
		\eta_B(\gamma_1\circ \gamma_2,[\alpha])
		& = \eta_B(\gamma_1,[\alpha]) + \eta_B(\gamma_2,[\alpha])
	\end{align*}
	The identity $\int_B \eta_B(\gamma,[\alpha]) = [\xi(\gamma)].[\alpha]$ follows from \autoref{eqn:cohomology_action_explicit} and the fact that $[\alpha].[E]=0$ by assumption.

	Symmetry is established in \autoref{cor:symmetry_of_pairing} and positivity is established in \autoref{cor:positivity_of_pairing}.
	The proofs of both of these facts involve taking a large power of an element and passing to the limit, and do not depend on each other.
	They do require understanding the iterates of forms that integrate nontrivially on fibers.
\end{proof}

\noindent We now study forms that integrate nontrivially on the fiber.

\begin{proposition}[Large iterates of one automorphism]
	\label{prop:large_iterates_of_one_automorphism}
	Let $\omega$ be a closed smooth $(1,1)$-form with $\int_{X_b}\omega=1$ and let $\gamma\in \Aut_\pi^{\circ}(X)$ be an automorphism.
    Define the closed $(1,1)$-form $\alpha$ by
	\[
		\gamma_*\omega = \omega + \alpha,
	\]
	so it satisfies $[\alpha].[E]=0$ and let $\phi$ be the function provided by \autoref{thm:preferred_currents} applied to $\alpha$.
	Then for $n\geq 0$ we have that
	\[
		\gamma_*^n \omega
		=
		\omega + n\cdot \alpha
		+ dd^c \left[n\phi - S_n(\gamma,\phi)\right]
		+ \tfrac{n(n-1)}{2}\pi^*\eta_{\gamma}
	\]
	where $\eta_{\gamma}:=\eta_B(\gamma,\xi(\gamma))$ is the current provided by \autoref{thm:current_valued_pairing}, and $S_{k}(\gamma,\phi)=\sum_{i=0}^{k-1} \gamma_*^i\phi$ is a Birkhoff sum.
\end{proposition}
\begin{proof}
	From the formula in \autoref{eqn:cohomology_action_explicit}, the class of $[\alpha]$ in $[E]^\perp/[E]$ is $\xi$, where $\xi=\xi(\gamma)$ is the element corresponding to the automorphism.
	Taking $\phi$ as provided by \autoref{thm:preferred_currents} gives
	\[
		\gamma_*(\alpha + dd^c\phi)
		=
		(\alpha+dd^c\phi) + \pi^*\eta_B(\gamma,\xi)
	\]
	where $\eta_B(\gamma,\xi)$ is the current on $B$ provided by \autoref{thm:current_valued_pairing}.
	Let $\eta_\gamma:=\eta_B(\gamma,\xi)$ as per the statement of the claim and note that more generally we have
	\begin{align*}
		\gamma_*^k(\alpha + dd^c\phi) & = (\alpha +dd^c\phi)
		+ \pi^*\eta_B(\gamma^k,\xi)\\
		& = (\alpha+dd^c\phi) + k\pi^*\eta_\gamma
	\end{align*}
	We then have by \autoref{thm:current_valued_pairing}
	\begin{align*}
		\gamma_*^n \omega & =
		\gamma_*^{n-1}(\gamma_*\omega) =
		\omega + \alpha + \gamma_*\alpha + \cdots + \gamma^{n-1}_*\alpha\\
		& = \omega + n\alpha + dd^c[n\phi - S_{n}(\gamma,\phi)]\\
		& + (1+2+\cdots + (n-1))\pi^*\eta_\gamma
	\end{align*}
	as required.	
\end{proof}

\begin{remark}[On notation for pairings]
	\label{rmk:on_notation_for_pairings}
From now on, to ease the notational burden and emphasize certain symmetries, we will use the notation
\[
	\eta(\gamma,\gamma'):=\pi^*\eta_B(\gamma,\xi(\gamma'))
\]
viewed as a $(1,1)$-current on $X$.
To ease notation, below we write pushed-forward forms or functions as $\gamma \omega$ instead of $\gamma_*\omega$.

Later on, in \autoref{sec:boundary_currents_direct_construction}, when the dependence of the pairing on the elliptic fibration will be relevant, we will write $\eta_{[E]}(\gamma,\gamma')$ for the corresponding current on $X$.
\end{remark}

\begin{proposition}[Large iterates of several automorphisms]
	\label{prop:large_iterates_of_several_automorphisms}
	Let $\omega$ be a closed smooth $(1,1)$-form with $\int_{X_b}\omega=1$ and let $\gamma_1,\ldots, \gamma_k\in \Aut_\pi^{\circ}(X)$ be automorphism.
    Define closed $(1,1)$-forms $\alpha_i$ with $[\alpha_i].[E]=0$ by
	\[
		\gamma_i\omega = \omega + \alpha_i,
	\]
	and let $\phi_i$ be the functions provided by \autoref{thm:preferred_currents} applied to $\alpha_i$.
	Then for $n_1,\ldots n_k \geq 0$ we have that
	\begin{align*}
		\gamma_k^{n_k}\cdots \gamma_1^{n_1} \omega
		&
		=
		\omega + n_k \alpha_k + \cdots + n_1 \alpha_1 + \\
		 & + dd^c \left[
		 \sum_{i=1}^k n_i \phi_i - S_{n_i}\left(\gamma_i, \gamma_k^{n_k}\cdots \gamma_{i+1}^{n_{i+1}}\phi_i\right)
		 \right]\\
		 & + \left[
		 \sum_{i=1}^k \frac{n_i(n_i-1)}{2}\eta(\gamma_i,\gamma_i)
		 +
		 \sum_{1\leq i <j \leq k}n_i n_j \eta(\gamma_i,\gamma_j)
		 \right]
	\end{align*}
	where as before $S_{n}(\gamma,\phi):=\phi + \cdots + \gamma^{n-1}\phi$ denotes a Birkhoff sum.	
\end{proposition}
\begin{proof}
	The case $k=1$ is \autoref{prop:large_iterates_of_one_automorphism}.
	The general formula is verified (by induction) on $k$ by direct calculation.
	Let us apply $\gamma_{k}^{n_{k}}$ to the formula assumed valid for $k-1$.
	Observe that the terms pulled back from $B$ are not changed, and we have, combining the previous calculations in \autoref{prop:large_iterates_of_one_automorphism} and \autoref{thm:current_valued_pairing}:
	\begin{align*}
		\gamma_k^{n_k}\omega & =
		\omega + n_k \alpha_k + dd^c\left[n_k\phi_k - S_{n_k}(\gamma_k,\phi_k)	\right] + \frac{n_k(n_k-1)}{2}\eta(\gamma_k,\gamma_k)\\
		\gamma_k^{n_k}\left(\sum_{i=1}^{k-1}n_i\alpha_i\right) & =
		\sum_{i=1}^{k-1}
		n_i\Big[\alpha_i +
		dd^c
		(\phi_i - \gamma_k^{n_k}\phi_i)
		+ \eta\left(\gamma_k^{n_k},[\alpha_i]\right)
		\Big]
	\end{align*}
	Recall that $[\alpha_i]=\xi(\gamma_i)$ and so the last term above can be also expressed as $n_k\cdot \eta(\gamma_k,\gamma_i)$.
	Finally we have
	\begin{align*}
		\gamma_k^{n_k}\sum_{i=1}^{k-1} n_i \phi_i - S_{n_i}\left(\gamma_i, \gamma_{k-1}^{n_{k-1}}\cdots \gamma_{i+1}^{n_{i+1}}\phi_i\right)
		= \sum_{i=1}^{k-1}
		n_i \gamma_k^{n_k}\phi_i -
		S_{n_i}\left(\gamma_i,\gamma_k^{n_k}\cdots \gamma_{i+1}^{n_{i+1}}\phi_i\right)
	\end{align*}
	where we used that the parabolic automorphisms $\gamma_i$ commute.
	Adding up the contributions yields the result, noting that the terms $n_i \gamma_k^{n_i}\phi_i$ cancel out.
\end{proof}

\begin{corollary}[Symmetry of pairing]
	\label{cor:symmetry_of_pairing}
	The pairing in \autoref{thm:current_valued_pairing} is symmetric, i.e. for two automorphisms $\gamma_1,\gamma_2\in \Aut_\pi^{\circ}$ with corresponding vectors $\xi_i=\xi(\gamma_i)$ we have
	\[
		\eta_B(\gamma_1,\xi_2) = \eta_B(\gamma_2,\xi_1) \text{ or as we shall also write }\eta(\gamma_1,\gamma_2) = \eta(\gamma_2,\gamma_1).
	\]
\end{corollary}
\begin{proof}
	Since the two automorphisms commute we have that $\gamma:=\gamma_1\gamma_2 = \gamma_2\gamma_1$.
	Select $\omega$ as in \autoref{prop:large_iterates_of_several_automorphisms} and apply the statement to $\gamma_1^n\gamma_2^n$ and $\gamma_2^n\gamma_1^n$.
	Dividing by $n^2$ and sending $n\to+\infty$ gives
	\begin{align*}
		\eta(\gamma_1,\gamma_1) + \frac{1}{2}\eta(\gamma_1,\gamma_2)
		+ \eta(\gamma_2,\gamma_2)
		=
		\lim_{n\to +\infty} \frac{\gamma_*^n \omega}{n^2}
		=
		\eta(\gamma_2,\gamma_2)+
		\frac{1}{2}\eta(\gamma_2,\gamma_1)
		+ \eta(\gamma_1,\gamma_1)
	\end{align*}
	thus yielding the symmetry.
\end{proof}

\noindent Using the symmetry property just established, we can write the conclusion of \autoref{prop:large_iterates_of_several_automorphisms} in a more intrinsic form.
Specifically, set $\gamma:=\gamma_k^{n_k}\cdots \gamma_1^{n_1}$ and then
\begin{align}
	\label{eqn:large_parabolic_automorphism}
	\begin{split}
	\gamma_* \omega
		&
		=
	\omega + n_k \left(\alpha_k - \frac{1}{2}\eta(\gamma_k,\gamma_k)\right) + \cdots + n_1\left( \alpha_1 - \frac{1}{2}\eta(\gamma_k,\gamma_k)\right) \\
	 & + dd^c \left[
	 \sum_{i=1}^k n_i \phi_i - S_{n_i}\left(\gamma_i, \gamma_k^{n_k}\cdots \gamma_{i+1}^{n_{i+1}}\phi_i\right)
	 \right]\\
	 & + \frac{1}{2}\eta(\gamma,\gamma)
	 \end{split}
\end{align}
In this form, the dominant (quadratic) term of order $-[\xi(\gamma)].[\xi(\gamma)]$ is the last one, whereas the intermediate ones are linear in the size of $\xi(\gamma)$ or bounded.

\begin{corollary}[Positivity of pairing]
	\label{cor:positivity_of_pairing}
	For any $\gamma\in \Aut^\circ_{\pi}(X)$ and any \Kahler form $\omega$ on $X$ with $\int_{X_b}\omega=1$, we have that
$$\lim_{n\to+\infty} \frac{\gamma^n_*\omega}{n^2}=\frac{1}{2}\eta(\gamma,\gamma),$$
in the weak topology. In particular we have
	\[
		\eta(\gamma,\gamma)\geq 0.
	\]
\end{corollary}
\begin{proof}
	In \autoref{eqn:large_parabolic_automorphism}, take $\omega$ to be a \Kahler form and express $\gamma$ as a product of generators.
	Then any weak limit of $\frac{\gamma^n_*\omega}{n^2}$ is on the one hand a positive current, and on the other hand equal to $\frac{1}{2}\eta(\gamma,\gamma)$.
\end{proof}

\subsubsection{Extension to real coefficients}
	\label{sssec:extension_to_real_coefficients}
Let us note that the current-valued bilinear pairing $\eta$ from \autoref{thm:current_valued_pairing} extends from the lattice $[E]^{\perp}/[E]$ to the vector space $([E]^{\perp}/[E])\otimes_{\bZ}\bR=:([E]^{\perp}/[E])_{\bR}$ by taking a set of generators $\gamma_i$ of $\Aut^{\circ}_{\pi}(X)$ and allowing the expression $\sum n_i \gamma_i$ to have real coefficients $n_i$.
We see, in particular, that restricted to the unit sphere, this gives a mapping to positive currents, which is moreover continuous in the $C^0$ topology of potentials.
Indeed, the current for an arbitrary real-valued point is given as a linear combination of a finite set of basis currents, with coefficients determined by the coordinates of the point, and the basis currents have continuous potentials.

\subsubsection{Summary}
	\label{sssec:summary_elliptic_fibrations}
\autoref{prop:large_iterates_of_several_automorphisms} applies to the case of $[\omega].[E]=1$.
We can reduce to this case by scaling, and the explicit formula is as follows.
Assume that $\gamma_i \omega = \omega + \alpha_i$ and $\phi_i$ is the result of applying \autoref{thm:preferred_currents} to $\alpha_i$.
Then applying \autoref{prop:large_iterates_of_several_automorphisms} to $\omega/([\omega].[E])$ and multiplying by $[\omega].[E]$ the result gives (with $\gamma=\gamma_k^{n_k}\cdots \gamma_1^{n_1}$):
\begin{align*}
	\gamma_k^{n_k}\cdots \gamma_1^{n_1} \omega
		&
		=
	\omega + n_k \left(\alpha_k - \frac{1}{2}\eta(\gamma_k,\gamma_k)\right) + \cdots + n_1\left( \alpha_1 - \frac{1}{2}\eta(\gamma_k,\gamma_k)\right) \\
	 & + dd^c \left[
	 \sum_{i=1}^k n_i \phi_i - S_{n_i}\left(\gamma_i, \gamma_k^{n_k}\cdots \gamma_{i+1}^{n_{i+1}}\phi_i\right)
	 \right]\\
	 & + \frac{([\omega].[E])}{2}\eta(\gamma,\gamma)
\end{align*}
Note that the middle terms, involving $dd^c$, are cohomologically trivial and also of linear magnitude in the size of $\gamma$.
Note also that the first row of terms has magnitude which is linear in the size of $\gamma$, while the last term is of quadratic magnitude.

For later estimates, we will use the notation
\begin{align}
	\label{eqn:norm_gamma}
	\norm{\gamma}_{\NS}^2:= -[\xi(\gamma)].[\xi(\gamma)]
\end{align}
which is a natural strictly positive-definite norm on $\Aut^{\circ}_{\pi}(X)$ after it is identified with $[E]^\perp/[E]$ via $\xi$.

Also importantly, let us note that once we fix a positive-definite norm $\norm{\cdot}$ on $\NS(X)$, there exists a constant $C=C(E)$ such that:
\begin{align}
	\label{eqn:norm_gamma_operator_square}
	\frac 1C \left(1+\norm{\gamma}^{2}_{\NS}\right) \leq \norm{\gamma}_{op} \leq C \left(1+\norm{\gamma}^{2}_{\NS}\right)
\end{align}
where $\norm{\gamma}_{op}$ refers to the operator norm of $\gamma$ acting (as a matrix) on $\NS(X)$.
Note, importantly, that $\norm{\gamma}^2_{\NS}$ appears and not $\norm{\gamma}_{\NS}$.
Indeed, for a unit vector $v$ whose cup product with $[E]$ is bounded below by a constant, the size of $\norm{\gamma v}$ is of the order $\norm{\gamma}_{\NS}^2$, up to some multiplicative constant, because of \autoref{eqn:cohomology_action_explicit}.
The estimate in \autoref{eqn:norm_gamma_operator_square} is compatible (and implies) the fact that the operator norm of $\gamma^k$ grows quadratically in $k$.




\section{Boundary currents, direct construction}
	\label{sec:boundary_currents_direct_construction}

\subsubsection*{Outline of section}
We can now construct the canonical currents on the boundary of the ample cone of a K3 surface.
The main result of this section is \autoref{thm:continuous_family_of_boundary_currents} that establishes the existence of a continuous family of boundary currents.

In \autoref{ssec:potentials_from_a_basis} we fix smooth representatives for every cohomology class, so that any other current is expressed relative to the fixed basis using a potential function.
In \autoref{ssec:existence_of_boundary_potentials} we establish the existence of the canonical boundary currents and establish their properties.
In \autoref{ssec:uniqueness_results_after_verbitsky_sibony} we include a result due to Sibony--Verbitsky that shows the uniqueness of the closed positive currents in irrational boundary classes.
Finally, in \autoref{ssec:boundary_currents_and_ricci_flat_metrics} we include some consequences for degenerating Ricci-flat metrics, and discuss further geometric interpretations of our construction.


\subsection{Potentials from a basis}
	\label{ssec:potentials_from_a_basis}

\subsubsection{Setup}
	\label{sssec:setup_potentials_from_a_basis}
For this section and the next, fix a basis $\omega_0,\omega_1\ldots, \omega_{\rho-1}$ spanning the subspace in cohomology given by the \Neron--Severi group $\NS(X)\subseteq H^{1,1}(X)$.
For a closed $(1,1)$-form $\omega$ with $[\omega]\in \NS(X)$, there exists a unique way to write it as
\[
	\omega = \sum c_i \omega_i + dd^c \phi(\omega) \qquad \text{ with }\int_X \phi(\omega)\dVol = 0\text{ and }c_i\in \bR.
\]
We will call $\phi(\omega)$ the \emph{non-canonical potential} of $\omega$.
Given a cohomology class $[\omega]$, we will write $A([\omega])$ for the $(1,1)$-form satisfying
\[
	A([\omega]) = \sum c_i \omega_i \qquad \text{ where }[\omega]=\sum c_i[\omega_i].
\]
With this notation, we always have
\begin{align}
	\label{eqn:omega_noncanonical_potential}
	\omega = A([\omega]) + dd^c \phi(\omega).
\end{align}
We will use the notation $\phi(\omega)$ also for currents which have continuous potentials, such as $\eta(\gamma,\gamma')$ defined in \autoref{sec:elliptic_fibrations_and_currents}.
From now on, all constants are allowed to depend on the choice of the initial representative $(1,1)$-forms and this will not be stated additionally.

\subsubsection{Mass and normalization}
	\label{sssec:mass_and_normalization}
We assume that $\omega_0$, the first element in the basis, is a \Kahler metric normalized to $[\omega_0]^2=1$ and such that the class $[\omega_0]$ coincides with the reference basepoint from \autoref{ssec:groups_and_spaces}.
Recall also that we defined
\[
	M([\omega]):=[\omega_0].[\omega] \text{ to be the \emph{mass} of }[\omega].
\]
To study the (non-canonical) potentials of \Kahler forms as they approach the boundary, we will always normalize them to have mass $1$.
We will say that $\omega$, or $[\omega]$, is \emph{normalized} if we have that
\[
	[\omega]^2\geq 0 \text{ and }[\omega].[\omega_0]=1,\text{ i.e. }M([\omega])=1 \text{ and it is in the positive cone.}
\]

\subsubsection{Action of automorphisms on potentials}
	\label{sssec:action_of_automorphisms_on_potentials}
For the action of an automorphism $\gamma\in \Aut(X)$ on forms and potentials, we will use the following expressions.
First we have from the definitions
\[
	\gamma_* A([\omega]) = A(\gamma_*[\omega])  + dd^c \phi\Big(\gamma_* A([\omega])\Big)
\]
For a general $(1,1)$-form $\omega$ we have, after acting by $\gamma$ on \autoref{eqn:omega_noncanonical_potential}:
\begin{align}
	\label{eqn:gamma_action_noncanonical_potential}
	\begin{split}
	\gamma_* \omega & =
	\gamma_* A([\omega]) + dd^c \gamma_* \phi(\omega)\\
	& =
	A\left(\gamma_*[\omega]\right) + dd^c \Big(\phi\big(\gamma_* A([\omega])\big) + \gamma_*\phi(\omega) \Big) \text{ and so:}\\
	\phi(\gamma_*\omega) & =
	\phi\left(\gamma_{*}A([\omega])\right) + \gamma_*\phi(\omega)
	\end{split}
\end{align}

\begin{proposition}[Boundedness for fixed automorphism]
	\label{prop:boundedness_for_fixed_automorphism}
	For $\gamma\in \Aut(X)$ fixed, there exists a constant $C(\gamma)$ such that for any $[\omega]$ with $[\omega]^2\geq 0$ and $[\omega].[\omega_0]\geq 0$ we have the bound on potentials:
	\[
		\norm{\phi\big(\gamma_* A([\omega])\big)}_{C^0} \leq C(\gamma)M([\omega])
	\]
	Furthermore, for any $[\omega],[\omega']$ which are normalized we have
	\[
		\norm{\phi\big(\gamma_* A([\omega])\big) - \phi\big(\gamma_* A([\omega'])\big)}_{C^0} \leq 2C(\gamma)
	\]
\end{proposition}
\begin{proof}
	By scaling the first inequality we can normalize to $M([\omega])=1$ and then the bound follows since $[\omega]$ belongs to a compact set and $\phi\big(\gamma_* A([\omega])\big)$ varies continuously (even smoothly) on it.
	The second inequality follows from the first by the triangle inequality, and using that $M([\omega])=1$.
\end{proof}

\subsubsection{Estimates for elliptic fibrations}
	\label{sssec:estimates_for_elliptic_fibrations}
Suppose now that $X\xrightarrow{\pi}\bP^1$ is an elliptic fibration, with class of fiber $[E]$, with parabolic automorphism group $\Aut^\circ_{\pi}(X)$ as in \autoref{sec:elliptic_fibrations_and_currents}.
We will sharpen the estimate in \autoref{prop:boundedness_for_fixed_automorphism} to include the size of a parabolic automorphism $\gamma$ (see \autoref{sssec:summary_elliptic_fibrations}) and obtain convergence towards a positive current.

Specifically, fix $\gamma_1,\ldots, \gamma_k$ generators of $\Aut_{\pi}^{\circ}(X)$ with $k=\rho-2$ and a constant $C([E])$ depending on the fibration.
We will say that $[\omega]$ is \emph{parabolically normalized} if, in addition to being normalized (i.e. $[\omega]^2\geq 0$ and $[\omega].[\omega_0]=1$) we also have that $[\omega].[E]\geq \frac{1}{C([E])}$.
Note that this implies that there also exists a constant $C'([E])$ such that $C'([E])\geq [\omega].[E]$ since $[\omega]$ is restricted to a compact set.
Observe that being parabolically normalized is equivalent to being transverse (with a uniform constant) for any element $\gamma\in \Aut_\pi^\circ(X)$, i.e. satisfying the inequality $\norm{\gamma [\omega]}\geq \frac{1}{C''}\norm{\gamma}_{op}\norm{\omega}$ for a constant $C''$ depending only on $[E]$.
Finally, recall that for norms of parabolic automorphisms we have $\norm{\gamma}_{op}\approx \norm{\gamma}_{\NS}^2$, see \autoref{eqn:norm_gamma_operator_square}.

\begin{proposition}[Parabolic convergence to the boundary]
	\label{prop:parabolic_convergence_to_the_boundary}
	Consider a parabolic automorphism $\gamma=\gamma_k^{n_k}\cdots \gamma_1^{n_1}$ written as a product of generators and $[\omega]$ a parabolically normalized class.
	\begin{enumerate}
		\item There exists a constant $C_1=C_1(E)$ giving the mass estimate:
		\[
			\norm{
			\frac{M(\gamma_*[\omega])}{M\left(\frac{[\omega].[E]}{2}\eta(\gamma,\gamma)\right)} - 1}
			\leq
			\frac{C_1}{1 + \norm{\gamma}_{\NS}}
		\]
		\item
		\label{large_parabolic_brings_close_to_eta}
		There exists a constant $C_2=C_2(E)$ giving the estimate on potentials:
		\[
			\norm{
			\frac{\phi\left(\gamma_*A([\omega])\right)}{M(\gamma_* [\omega])
			}
			-
			\frac{\phi\left(\eta(\gamma,\gamma)\right)}{M\left(\eta(\gamma,\gamma)\right)}
			}_{C^0}
			\leq
			\frac{C_2}{1 + \norm{\gamma}_{\NS}}
		\]
		\item \label{key_parabolic_estimate}
		We also have the estimate
		\[
			\norm{
			\phi\left(\gamma_* A([\omega])\right)
			}_{C^0}
			\leq C_3 \left(1+\norm{\gamma}_{\NS}^2\right)\norm{[\omega]}
		\]
		which holds for arbitrary $[\omega]$.
	\end{enumerate}
\end{proposition}
\noindent The estimates could be rewritten in terms of $\norm{\gamma}_{op}$, keeping in mind the bound in \autoref{eqn:norm_gamma_operator_square} that gives $\norm{\gamma}_{op}\approx \norm{\gamma}_{\NS}^2$.

\begin{proof}
	We apply the expression in \autoref{sssec:summary_elliptic_fibrations} to $A([\omega])$.
	Part (i) then follows immediately since it only involves cohomology classes.
	For part (ii), recall that the generators $\gamma_i$ are fixed at the start, so as $[\omega]$ varies in the compact set of parabolically normalized classes, so do the $\alpha_i$ so we have uniform $C^0$ (even $C^k$) bounds on the functions $\phi_i$ appearing in the formula.
	We then have the estimates
	\begin{align*}
		\norm{\phi(A([\omega]))}_{C^0} & = O(1)\\
		\norm{\phi\left(\sum n_i (\alpha_i - \tfrac{1}{2}\eta(\gamma_i,\gamma_i))\right)}_{C^0} & = O\left(\norm{\gamma}_{\NS}\right)\\
		\norm{\sum {n_i \phi_i} - S_{n_i}(\gamma_i, \gamma_k^{n_k}\cdots \gamma_{i+1}^{n_{i+1}}\phi_i)}_{C^0} & = O\left(\norm{\gamma}_{\NS}\right)
	\end{align*}
	so the main contribution comes from the term $\frac{[\omega].[E]}{2}\eta(\gamma,\gamma)$ which has mass $\frac{([\omega].[E])}{2}\norm{\gamma}^2_{\NS}$, while the other terms have size $O(\norm{\gamma}_{\NS})$, and hence the quadratic term dominates the rest.
	So to get the claimed estimate in part (ii), multiply both sides by $M(\gamma_*[\omega])$ and use the estimates from part (i) that give
	\[
		\frac{M(\gamma_*[\omega])}{M(\eta(\gamma,\gamma))}=\tfrac12 [\omega].[E] + O\left(\frac{1}{1+\norm{\gamma}_{\NS}}\right)
	\]
	and then observe that we have
	\[
		\phi(\gamma_*A([\omega])) = \phi\left(\tfrac{[\omega].[E]}{2}\eta(\gamma,\gamma)\right) + O(\norm{\gamma}_{\NS})
	\]
	This last estimate also gives the desired bound in (iii) for parabolically normalized classes, keeping in mind that $\eta(\gamma,\gamma)$ scales quadratically in $\gamma$ and is uniformly bounded for $\gamma$ in a bounded set.
	To obtain the estimate in (iii) for arbitrary classes, observe that we can fix a basis of parabolically normalized classes, apply the estimate to each of them, and sum up the estimates (using the linearity of the assignments $[\omega]\mapsto A([\omega])$ and $\omega\mapsto \phi(\omega)$).
\end{proof}

\begin{corollary}[Parabolic convergence to current]
	\label{cor:parabolic_convergence_to_current}
	Suppose that $\{\chi_i\}\in \Aut^{\circ}_{\pi}(X)$ is a sequence of elements with $\norm{\chi_i}\to +\infty$ and such that $\frac{\chi_i}{\norm{\chi_i}_{\NS}}\in \left([E]^\perp/E\right)_\bR$ convergences to an element $\chi$.
	Then for any smooth $(1,1)$-form $\omega$ we have the convergence:
	\[
		\frac{(\chi_i)_* \omega}{M((\chi_i)_*\omega )} \to
		\frac{\eta(\chi,\chi)}{M(\eta(\chi,\chi))}
	\]
	in the $C^0$ topology of the (non-canonical) potentials.
\end{corollary}
Note that the space of closed $(1,1)$-forms with locally $C^0$-potentials is a finite-dimensional extension of the Banach space $C^0(X)/const.$ and hence carries a natural topology.
In fact we can also put a norm by using the fixed basis of cohomology from \autoref{sssec:setup_potentials_from_a_basis}.
\begin{proof}
	Given a fixed $\omega$, there is some constant $C(E)$ such that $\omega$ is parabolically normalized.
	Then the result follows from part (ii) of \autoref{prop:parabolic_convergence_to_the_boundary}.
\end{proof}



\subsection{Existence of boundary potentials}
	\label{ssec:existence_of_boundary_potentials}

\subsubsection{Rational boundary points and currents}
	\label{sssec:rational_boundary_points_and_currents}
Recall that the boundary of interest $\partial^\circ \Amp_c(X)$ has a stratification as in \autoref{eqn:boundary_stratification}.
It consists of irrational cohomology classes, as well as elements of the form $k[E]\times [\xi]$ where $k$ is a real strictly positive scalar, $[E]$ is the (integral, primitive) class of an elliptic fibration, and $[\xi]\in \bP\left([E]^\perp/[E]\right)$.

For a class $[E]$ associated to a genus one fibration, from now on we denote the current-valued pairing from \autoref{thm:current_valued_pairing} by $\eta_{[E]}$, see also \autoref{rmk:on_notation_for_pairings}.
This is to distinguish it from pairings of other genus one fibrations, and to distinguish it from the global family of boundary currents that the main result provides.

By our standing assumptions on K3 surfaces in \autoref{sssec:standing_assumptions_main}, and due to Sterk's theorem, the subgroup $\Aut^\circ_\pi(X)$ of parabolic automorphisms gives a lattice in $[E]^\perp/[E]$ using the embedding $\xi_{[E]}$ given by $\xi_{[E]}(\gamma):=\gamma_*[\omega]-[\omega]$ where $[\omega].[E]=1$ is fixed, see \autoref{sssec:linearizing_automorphisms}.
By extension of scalars, we will use below the currents $\eta_{[E_0]}(\xi_1,\xi_2)$ where $\xi_i\in [E]^\perp/[E]$ are real vectors.

Here is the main result of this section.
Recall that $\partial^\circ \Amp_c(X)$ is defined in \autoref{sssec:blown_up_boundary_of_the_ample_cone} and $\cZ_{1,1}(X)$ denotes closed $(1,1)$-currents.

\begin{theorem}[Continuous family of boundary currents]
	\label{thm:continuous_family_of_boundary_currents}
	There exists a map
	\[
		\eta_X\colon \partial^\circ \Amp_c(X) \to \cZ_{1,1}(X)
	\]
	compatible with the map from $\partial^\circ \Amp_c(X)$ to $\partial \Amp(X)\subset \NS(X)$, and from $\cZ_{1,1}(X)$ to $H^{1,1}(X)$.
	The map has the following additional properties:
	\begin{enumerate}
		\item \label{eta_positive_currents}
		For any $[\alpha]\in \partial^\circ \Amp_c(X)$ the current $\eta_X([\alpha])$ is positive and has continuous potentials.
		\item \label{eta_is_C0_continuous}
		The map $\eta_X$ is continuous for the $C^0$-topology on potentials.
		\item \label{eta_is_equivariant}
		The map $\eta_X$ is $\Aut(X)$-equivariant, and also equivariant for the scaling action of $\bR_{>0}$.
		\item
		\label{eta_on_spheres}
		On the real projective $(\rho-2)$-spaces in $\partial^\circ \Amp_c(X)$ associated to a rational boundary ray $\bR[E]$ corresponding to a genus one fibration, the map $\eta_X$ agrees with the current-valued pairing from \autoref{thm:current_valued_pairing} with the assignment
		\[
			\eta_X(k[E]\times [\xi])=k\cdot \frac{\eta_{[E]}([\xi],[\xi])}{\norm{\xi}_{\NS}^2} \text{ for }[\xi]\in \bP([E]^{\perp}/[E]), k\in\mathbb{R}.
		\]
		Note that the formula uses, but is independent of, a choice of a lift of $[\xi]$ to $[E]^\perp/[E]$.
		\item For any $[\alpha]\in \partial^\circ \Amp_c(X)$ let $\{\gamma_i\}$ be the associated sequence of automorphisms as in \autoref{sssec:fixed_finite_set_of_generators}.
		Then for any smooth representative $\omega_0$ of a cohomology class $[\omega_0]\in \Amp^1(X)$ we have that
		\[
			\lim_{n\to +\infty} \frac{\gamma_1\cdots \gamma_n \omega_0}{M(\gamma_1\cdots \gamma_n[\omega_0])}
			=\eta_X\left(\frac{[\alpha]}{M([\alpha])}\right)
		\]
		\item For boundary classes $[\alpha_\pm]$ expanded/contracted by hyperbolic automorphisms of $X$, the currents $\eta_X([\alpha_{\pm}])$ agree with those constructed by Cantat \cite{Cantat_Dynamique-des-automorphismes-des-surfaces-K3}.
	\end{enumerate}
	\hspace{\parindent}Furthermore, the combination of either \ref{eta_positive_currents} and \ref{eta_is_C0_continuous}, or the combination of \ref{eta_positive_currents} and \ref{eta_on_spheres}, uniquely determine $\eta_X$.
\end{theorem}
\begin{proof}
	To each boundary point $[\alpha]\in \partial^\circ \Amp_c(X)$ we have an associated sequence of automorphisms $\gamma_1,\gamma_2,\ldots$ from the list fixed in \autoref{sssec:fixed_finite_set_of_generators}.
	We will check that the assignment (after normalization)
	\[
		\frac{[\alpha]}{M([\alpha])} \mapsto \lim_{n\to +\infty} \frac{\gamma_1\cdots \gamma_n \omega_0}{M(\gamma_1\cdots \gamma_n[\omega_0])}
	\]
	works, i.e. the limit exists and has the required properties.
	When we approach a parabolic boundary point, we interpret the limit $\lim_{n\to +\infty}$ as $\lim_{i\to +\infty}$ and we have the sequence $\gamma_1\cdots \gamma_{n-1}\gamma_{n,i}$ instead, where the sequence $\gamma_{n,i}$ belongs to the corresponding parabolic subgroup.

	Once we establish \ref{eta_positive_currents} and \ref{eta_is_C0_continuous}, i.e. positivity and continuity in the $C^0$-topology of potentials, the rest follows.
	Indeed, uniqueness (and hence equivariance) follows from \autoref{thm:uniqueness_in_irrational_classes_verbitsky_sibony} below, due to Sibony--Verbitsky, stating that the closed positive currents in the irrational boundary classes are unique.
	Since the irrational points are dense in both $\partial \Amp(X)$ and $\partial^\circ \Amp_c(X)$ the uniqueness and equivariance properties of $\eta$ follow from its continuity.
	One can alternatively use the dense set of classes expanded by hyperbolic automorphisms, on which uniqueness holds by a simpler, by now standard, argument.
	Again continuous dependence of potentials on the boundary class, for the construction of $\eta$ below, implies the uniqueness and equivariance.

	To define the map at the set-theoretical level, recall from \autoref{sssec:blown_up_boundary_of_the_ample_cone} that
	\[
		\partial^\circ\Amp_c(X)=\partial \Amp(X)^{irr}\coprod_{[E]}\left(\bR_{>0}[E]\right)\times \bP\left([E]^{\perp}/[E]\right)
	\]
	where $\partial \Amp(X)^{irr}$ denotes the irrational points and $\bR_{>0}[E]$ denotes the ray corresponding to the parabolic boundary class $[E]$.
	For $[\alpha]\in \partial^\circ\Amp_c(X)$, we divide the discussion into two cases, according to whether it is in the rational or irrational part.

\subsubsection*{Case $[\alpha]$ is rational}
	At the rational points of the form $[\alpha]:=[E]\times [\xi]$, where $\xi\in \bP\left([E]^\perp/[E]\right)$ define the current to be
	\[
		\eta\left([E]\times [\xi]\right) := \frac{\eta_{[E]}(\xi,\xi)}{\norm{\xi}^2_{\NS}}
	\]
	which is in particular in the cohomology class $[E]$.
	Extend the assignment by scaling to the entire ray $\bR_{>0}[E]$.
	Positivity follows from \autoref{thm:current_valued_pairing}.

	For continuity at this point, we argue as follows.
	First, by applying a finite sequence of generators $\gamma_1([\alpha]),\ldots,\gamma_{n-1}([\alpha])$ (given by a corresponding geodesic going into that cusp) we can assume that $\left(\gamma_1([\alpha])\cdots\gamma_{n-1}([\alpha])\right)^{-1}[E]$ is one of the fixed representatives for $\Aut(X)$-equivalences of genus one fibrations, say $[E_0]$.
	We can thus reduce to this case.
	By the construction in \autoref{sssec:fixed_finite_set_of_generators} for any $R>0$, there exists a neighborhood $U$ of the point $[\alpha]$ in $\partial^\circ\Amp_c(X)$ such that any geodesic $s$ starting at $[\omega_0]$ and landing in the neighborhood $U$ must have that $\gamma_{1}(s)$ is a parabolic element in $\Gamma_{[E_0]}$, with $\norm{\xi(\gamma_{1}(s))}_{\NS}\geq R$, and furthermore such that the ray determined by $\xi(\gamma_{1}(s))$ is in the prescribed neighborhood of $[\xi]$ in $\bP\left([E]^{\perp}/[E]\right)$.
	The estimates in \autoref{sssec:estimates_for_continuity_at_a_parabolic_boundary_point} below establish the required continuity.

\subsubsection*{Case $[\alpha]$ is irrational.}	
	At the irrational points, the map $\eta$ is defined by the current whose noncanonical potential is provided by \autoref{prop:cauchy_property_for_a_fixed_geodesic}.
	The proposition gives us a current normalized to mass $1$, but the map is extended naturally by scaling to the entire ray.
	Positivity follows by construction, and continuity follows from the estimate provided by \autoref{prop:cauchy_property_for_a_fixed_geodesic} and the fact that given $n$, there is a neighborhood of the given irrational endpoint such that geodesics $s$ landing in the neighborhood must have the sequence $\gamma_i(s)$ agree with that of the irrational point up to $n$.
\end{proof}

\subsubsection{Setup}
	\label{sssec:setup_existence_of_boundary_potentials}
To ease notation, all actions of group elements $\gamma\in \Aut(X)$ on $(1,1)$-forms and on functions are by pushforward and we write $\gamma \omega$ for $\gamma_*\omega$ and similarly for functions.
Applying \autoref{eqn:gamma_action_noncanonical_potential} successively, we have the formula
\begin{align}
	\label{eqn:gamma1n_omega}
	\begin{split}
	\phi(\gamma_1\cdots \gamma_n\omega) & =
	\phi\big(\gamma_1 A(\gamma_2\cdots \gamma_n[\omega])\big)+ \cdots \\
	& + \gamma_1 \cdots \gamma_{i-1}\phi\big(\gamma_i A(\gamma_{i+1}\cdots \gamma_n[\omega])\big)+ \cdots \\
	& + \gamma_1 \cdots \gamma_n \phi(\omega)
	\end{split}
\end{align}
Note that this is the correct order of multiplication, i.e. in the sequence $\gamma_1\cdots \gamma_n \omega$ the next term is obtained by applying $\gamma_{n+1}$ ``based'' at the previous element, i.e. we must conjugate $\gamma_{n+1}$ by $\gamma_1\cdots \gamma_n$ and apply it to the point, hence we get $\gamma_1\cdots \gamma_{n+1}\omega$.
We will use the following notation to occasionally shorten expressions:
\[
	\omega_{i,j}:=\gamma_i\cdots \gamma_j \omega_0 \qquad
	M_{i,j}=M([\omega_{i,j}]) \text{ with }i\leq j.
\]
\subsubsection{The necessary bounds}
	\label{sssec:the_necessary_bounds_currents}
Let us summarize the inequalities that we will use in the proof.
From \autoref{ssec:groups_and_spaces} and \autoref{ssec:linear_algebra_and_hyperbolic_geometry} we have that
\begin{align}
	\label{eqn:useful_inequalities}
	\begin{split}
	e^{\lambda_i + \cdots + \lambda_j}\leqapprox &
	\quad M_{i,j}\quad \quad \leqapprox
	e^{\lambda_i + \cdots + \lambda_j}\\
	e^{\lambda_i + \cdots + \lambda_j}\leqapprox & \norm{\gamma_i\cdots \gamma_j}_{op}\leqapprox
	e^{\lambda_i + \cdots + \lambda_j}
	\end{split}
\end{align}
for all $j\geq i \geq 1$.
Additionally, we will use that $[\omega_0]$ and $\gamma_{j+1}\cdots \gamma_{j+k}[\omega_0]$ is transverse for $\gamma_i\cdots \gamma_j$ and hence the estimates in \autoref{prop:contraction_from_transversality} and \autoref{def:transversality_of_a_vector_relative_to_a_transformation} hold.
We will also use systematically the bound
\begin{align}
	\label{eqn:basic_bound_potential_all_automorphisms}
	\norm{\phi\big(\gamma A([\omega])\big)}_{C^0}\leqapprox \norm{\gamma}_{op} \norm{[\omega]}
\end{align}
which is valid when $\gamma$ belongs to the fixed set of generators (since we consider a finite number of such generators), or to one of the finitely many fixed parabolic subgroups.
For elements of parabolic subgroups, this is \autoref{prop:parabolic_convergence_to_the_boundary}\ref{key_parabolic_estimate}, and for a fixed finite list of elements this is \autoref{prop:boundedness_for_fixed_automorphism}.

\begin{proposition}[Boundedness of potentials]
	\label{prop:boundedness_of_potentials}
	Let $s$ be any geodesic starting at $[\omega_0]$ and landing at an irrational point on the boundary, with associated sequence of generators $\gamma_i=\gamma_i(s)$ (see \autoref{sssec:fixed_finite_set_of_generators}).
	Then the sequence of normalized potentials
	\[
		\frac{\phi(\gamma_1\cdots\gamma_n \omega_0)}{M(\gamma_1\cdots\gamma_n [\omega_0])}
	\]
	is bounded by a constant independent of the geodesic.
	The same boundedness holds for any fixed $i$ and the sequence $\frac{\phi(\gamma_i\cdots \gamma_n \omega_0)}{M(\gamma_i\cdots \gamma_n[\omega_0])}$, with a uniform constant independent of $i$ or the geodesic.
\end{proposition}
\begin{proof}
	We estimate the $C^0$ norm of every term in \autoref{eqn:gamma1n_omega}, taking into account that pushing forward by $\gamma_1\cdots \gamma_{i-1}$ does not affect the $C^0$-norm.
	The manipulations are:
	\begin{align*}
		\norm{\frac{\phi(\gamma_i A(\gamma_{i+1}\cdots \gamma_n[\omega_0]))}{M(\gamma_1\cdots \gamma_n [\omega_0])}}_{C^0} & =
		\norm{\frac{\phi(\gamma_i A([\omega_{i+1,n}]))}{M_{1,n}}}_{C^0}\\
		& =
		\frac{M_{i+1,n}}{M_{1,n}}
		\norm{\phi\left(\gamma_i \frac{A([\omega_{i+1,n}])}{M_{i+1,n}}\right)}_{C^0}\\
		& \leqapprox
		\frac{M_{i+1,n}}{M_{1,n}}\norm{\gamma_i}_{op} \leqapprox
		\frac{e^{\lambda_{i+1}+\cdots + \lambda_n}}{e^{\lambda_1 + \cdots + \lambda_n}} \norm{\gamma_i}_{op}\\
		& \leqapprox
		e^{-(\lambda_1 + \cdots + \lambda_{i-1})}
	\end{align*}
	Using the estimate from \autoref{prop:displacement_estimate} allows us to sum the geometric series and conclude.
\end{proof}

\begin{proposition}[Cauchy property for a fixed geodesic]
	\label{prop:cauchy_property_for_a_fixed_geodesic}
	Let $s$ be any geodesic starting at $[\omega_0]$ and landing at an irrational point on the boundary, with associated sequence of generators $\gamma_i=\gamma_i(s)$ (see \autoref{sssec:fixed_finite_set_of_generators}).
	Then the sequence of normalized potentials
	\[
		\frac{\phi(\gamma_1\cdots\gamma_n \omega_0)}{M(\gamma_1\cdots\gamma_n [\omega_0])} = \frac{\phi(\omega_{1,n})}{M_{1,n}}
	\]
	forms a Cauchy sequence and we furthermore have the estimate
	\[
	\norm{
		\frac{\phi(\gamma_1\cdots\gamma_n \omega_0)}{M(\gamma_1\cdots\gamma_n [\omega_0])}
		-
		\frac{\phi(\gamma_1\cdots\gamma_{n+m} \omega_0)}{M(\gamma_1\cdots\gamma_{n+m} [\omega_0])}
	}_{C^0}
	\leqapprox ne^{-\delta \cdot n}
	\]
	for some fixed $\delta>0$ (depending only on the geometric constructions in \autoref{sec:hyperbolic_geometry_background}).

	Therefore, there exists a closed positive current with continuous potentials in the boundary class normalized to mass $1$ corresponding to the geodesic $s$.
\end{proposition}
\begin{proof}
Since $\omega_0$ is a \Kahler metric, any weak limit of the sequence is a positive current.
The estimate on the sequence of potentials implies, in particular, that we also have $C^0$ convergence of potentials and the limit has a continuous potential.

To analyze the sequence of potentials, we will use the expansion in \autoref{eqn:gamma1n_omega} to match up the first $n$ terms and estimate their difference.
For the remaining $m$ terms in $\phi(\omega_{1,n+m})$, we use the same type of estimate as in \autoref{prop:boundedness_of_potentials}, plus the estimates on masses \autoref{eqn:useful_inequalities}, to find
\begin{align*}
	\frac{\norm{\phi(\gamma_{n+i}A([\omega_{n+i+1,n+m}]))}_{C^0}}{M_{1,n+m}}
	& \leqapprox
	\frac{\norm{\gamma_{n+i}}_{op} M_{n+i+1,n+m}}{M_{1,n+m}}\\
	& \leqapprox e^{-(\lambda_1+\cdots + \lambda_n)}e^{-(\lambda_{n+1} + \cdots + \lambda_{n+i-1})}
\end{align*}
We next apply the estimate from \autoref{prop:displacement_estimate} to sum the above estimate in $i$ and obtain a bound of $\leqapprox e^{-(\lambda_1+\cdots +\lambda_n)}\leqapprox e^{-\delta n}$.

For the remaining $n$ differences, we need to estimate (omitting $\gamma_1\cdots \gamma_{i-1}$ applied to both terms, which does not affect the $C^0$ norm):
\begin{multline*}
	\norm{
	\frac{\phi(\gamma_iA(\gamma_{i+1}\cdots \gamma_n[\omega]))}{M_{1,n}}
	-
	\frac{\phi(\gamma_i A(\gamma_{i+1}\cdots \gamma_{n+m}[\omega]))}{M_{1,n+m}}
	}_{C^0}
	=\\
	\norm{
	\frac{\phi(\gamma_i A([\omega_{i+1,n}]))}{M_{1,n}}
	-
	\frac{\phi(\gamma_i A([\omega_{i+1,n+m}]))} {M_{1,n+m}}
	}_{C^0}
\end{multline*}
The first step is to rewrite each term in normalized form as
\[
	\frac{M_{i+1,n}}{M_{1,n}}{\phi\left(\gamma_i A\left(\left[\frac{\omega_{i+1,n}}{M_{i+1,n}}\right]\right)\right)}
	\text{ and }
	\frac{M_{i+1,n+m}}{M_{1,n+m}}{\phi\left(\gamma_i A\left(\left[\frac{\omega_{i+1,n+m}}{M_{i+1,n+m}}\right]\right)\right)}
\]
To proceed, we will use the estimate from \autoref{prop:contraction_from_transversality}:
\[
	\norm{
	\frac{[\omega_{i,j}]}{M_{i,j}} -
	\frac{\gamma_i\cdots \gamma_j [\alpha]}{M(\gamma_i\cdots \gamma_j[\alpha])}
	}
	\leq
	\frac{C}{\norm{\gamma_i\cdots \gamma_j}_{op}}
\]
whenever $[\alpha]$ is transverse to $\gamma_i\cdots \gamma_j$.
Let us apply this estimate to the expressions to which we apply the operator $A$:
\begin{align*}
	\frac{[\omega_{i+1,n+m}]}{M_{i+1,n+m}}
	& =
	\frac{\gamma_{i+1}\cdots \gamma_n [\omega_{n+1,n+m}]}{M(\gamma_{i+1}\cdots \gamma_n [\omega_{n+1,n+m}])}
	=
	\frac{[\omega_{i+1,n}]}{M_{i+1,n}} + \frac{O(1)}{\norm{\gamma_{i+1}\cdots \gamma_n}_{op}}
\end{align*}
We now also take $[\alpha]= \frac{[\omega_{n+1,n+m}]}{M_{n+1,n+m}}$ to find
\begin{align*}
	\frac{[\omega_{i+1,n}]}{M_{i+1,n}} -
	\frac{\gamma_{i+1}\cdots \gamma_n\frac{[\omega_{n+1,n+m}]}{M_{n+1,n+m}}}{M\left(\gamma_{i+1}\cdots \gamma_n\frac{[\omega_{n+1,n+m}]}{M_{n+1,n+m}}\right)}
	= \frac{O(1)}{\norm{\gamma_{i+1}\cdots \gamma_n}_{op}}
\end{align*}
So we have the estimate for the masses:
\begin{align*}
	\abs{\frac{M_{i+1,n}}{M_{1,n}}
	-
	\frac{M_{i+1,n+m}}{M_{1,n+m}}}
	& =
	\abs{
	\frac{1}{M\left(\gamma_1\cdots \gamma_i \frac{\omega_{i+1,n}}{M_{i+1,n}}\right)}
	-
	\frac{1}{M\left(\gamma_1\cdots \gamma_i \frac{\omega_{i+1,n+m}}{M_{i+1,n+m}}\right)}
	}
	\\
	& =
	\frac{M\left(\gamma_1\cdots \gamma_i \frac{O(1)}{\norm{\gamma_{i+1}\cdots \gamma_n}_{op}}\right)}
	{M\left(\gamma_1\cdots \gamma_i \frac{\omega_{i+1,n}}{M_{i+1,n}}\right)
	\cdot
	M\left(\gamma_1\cdots \gamma_i \frac{\omega_{i+1,n+m}}{M_{i+1,n+m}}\right)
	}
	\\
	& \leq
	\frac{O(1)}{\norm{\gamma_1\cdots \gamma_i}_{op}{\norm{\gamma_{i+1}\cdots \gamma_n}_{op}}}\\
	& \leqapprox e^{-(\lambda_1 + \cdots + \lambda_n)}
\end{align*}
where we use for the last inequality again transversality to say that the mass, and the norm of the operator are the same up to constants.

So the terms we must subtract from each other are (using the above estimates)
\begin{align*}
	\left(\frac{M_{i+1,n+m}}{M_{1,n+m}} +
	\frac{O(1)}{e^{\lambda_1 + \cdots + \lambda_n}}\right)
	& \cdot \phi\left(\gamma_i A \left(\frac{[\omega_{i+1,n}]}{M_{i+1,n}}\right)\right)\\
	\intertext{ and }
	\left(\frac{M_{i+1,n+m}}{M_{1,n+m}}
	\phantom{e^{-(\lambda_1 + \cdots + \lambda_n)}}
	\right)
	& \cdot \phi\left(\gamma_i A \left(\frac{[\omega_{i+1,n}]}{M_{i+1,n}}
	+ \frac{O(1)}{e^{\lambda_{i+1}+\cdots+\lambda_n}}\right)\right)
\end{align*}
Upon subtraction, we are left with the two terms
\begin{align*}
	\frac{O(1)}{e^{\lambda_1+\cdots + \lambda_n}}
	\cdot \phi\left(\gamma_i A \left(\frac{[\omega_{i+1,n}]}{M_{i+1,n}}\right)\right)\text{ and }
	\frac{M_{i+1,n+m}}{M_{1,n+m}}
	\cdot \frac{\phi\left(\gamma_i A\left(O(1)\right)\right)}{e^{\lambda_{i+1+\cdots +\lambda_n}}}
\end{align*}
Using that $\norm{\phi(\gamma_iA(O(1)) )}_{C^0}\leq O(1)\norm{\gamma_i}_{op}$ which is \emph{valid for both situations, when $\gamma_i$ is parabolic, or when $\gamma_i$ belongs to some fixed set of generators} (see \autoref{eqn:basic_bound_potential_all_automorphisms}),
we get that each term in the difference is bounded by
\[
	\norm{\text{difference}}_{C^0}
	\leqapprox e^{-(\lambda_1+\cdots + \lambda_n)}e^{\lambda_i}
\]
We again apply the estimate from \autoref{prop:displacement_estimate} to $\lambda_1+\cdots +\lambda_{i-1}$ and to $\lambda_{i+1}+\cdots + \lambda_n$ to find that the term above is bounded by $O(1)\cdot e^{-n\delta}$, and since we have $n$ terms the result follows.
\end{proof}

\subsubsection{Estimates for continuity at a parabolic boundary point}
	\label{sssec:estimates_for_continuity_at_a_parabolic_boundary_point}
We now assemble also the estimates that give continuity at the parabolic boundary points, needed for the proof of \autoref{thm:continuous_family_of_boundary_currents}.
Recall that we have a sequence $\gamma_1,\gamma_2,\ldots$ such that $\gamma_1\in \Gamma_{[E_0]}$ and $\norm{\xi(\gamma_1)}_{\NS}\geq R$, and we are allowed to choose $R$ as large as we want.
We need to estimate a potential, which we rewrite using \autoref{eqn:gamma_action_noncanonical_potential}:
\begin{align*}
	\frac{\phi(\gamma_1\cdots \gamma_n \omega_0)}{M(\gamma_1\cdots \gamma_n[\omega_0])} & =
	\frac{\phi\left(\gamma_1\cdot A\big( \gamma_2\cdots \gamma_n[\omega_0] \big)\right)}
	{M\left(\gamma_1\cdot \big( \gamma_2\cdots \gamma_n[\omega_0] \big)\right)} +
	\\
	& +
	\gamma_1\phi\left(\frac{\gamma_2\cdots \gamma_n[\omega_0]}{M(\gamma_2\cdots \gamma_n[\omega_0])}\right)
	\cdot
	\frac{M(\gamma_2\cdots \gamma_n[\omega_0])}{M(\gamma_1\cdots \gamma_n[\omega_0])}
\end{align*}
We consider the behavior of this expression as $\norm{\xi(\gamma_1)}_{\NS}\to +\infty$.

The second term goes to zero, since \autoref{prop:boundedness_of_potentials} says that the $C^0$-norm of $\phi\left(\frac{\gamma_2\cdots \gamma_n[\omega_0]}{M(\gamma_2\cdots \gamma_n[\omega_0])}\right)$ is uniformly bounded, it does not change when we apply $\gamma_1$, while the factor $\frac{M(\gamma_2\cdots \gamma_n[\omega_0])}{M(\gamma_1\cdots \gamma_n[\omega_0])}$ goes to zero using the estimates in \autoref{eqn:useful_inequalities}.

For the first term, to ease notation let $\xi_1:=\xi(\gamma_1)$.
Recall that as $R\to +\infty$ (i.e. $\norm{\xi_1}_{\NS}\to +\infty$) we have that $\frac{\xi_1}{\norm{\xi_1}_{\NS}}\to \xi$ and thus $\frac{\eta_{[E_0]}(\xi_1,\xi_1)}{\norm{\xi_1}^{2}_{\NS}}\to \eta_{[E_0]}(\xi,\xi)$, in the $C^0$-topology of potentials.
By \autoref{prop:parabolic_convergence_to_the_boundary}\ref{large_parabolic_brings_close_to_eta} we also have that
\[
	\frac{\phi\left(\gamma_1\cdot A\big( \gamma_2\cdots \gamma_n[\omega_0] \big)\right)}
	{M\left(\gamma_1\cdot \big( \gamma_2\cdots \gamma_n[\omega_0] \big)\right)}
	=
	\frac{\phi(\eta_{[E_0]}(\xi_1,\xi_1))}{M(\eta_{[E_0]}(\xi_1,\xi_1))}
	+
	O\left(\frac{1}{1+\norm{\gamma_1}_{\NS}}\right)
\]
The assumption of the proposition are verified since the cohomology class $\gamma_2\cdots \gamma_n[\omega_0]$ is parabolically normalized for the class $[E_0]$ by the geometric interpretation in \autoref{sssec:fundamental_domain_basepoint}.
It remains to observe that $M(\eta_{[E_0]}(\xi_1,\xi_1))=\norm{\xi_1}^2_{\NS}M([E_0])$ and analogously for $\xi$, and from this the convergence to the correctly normalized current follows.
In fact, we get a quantitative estimate depending on the distance between $\frac{\xi_1}{\norm{\xi_1}_{\NS}}$ and $\xi$, which goes to zero as $R\to +\infty$.
\hfill \qed

\subsubsection{Laminarity}
	\label{sssec:laminarity}
It follows from the work of Dujardin \cite{Dujardin2003_Laminar-currents-in-Bbb-P2} (see also de Th\'{e}lin \cite[Thm.~1]{Thelin2004_Sur-la-laminarite-de-certains-courants}) that the currents constructed in \autoref{thm:continuous_family_of_boundary_currents} are also weakly laminar, in the sense of \cite[Def.~2]{Thelin2004_Sur-la-laminarite-de-certains-courants}.
Indeed, all the currents can also be obtained as weak limits of $\frac{1}{m_n}(\gamma_n)_{*} T_C$, where $C$ is a fixed smooth ample curve, and $T_C$ is the current of integration on it, $\gamma_n$ is a sequence of automorphisms, and $m_n\to +\infty$ is a sequence of normalization factors.
By the results the aforementioned papers, it suffices to provide a bound on the genus and singularities of the curves $(\gamma_n)_* T_C$ in terms of $m_n$, but since the $\gamma_n$ is an automorphism the genus doesn't change and no singularities are introduced.

Specifically, fix an ample linear divisor $A$ on $X$ and let $\pi_{A}\colon X\dashrightarrow \bP^1$ be a rational projection determined by a $2$-dimensional linear subsystem.
We can select the $2$-dimensional subspace in a Baire-generic way such that the locus of indeterminacy of $\pi_A$ doesn't contain any point on any of the curves $\gamma_n(C)$.
By construction it follows that $\deg \pi_A\vert_{\gamma_n(C)}=[A].[\gamma_n(C)]=:d_n$.
By Riemann--Hurwitz the number of ramification points of $\pi_A\vert_{\gamma_n(C)}$ is equal to $2\left(g(C)-1+d_n\right)$, since the genus of $C$ and $C_n$ agree.
Now the required bounds in the proofs of Dujardin \cite[Thm.~3.1]{Dujardin2003_Laminar-currents-in-Bbb-P2} hold.
The proof as written in loc. cit. is for $\bP^2$, but for generalizing to the case of K3s it suffices to bound the number of ramification points by $O(d_n)$, as just established.



\subsection{Uniqueness results, after Sibony--Verbitsky}
	\label{ssec:uniqueness_results_after_verbitsky_sibony}

We quote here some results established by Sibony and Verbitsky.
While these results are not needed for the construction of currents in \autoref{thm:continuous_family_of_boundary_currents}, they do show that the resulting currents are unique and hence canonical.

\begin{theorem}[Uniqueness in irrational classes, Sibony--Verbitsky \cite{VerbitskySibony}]
	\label{thm:uniqueness_in_irrational_classes_verbitsky_sibony}
	Let $X$ satisfy the standing assumptions in \autoref{sssec:standing_assumptions_main}.
	Suppose that $[\alpha]\in \partial \Amp(X)$ is an irrational class, i.e. no nonzero rescaling of $[\alpha]$ is integral.
	Then the closed positive current in the class $[\alpha]$ is unique.
\end{theorem}
\noindent For completeness, we include the proof but we emphasize that the result is due to Sibony--Verbitsky.
\begin{proof}
	Let $\cP_{[\alpha]}$ denote the closed positive currents in the class $[\alpha]$.
	It is a metric space equipped with the $L^1$ metric:
	\[
		\dist(\alpha,\alpha'):=\int_X |\phi|\dVol \text{ where } \alpha=\alpha'+dd^c \phi \text{ and } \int_X \phi\dVol =0.
	\]
	Let $d([\alpha]):=\diam \cP_{[\alpha]}$ denote the diameter of this metric space.
	This is an upper semicontinuous function, since if $\alpha_i\to \alpha$ and $\alpha_i'\to \alpha'$ weakly in the sense of currents, and $[\alpha_i]=[\alpha_i']$ then $\dist([\alpha],[\alpha'])=\lim_i \dist([\alpha_i],[\alpha_i'])$ since we have $L^1$ convergence of the normalized difference of potentials (see e.g. \cite[Thm.4.1.8]{Hormander_Notions-of-convexity} for the local version of the statement).
	The diameter function also satisfies $d(\gamma[\alpha]) = d([\alpha])$ for any $\gamma\in \Aut(X)$ and is homogeneous of degree $1$, i.e. $d\left(e^{\lambda}[\alpha]\right)=e^{\lambda} d([\alpha])$ for any $\lambda\in \bR$.
	
	Take now a hyperbolic automorphism $\gamma$ and boundary class $[\alpha_+]$ such that $\gamma[\alpha_+]=e^\lambda [\alpha_+]$ with $\lambda>0$.
	From the relation
	\[
		d([\alpha_+])=d(\gamma[\alpha_+])=d\left(e^{\lambda}[\alpha_+]\right)
		= e^{\lambda}d([\alpha_+])
	\]
	it follows that $d([\alpha_+])=0$.
	
	Take now an arbitrary irrational class $[\alpha]\in \partial \Amp(X)$.
	By \autoref{thm:orbit_closure_dichotomy} below, its orbit under $\Aut(X)$ is dense in $\partial \Amp(X)$.
	So there exists a sequence of automorphisms $\gamma_i$ such that $\gamma_i[\alpha]\to [\alpha_+]$ and by upper semicontinuity of the diameter function, combined with the trivial bound $d([\alpha])\geq 0$, it follows that $d([\alpha])=0$.
\end{proof}

\begin{remark}[On the diameter function]
	\label{rmk:on_the_diameter_function}
	We think it is instructive to visualize the diameter function in the simplest situation, when $\Aut(X)$ is a lattice in $\SO_{1,2}(\bR)$.
	The action of $\Aut(X)$ on $\partial \Amp(X)$ in this case is isomorphic to the action of a lattice in $\SL_2(\bR)$ on $\bR^2\setminus 0$.
	There is a finite set of nonzero values corresponding to the diameters for the classes $[E]$ of elliptic fibrations (which fall finitely many $\Aut(X)$ equivalence classes).
	
	In the simplest case when the lattice is commensurable to $\SL_2(\bZ)$, the function $d\colon \bR^2\setminus 0\to \bR_{\geq 0}$ is equal to one of those constants at the primitive integral vectors and extended by homogeneity to the each rational ray, and equals zero on irrational rays.
	The function is upper semicontinuous since there are only finitely many primitive vectors in a disc of given radius.
	Note also that restricting to say a vertical line $x=1$, the diameter function becomes at rational points of the form $\frac{p}{q}\mapsto \frac{c(p/q)}{q}$ where $c(p/q)$ takes finitely many values, and the function vanishes at irrational points.
\end{remark}



\subsection{Boundary currents and Ricci-flat metrics}
	\label{ssec:boundary_currents_and_ricci_flat_metrics}

\subsubsection{Setup}
	\label{sssec:setup_boundary_currents_and_ricci_flat_metrics}
For each \Kahler class $[\omega]$ on a K3 surface, Yau's solution of the Calabi conjecture \cite{Yau_Ricci} ensures that there exists a unique Ricci-flat representative $[\omega]$.
The degeneration of these metrics, as the cohomology class $[\omega]$ approaches the boundary, has been the focus of intense study, see e.g. the second-named author's recent survey \cite{Tos_survey} and references therein.
We now interpret our earlier constructions and deduce some consequences for these degenerations, proving in particular \autoref{ricciflat}, which follows from the discussion in \autoref{sssec:irrational_boundary_classes} together with \autoref{thm:Linfty_bound_potentials} below.

Fix, for definiteness, a \Kahler class in $\NS(X)_{\bR}$ with a Ricci-flat metric $\omega_0$ inside.

\subsubsection{Irrational boundary classes}
	\label{sssec:irrational_boundary_classes}
Suppose that $[\alpha]\in \partial \Amp(X)$ is an irrational class.
Let also $\gamma_1,\gamma_2,\ldots$ be a sequence of generators as in \autoref{sssec:adapted_generators}, such that $\frac{1}{m_n}(\gamma_1)_*\cdots (\gamma_n)_{*}[\omega_0]\to [\alpha]$, where $m_n$ is an appropriate sequence of scalars with $m_n\to +\infty$.
From \autoref{thm:continuous_family_of_boundary_currents} it follows that the currents $\frac{1}{m_n}(\gamma_1\cdots \gamma_n)_*\omega_0$ converge in the $C^0$ sense of potentials to the unique closed positive current in the class $[\alpha]$.

At the same time, the \Kahler metrics $(\gamma_1\cdots \gamma_n)\omega_0$ and $\omega_0$ are isometric, since $\gamma_i$ are automorphisms.
Therefore, in the Gromov--Hausdorff sense, the \Kahler metrics $\frac{1}{m_n}(\gamma_1\cdots \gamma_n)_*\omega_0$ converge to a point, since $m_n\to +\infty$. This confirms \cite[Conjecture 3.25]{Tos_survey} for projective K3 surfaces with Picard number at least $3$ and no $(-2)$-curves, for classes on the boundary of the ample cone.

We thus see that for irrational boundary points, the Gromov--Hausdorff limit and weak limit as currents (or even stronger $C^0$-potentials limit) disagree. Of course, in this discussion the Ricci-flatness of $\omega_0$ did not play any role.

\subsubsection{Rational boundary points}
	\label{sssec:rational_boundary_points}
Suppose now that $[E]$ is the class of the fiber in a genus one fibration $X\xrightarrow{\pi}B$.
Then it is a famous result of Gross-Wilson \cite{GW} that, under the assumption that all singular fibers of $\pi$ are of type $I_1$, the Ricci-flat metrics in the class $t[\omega_0]+[E]$ converge as $t\to 0$ in the Gromov--Hausdorff sense to the metric completion of $(B^\circ,c\cdot \omega_B)$, where $B^\circ$ denotes the complement of the image of the singular fibers of $\pi$, $\omega_B:=\pi_*\dVol$ is the $(1,1)$-form on $B$ obtained by pushing forward the volume on $X$, and $c>0$ is a normalizing constant.
Furthermore, these Ricci-flat metrics converge locally smoothly on $X^\circ=\pi^{-1}(B^\circ)$ to $\pi^*\omega_B$, and the above metric completion is homeomorphic to $B$, has singularities at finitely many points, and on $B^\circ$ there is an affine structure that makes $\omega_B$ a ``Monge-Amp\`ere metric'' in the sense of Cheng-Yau \cite{ChYa}. Later a new proof of all these results was given in \cite{GTZ,GTZ2}, which also allows for arbitrary singular fibers, and this was extended to higher-dimensional hyperk\"ahler manifolds in \cite{TZ}.

It is therefore natural to compare $\omega_B$, which is canonically associated to the class $[E]$, and the currents coming from \autoref{thm:current_valued_pairing} via the dynamics of parabolic automorphisms.
With the same normalization, there is a $\bR\bP^{\rho-3}$ of currents coming from dynamics, and these are in general distinct as soon as $\rho\geq 4$, as follows from \cite[Thm.1.12]{DeMarcoMyrto-Mavraki2020_Elliptic-surfaces-and-arithmetic-equidistribution-for-R-divisors-on-curves}.
So it is not possible for the currents coming from dynamics to be equal to the current obtained by degenerating the Ricci-flat metric.
We suspect that even in the case $\rho=3$, the unique dynamically defined current does not, in general, equal the one coming from degenerations of Ricci-flat metrics.
See also \autoref{sssec:an_example_with_elliptic_fibrations} for an illustration of why the currents are, in general, different.

\subsubsection{Conical vs tangential approach to the boundary}
	\label{sssec:conical_vs_tangential_approach_to_the_boundary}
To put some of the discussion in this section into context, let us recall the notion of ``conical convergence'' to a boundary point $[\alpha]\in \partial\Amp^1(X)$, viewed as a point on the boundary of hyperbolic space.
Namely, $x_i\to[\alpha]$ conically, if there exists a geodesic ray $s$ starting at some point $p_0$ in the interior of hyperbolic space and going to $[\alpha]$, and a constant $C$, such that $\dist(x_i,s)\leq C$ and $\dist(p_0,x_i)\to+\infty$.
The construction in \autoref{thm:continuous_family_of_boundary_currents} uses a sequence of points converging conically to an irrational boundary point.
On the other hand, the currents associated to rational boundary points via \autoref{thm:current_valued_pairing}, are obtained from points in hyperbolic space that converge \emph{tangentially} to the rational boundary point, since they all lie on some fixed horosphere.
By contrast, the Gromov--Hausdorff collapse for elliptic fibrations happens along conical approaches to the rational boundary point.

\subsubsection{An application}
We now give another application of our main theorem \ref{thm:continuous_family_of_boundary_currents} to degenerations of Ricci-flat K\"ahler metrics. The setup is the following: $X$ is a K3 surface to which Theorem \ref{thm:continuous_family_of_boundary_currents} applies (so $X$ projective with Picard number at least $3$ and no $(-2)$-curves), with a fixed unit-volume Ricci-flat K\"ahler metric $\omega$ and $\alpha$ a closed real $(1,1)$-form whose class $[\alpha]\in \NS_{\mathbb{R}}$ is on the boundary of the ample cone and satisfies $\int_X\alpha^2=0$. As in \S \ref{sssec:rational_boundary_points}, for $t>0$ we let $\omega_t$ be the Ricci-flat K\"ahler metric on $X$ cohomologous to $\alpha+t\omega$, so that we can write $\omega_t=\alpha+t\omega+i\partial\overline{\partial}\vp_t$ where the smooth potentials $\vp_t$ are normalized by $\int_X\vp_t\omega^2=0$ and satisfy the complex Monge-Amp\`ere equation
\begin{equation}\label{ma_eqn}
\frac{\omega_t^2}{\int_X(\alpha+t\omega)^2}=\omega^2.
\end{equation}

\begin{theorem}
	\label{thm:Linfty_bound_potentials}
Under these assumptions, we have the uniform $L^\infty$ estimate
\begin{equation}\label{goal}
	\sup_X|\vp_t|\leq C,
\end{equation}
for all small $t>0$.
\end{theorem}
\noindent Note that while $C^0$ convergence of $\vp_t$ also seems plausible, the geometric arguments using $\Aut(X)$ can only establish it only along subsequences of $t$ which project to a compact set modulo $\Aut(X)$.
\begin{proof}
This result follows from \autoref{thm:continuous_family_of_boundary_currents} by applying \cite[Prop.1.1]{FGS}, which uses sophisticated pluripotential theory (and holds in all dimensions). In our setting we can give a simple proof using a Moser iteration type argument as in the classical work of Yau \cite{Yau_Ricci}, as follows.

First, observe that
\begin{equation}\label{volbd}
\int_X (\alpha+t\omega)^2\leq Ct,
\end{equation}
as $t\to 0$. Thanks to \autoref{thm:continuous_family_of_boundary_currents} we can find a closed positive current $\eta=\alpha+i\partial\overline{\partial} \vp_0\geq 0$ with $\vp_0\in C^0(X)$ normalized by $\int_X\vp_0\omega^2=0$. From $\int_X\vp_t\omega^2=0$ we see that $\sup_X\vp_t\geq 0$, and since
$$0<\omega_t=\alpha+t\omega+i\partial\overline{\partial}\vp_t\leq C_0\omega+i\partial\overline{\partial}\vp_t,$$
for a uniform $C_0$ independent of $t$ (small), and therefore
$$\Delta_{\omega}\vp_t\geq -2C_0.$$
It then follows from the Green's formula for $\omega$ that $\sup_X\vp_t\leq C$ for $C$ also independent of $t$. Indeed, picking a point $x\in X$ where $\vp_t(x)=\sup_X\vp_t$, we have
$$\vp_t(x)=-\int_X \Delta_{\omega}\vp_t(y) G(x,y) \omega^2\leq 2C_0\int_XG(x,y) \omega^2\leq C,$$
where $G(x,y)$ is the nonnegative Green's function for $\omega$, which is in $L^1(X)$. The upshot is that
\begin{equation}\label{stupid}
|\sup_X\vp_t|\leq C,
\end{equation}
uniformly as $t\to 0$.

Next applying the regularization theorems in \cite[Cor. 2]{GWu} or \cite{Dem_L2} to $\vp_0$ (basically given by convolution with a mollifier in local charts, plus a patching argument, see also \cite{BK}), we can find smooth functions $\psi_t,t>0$ with $\int_X\psi_t\omega^2=0$ such that $\psi_t\to\vp_0$ uniformly on $X$ as $t\to 0$ and
$$\alpha+i\partial\overline{\partial}\psi_t\geq -\frac{t}{2}\omega,$$
on $X$. 
We then define
\begin{equation}\label{lowbd}
\hat{\omega}_t=\alpha+t\omega+i\partial\overline{\partial}\psi_t\geq \frac{t}{2}\omega,
\end{equation}
so these are K\"ahler forms cohomologous to $\omega_t$, and letting
$$\hat{\vp}_t=\vp_t-\psi_t-\sup_X(\vp_t-\psi_t)-1,$$ we can then write
$$\omega_t=\hat{\omega}_t+i\partial\overline{\partial}\hat{\vp}_t,$$
and $\sup_X\hat{\vp}_t=-1.$ Since $\psi_t$ converges uniformly to $\vp_0$, its $L^\infty$ norm is uniformly bounded as $t\to 0$, so to prove \eqref{goal} it suffices to obtain a uniform $L^\infty$ bound for $\vp_t-\psi_t$, and recalling \eqref{stupid} we can equivalently derive a uniform $L^\infty$ bound for $\hat{\vp}_t$ instead.

For this we employ the Moser iteration method. For $p\geq 2$ compute using \eqref{ma_eqn}, \eqref{volbd} and $\hat{\vp}_t\leq -1$
$$\int_X (-\hat{\vp}_t)^{p-1}(\omega_t^2-\hat{\omega}_t^2)\leq \int_X (-\hat{\vp}_t)^{p-1}\omega_t^2\leq Ct \int_X (-\hat{\vp}_t)^{p-1}\omega^2\leq Ct \int_X (-\hat{\vp}_t)^{p}\omega^2,$$
while also
\[\begin{split}
\int_X (-\hat{\vp}_t)^{p-1}(\omega_t^2-\hat{\omega}_t^2)&=\int_X (-\hat{\vp}_t)^{p-1}i\partial\overline{\partial}\hat{\vp}_t\wedge(\omega_t+\hat{\omega}_t)\\
&=(p-1)\int_X(-\hat{\vp}_t)^{p-2}i\partial\hat{\vp}_t\wedge\overline{\partial}\hat{\vp}_t\wedge(\omega_t+\hat{\omega}_t)\\
&\geq (p-1)\int_X(-\hat{\vp}_t)^{p-2}i\partial\hat{\vp}_t\wedge\overline{\partial}\hat{\vp}_t\wedge \hat{\omega}_t\\
&\geq \frac{(p-1)t}{2}\int_X(-\hat{\vp}_t)^{p-2}i\partial\hat{\vp}_t\wedge\overline{\partial}\hat{\vp}_t\wedge\omega\\
&=\frac{(p-1)t}{p^2}\int_X\left|\partial\left((-\hat{\vp}_t)^{\frac{p}{2}}\right)\right|^2_{g}\omega^2,
\end{split}\]
using \eqref{lowbd}, and so
$$\int_X\left|\partial\left((-\hat{\vp}_t)^{\frac{p}{2}}\right)\right|^2_{g}\omega^2\leq Cp\int_X (-\hat{\vp}_t)^{p}\omega^2,$$
for all $p\geq 2$ (and $C$ independent of $p$) and using the Sobolev inequality this gives
$$\left(\int_X (-\hat{\vp}_t)^{2p}\omega^2\right)^{\frac{1}{2}}\leq C\left(\int_X\left|\partial\left((-\hat{\vp}_t)^{\frac{p}{2}}\right)\right|^2_{g}\omega^2+\int_X (-\hat{\vp}_t)^{p}\omega^2\right)
\leq Cp\int_X (-\hat{\vp}_t)^{p}\omega^2,$$
i.e.
$$\|\hat{\vp}_t\|_{L^{2p}(X)}\leq C^{\frac{1}{p}}p^{\frac{1}{p}} \|\hat{\vp}_t\|_{L^p(X)},$$
and substituting $p\mapsto 2p$ and iterating this inequality, a standard Moser iteration argument (taking in the end $p=2$) gives
$$\sup_X (-\hat{\vp}_t)\leq C\left(\int_X (-\hat{\vp}_t)^{2}\omega^2\right)^{\frac{1}{2}}\leq C(\sup_X (-\hat{\vp}_t))^\frac{1}{2}\left(\int_X (-\hat{\vp}_t)\omega^2\right)^{\frac{1}{2}},$$
and so
$$\sup_X (-\hat{\vp}_t)\leq C\int_X (-\hat{\vp}_t)\omega^2=\sup_X(\vp_t-\psi_t)+1\leq C,$$
by \eqref{stupid} (and again the uniform boundedness of $\psi_t$). This proves the uniform $L^\infty$ bound for $\hat{\vp}_t$, and therefore also for $\vp_t$.
\end{proof}

\subsubsection{An example with elliptic fibrations}
	\label{sssec:an_example_with_elliptic_fibrations}
Let $F_0,F_1\in \bC[X_0,X_1,X_2]$ be two polynomials of homogeneous degree $3$, determining two distinct smooth cubics in $\bP^2$ intersecting at $9$ points.
They also a determine a pencil of cubics $C_t:=\{aF_0+bF_1=0\}$ and we assume additionally that the cubics are in general position so that the singular elements of the pencil are nodal.
Let $\wtilde{\bP}^2$ be the blowup of $\bP^2$ at the base locus of the pencil (the $9$ points) so we have a genus one fibration
\[
	\wtilde{\bP}^2 \xrightarrow{\pi'} \bP^1
\]
which in addition has $9$ sections $\sigma_i'\colon \bP^1\to \wtilde{\bP}^2$ corresponding to the exceptional divisors of the blowup.
By selecting one of them as a basepoint, we obtain an abelian group of parabolic automorphisms of rank $8$.
The discussion in \autoref{sec:elliptic_fibrations_and_currents} applies (since only nodal singular fibers occur) and associates to each parabolic automorphism a corresponding current on $\bP^1$.

In order to build a K3 surface, let us choose two distinct points $t_1,t_2$ in the pencil, with corresponding smooth cubics $C_{t_1},C_{t_2}$.
The double cover $X_{t_1,t_2}\to \bP^2$ ramified over $C_{t_1}\cup C_{t_2}$ is a K3 surface, after we blow up the $9$ intersection points of the cubics.
In fact we have the morphisms
\begin{equation}
	\label{eqn_cd:elliptic_fibration_double_plane}
\begin{tikzcd}
	X
	\arrow[r, "c'"]
	\arrow[d, "\pi"]
	&
	\wtilde{\bP}^2
	\arrow[d, "\pi'"]
	\\
	\bP^1
	\arrow[r, "2:1"', "c"]
	&
	\bP^1
\end{tikzcd}
\end{equation}
such that $\pi$ is the genus one fibration, and the map $\bP^1\xrightarrow{c}\bP^1 $ is a two-to-one cover ramified over $t_1,t_2$.
The genus one fibration\footnote{This example does not fit our standing assumptions on K3s, since the fibrations have sections and hence the K3 has $(-2)$ curves. But the formalism in \autoref{sec:elliptic_fibrations_and_currents} does apply.} $X\xrightarrow{\pi}\bP^1$ also has (at least) $9$ sections $\sigma_i'$, one for each intersection point of the cubics.
There is correspondingly a rank $8$ abelian group of parabolic automorphisms of $X$, and the diagram in \autoref{eqn_cd:elliptic_fibration_double_plane} commutes for the action of $\bZ^8$ on $X$ and $\wtilde{\bP}^2$.

The dynamical currents on $\bP^1$ associated to parabolic automorphisms on $X$ are therefore pulled back from those on the right side of the diagram via $c^*$, and consequently if we push them back via $c_*$ we get twice the original currents.
In particular, these pushed-forward currents are independent of the parameters $t_1,t_2$ that we used to make the ramified double cover.

On the other hand, we can compute $c_*\pi_*\dVol_X$ by viewing it alternatively as $\pi'_*c'_*\dVol_X$.
To compute this explicitly, let us work in an affine chart where the degree $3$ polynomials are $f_0(x_1,x_2)$ and $f_1(x_1,x_2)$ and $X$ is given by $y^2=(f_0+af_1)(f_0+bf_1)$.
Then the holomorphic $2$-form on $X$ is given by $\Omega:=\frac{dx_1\wedge dx_2}{y}$ and the volume form is $\dVol_X=\frac{|dx_1\wedge dx_2|}{|y|^2}$.
This is invariant by the involution $y\mapsto -y$ and pushes down to $\bA^2$ to become
\[
	\dVol_{\wtilde{\bP}^2}(a,b) = \frac{|dx_1\wedge dx_2|^2}{|(f_0 + af_1)(f_0+bf_1)|}
\]
where $(a,b)$ on the left are introduced to note that the volume form does depend on the choice of these parameters.
We now observe that this volume form, as well as its pushforward to $\bP^1$ via the map $(x_1,x_2)\mapsto\frac{f_0(x_1,x_2)}{f_1(x_1,x_t)}=:t$, do depend on the parameters.
Indeed, we have
\begin{align*}
	\frac{\dVol_{\wtilde{\bP}^2}(a,b)}{\dVol_{\wtilde{\bP}^2}(a',b')}
	& =
	\frac{|(f_0 + a'f_1)(f_0+b'f_1)|}{|(f_0 + af_1)(f_0+bf_1)|}\\
	& = \frac{\left|\left(\tfrac{f_0}{f_1} + a'\right) \left(\tfrac{f_0}{f_1} + b'\right) \right|}{\left| (\tfrac{f_0}{f_1} + a) (\tfrac{f_0}{f_1} + b)\right|}\\
	& = \frac{|t+a'||t+b'|}{|t+a||t+b|}
\end{align*}
So the proportionality factor is a nontrivial function pulled back from $\bP^1$, therefore distinct choices of cubics for ramification lead to distinct pushed-forward volume forms.




\section{Elliptic fibrations and heights}
	\label{sec:elliptic_fibrations_and_heights}

\subsubsection*{Outline of section}

In this section we establish some preliminary properties of heights on surfaces with elliptic fibrations.
The main result of this section is \autoref{thm:preferred_heights_on_elliptic_fibrations} that establishes the existence of preferred height functions on an elliptic surface, associated to its parabolic automorphisms.
Our basic conventions are spelled out in \autoref{sssec:setup_the_pairing_from_the_variation_of_canonical_height}.

In \autoref{ssec:the_pairing_from_the_variation_of_canonical_height} we recall classical results of Silverman and Tate and formulate them in our setting of genus one fibrations without a section.
The dynamical constructions in \autoref{ssec:relative_dynamics_in_the_elliptic_fibration} are then used in \autoref{ssec:preferred_heights_on_elliptic_fibrations} to construct the preferred height functions.


\subsection{The pairing from the variation of canonical height}
	\label{ssec:the_pairing_from_the_variation_of_canonical_height}

\subsubsection{Setup}
	\label{sssec:setup_the_pairing_from_the_variation_of_canonical_height}
Fix a number field $k$ and an algebraic closure $\ov{k}\supset k$.
For all constructions, we allow implicitly to pass to a finite extension of $k$, so that various geometric objects are always defined over $k$.
In particular, all morphisms are over $k$.
Throughout $X$ and $B$ will be a smooth projective surface, respectively curve, defined over the number field $k$.
The notation follows closely that in \autoref{sec:elliptic_fibrations_and_currents}.
The morphism $X\xrightarrow{\pi}B$, also defined over $k$, is a fibration in genus one curves and we set $X^\circ\xrightarrow{\pi}B^{\circ}$ to be the morphism restricted to the open locus where the fibration is smooth.
We denote by $k(B)$ the function field of $B$ and by $X_{k(B)}$ the elliptic curve over $k(B)$ obtained from $X$ (i.e. the restriction of $X$ to the generic point of $B$).

\subsubsection{Standing assumptions}
	\label{sssec:standing_assumptions_elliptic_heights}
For this section only, we relax our earlier assumptions that all singular fibers are reduced and irreducible, and consider a slightly more general case.
We assume that $X$ is smooth and relatively minimal, i.e. there are no $(-1)$ curves in the fibers of the morphism to $B$, and that the fibration is not isotrivial (this implies that it has at least one singular fiber, see e.g. \cite[Theorem III.15.4]{BHPV}).

We additionally assume that for each singular fiber, at least one irreducible component has multiplicity one (other components can have higher multiplicity).
In the analytic topology near the singular fiber, this is equivalent to the fibration having locally a section.
This assumption is needed to avoid some torsion phenomena, and is implicit in the literature that considers elliptic curves, since those have a section.
Note also that this assumption is always satisfied when $X$ is a K3, by the classification of singular fibers of elliptic fibrations on K3s \cite[Theorem 11.1.9]{Huy}.

\subsubsection{Vertical divisors}
	\label{sssec:vertical_divisors}
A Weil or Cartier divisor on $X$, will be called \emph{vertical} if it is vertical in the sense of \cite[Definition~3.5, \S8.3, pg. 349]{Liu2002_Algebraic-geometry-and-arithmetic-curves}.
See also \cite[III.\S8]{Silverman1994_Advanced-topics-in-the-arithmetic-of-elliptic-curves} where such divisors are called \emph{fibral}.
Note that all fibers of the elliptic fibration have arithmetic genus one, so when a fiber is not smooth, its irreducible components must have geometric genus $0$.

\subsubsection{Jacobian surface}
	\label{sssec:jacobian_surface}
It will be technically convenient to refer to the Jacobian surface $J_X\to B$.
See \cite[\S11.4.1]{Huy} which describes the construction (the fact that the original $X$ is a K3 and $B\isom \bP^1$ is not used; it is the case, however, that if $X$ is a K3 then $J_X$ is also a K3).

One definition of $J_X$ is as follows.
View $X$ as a genus one curve over the function field $k(B)$.
It has an associated Jacobian elliptic curve $J_{X,k(B)}$ and take $J_X$ to be the surface over $k$ which is relatively minimal over $B$.

We now proceed to study the Picard group of $X$ and some of its subgroups.
\begin{proposition}
	\label{prop:injectivity_Pic0}
Let $\pi\colon X\to B$ be a non-isotrivial relatively minimal elliptic surface over $\mathbb{C}$ (we do not assume $X$ projective or that $\pi$ admits a section).
Then $\pi^*\colon\Pic^0(B)\to \Pic^0(X)$ is an isomorphism.
\end{proposition}
\begin{proof}
As in \cite[III.4.1]{Miranda}, the Leray spectral sequence gives the exact sequence
\[0\to H^1(B,\pi_*\mathcal{O}_X)\to H^1(X,\mathcal{O}_X)\to H^0(B, R^1 \pi_* \mathcal{O}_X)\to H^2(B,\pi_* \mathcal{O}_X),\]
but $\pi_*\mathcal{O}_X\cong\mathcal{O}_B$ since $\pi$ has connected fibers, hence
\[0\to H^1(B,\mathcal{O}_B)\to H^1(X,\mathcal{O}_X)\to H^0(B, R^1 \pi_* \mathcal{O}_X)\to 0,\]
is exact, and the dimensions of the first two terms are $g(B)$ and $q(X)$ respectively. As for the third term, \cite[Theorem III.18.2]{BHPV} (see also their Remark after Thm V.12.1 on p.213) shows that the degree of the dual of $R^1 \pi_* \mathcal{O}_X$ is strictly positive, unless all smooth fibers of $\pi$ are isomorphic (i.e. $\pi$ is isotrivial) and the only singular fibers are multiples of a smooth fiber.  So by our assumption it follows that the degree of the line bundle $R^1 \pi_* \mathcal{O}_X$ is strictly negative, so it has no sections, and the exact sequence above gives $q(X)=g(B)$.

We can now conclude as in \cite[Lemma VII.1.1]{Miranda}: first, the map $\pi:\Pic(B)\to\Pic(X)$ is injective since the projection formula gives $\pi_*\pi^*L\cong L\otimes\pi_*\mathcal{O}_X\cong L$ for every $L\in \Pic(B)$. And then since $q(X)=g(B)$, it follows that the tori $\Pic^0(B)$ and $\Pic^0(X)$ have the same dimension, hence $\pi^*:\Pic^0(B)\to\Pic^0(X)$ is an isomorphism.
\end{proof}

\subsubsection{Line bundles and automorphisms}
	\label{sssec:line_bundles_and_automorphisms}
Let $\Aut_\pi(X)$ be the group of parabolic automorphisms of the elliptic surface $X\xrightarrow{\pi}B$, i.e. those preserving the fibers of $\pi$.
Denote by $\Aut_\pi^\circ(X)$ those automorphisms that act trivially on the set of irreducible components of fibers of $\pi$, and such that the induced automorphism of the genus one curve $X_{k(B)}$ over the function field $k(B)$ is given by a translation.
Then $\Aut_{\pi}^{\circ}(X)$ is of finite index in $\Aut_{\pi}(X)$, since there are only finitely many fibers with distinct irreducible components, and because the translation automorphisms are of finite index in the group of all automorphisms of an elliptic curve (by composing by a translation, we can assume a point is fixed, and the group of automorphisms preserving a polarization is finite).
In the case when all singular fibers of $\pi$ are of type I (i.e a nodal $\bP^1$), the group $\Aut_\pi^\circ(X)$ defined here coincides with the one in \autoref{sec:elliptic_fibrations_and_currents} (after choosing an embedding $k\into \bC$ and changing base).

To handle more general elliptic fibrations compared to \autoref{sec:elliptic_fibrations_and_currents}, we will replace the group $[E]^\perp/[E]$ from previous sections by a different one.
Specifically, let $\Pic^0_{\pi}(X)$ denote the subgroup of line bundles on $X$ which restrict to have degree $0$ on each irreducible component of a fiber of the fibration.
Consider also the larger group $\Pic^0_{\pi}(X)^{sm}$ of line bundles that have degree $0$ on the generic smooth fiber.
Finally let $\pi^*\Pic(B)\subset \Pic(X)$ be the group of line bundles pulled back from $B$.
We have the filtration
\begin{align}
	\label{eqn:filtration_general_pic}
	\Pic(X)\supset
	\Pic^0_\pi(X)^{sm}\supset
	\Pic^0_\pi(X)\supset
	\pi^* \Pic(B).
\end{align}
Define now $\Pic^{rel}_\pi(X):=\Pic^0_{\pi}(X)/\pi^* \Pic(B)$ which will be the analogue of $[E]^\perp/[E]$.
Note that this is a discrete and finitely generated group by \autoref{prop:injectivity_Pic0}.
We will now consider the action of parabolic automorphisms on \autoref{eqn:filtration_general_pic}.

\begin{proposition}[Parabolic automorphisms and the filtration]
	\label{prop:parabolic_automorphisms_respect_the_filtration}
	Suppose that $\gamma\in \Aut^\circ_\pi(X)$ is a parabolic automorphism.
	\begin{enumerate}
		\item \label{item:parb_aut_to_PicB}
		If $L^0\in \Pic^0_{\pi}(X)^{sm}$, then we have that $\gamma^*L^0-L^0\in \pi^*\Pic(B)$.
		\item \label{item:par_aut_on_PicX}
		If $L\in \Pic(X)$ then $\gamma^*L-L\in \Pic^0_\pi(X)$.
	\end{enumerate}
\end{proposition}
\begin{proof}
	For part (i), note that adding a vertical divisor to $L^0$ doesn't affect the desired conclusion, and by \cite[III.Prop.8.3]{Silverman1994_Advanced-topics-in-the-arithmetic-of-elliptic-curves} we can adjust $L^0$ by vertical divisors to assume that it is in fact in $\Pic^0_\pi(X)$, i.e. it has degree $0$ on each irreducible component of each vertical divisor.

	For each fiber $X_b$ we have the Jacobian variety of degree $0$ line bundles $\Jac(X_b)$.
	The restriction of the parabolic automorphism $\gamma$ to each fiber induces a group automorphism of $\Jac(X_b)$.
	In the case of smooth fibers, and also the generic fiber, the induced automorphism on $\Jac(X_{k(B)})$ is the identity by the definition of $\Aut^\circ_{\pi}(X)$ in \autoref{sssec:line_bundles_and_automorphisms}.
	Therefore over the open locus $X^{\circ}\to B^{\circ}$ where the fibration is smooth, the line bundle is pulled back from the base.
	So there exists some line bundle $L_B$ on $B$ (which we identify with a Cartier divisors) such that $\gamma^*L^0-L^0 - \pi^*L_B$ is linearly equivalent to a Cartier divisor supported on the singular vertical fibers (see also \autoref{sssec:vertical_divisors}).
	However, this divisor must have zero intersection number with any irreducible component of any fiber, since this holds for $\gamma^*L^0 - L^0$.
	It follows that this integral divisor is a rational multiple of full singular fibers (recall that some singular fibers could have multiplicity) by \cite[III.8.2(11)]{BHPV}.
	However, since we assumed that each singular fiber has at least one component of multiplicity $1$, it follows that the rational number is an integer and so the divisor is pulled back from $B$.

	Part (ii) follows because the restriction of $L$ and $\gamma^*L$ to each irreducible component have the same degree, since $\gamma$ acts trivially on the set of irreducible components.
\end{proof}

\subsubsection{Assigning line bundles to automorphisms}
	\label{sssec:assigning_line_bundles_to_automorphisms}
Returning to \autoref{eqn:filtration_general_pic}, the quotient $\Pic(X)/\Pic_\pi^0(X)^{sm}$ is a cyclic subgroup of $\bZ$.
There is also a canonical choice among the two possible generators, given by the one whose multiple is in the image of an ample class from $\Pic(X)$.
Fix therefore $L\in \Pic(X)$ to be an ample line bundle projecting to the generator.
For an element $\gamma\in \Aut^\circ_\pi(X)$ we then set
\[
	\xi(\gamma):=\gamma^* L - L \in \Pic^{rel}_\pi(X).
\]
where by \autoref{prop:parabolic_automorphisms_respect_the_filtration} \ref{item:par_aut_on_PicX} the line bundle $\gamma^*L-L$ belongs to $\Pic^{0}_\pi(X)$, and we further consider its image under the quotient by $\pi^*\Pic(B)$.

\begin{proposition}[Homomorphism to relative Picard]
	\label{prop:homomorphism_to_relative_picard}
	The map $\xi$ in \autoref{sssec:assigning_line_bundles_to_automorphisms} is independent of the choice of representative line bundle $L$ mapping to the generator of $\Pic(X)/\Pic^0_{\pi}(X)^{sm}$ and gives a group homomorphism:
		\[
			\xi\colon \Aut^{\circ}_\pi(X) \to \Pic^{rel}_\pi(X).
		\]

\end{proposition}
\begin{proof}
	Suppose that $L'$ is another such line bundle.
	Then $L-L'\in \Pic^0_\pi(X)^{sm}$ and so by \autoref{prop:parabolic_automorphisms_respect_the_filtration}\ref{item:parb_aut_to_PicB} it follows that $\gamma^*(L-L')-(L-L')\in \pi^* \Pic(B)$ so the map $\xi$ is independent of the choice of $L$.
	To check the group homomorphism property, we must show that for two parabolic automorphisms $\gamma_1,\gamma_2$, the line bundles
	\[
		(\gamma_1^*L-L) + (\gamma_2^*L-L) \quad \text{ and }\quad
		(\gamma_1^*\gamma_2^*)L - L
	\]
	give the same element after quotienting by $\pi^* \Pic (B)$.
	We rewrite the second one as
	\[
		(\gamma_1^*\gamma_2^*)L - L = \gamma_1^*(\gamma_2^*L - L) + (\gamma_1^*L-L)
	\]
	and observe that since $\gamma_2^*L-L$ is an element of $\Pic^0_\pi(X)$, by \autoref{prop:parabolic_automorphisms_respect_the_filtration}\ref{item:parb_aut_to_PicB} its image under $\gamma_1^*$ is equal to itself, modulo an element of $\pi^*\Pic(B)$.
\end{proof}

\subsubsection{A pairing valued in line bundles}
	\label{sssec:a_pairing_valued_in_line_bundles}
The next construction appears in \cite[III.Thm.9.5]{Silverman1994_Advanced-topics-in-the-arithmetic-of-elliptic-curves} in a slightly different notation.
We have a bilinear map
\[
	\Aut^\circ_\pi(X)\times \Pic^{rel}_\pi(X) \xrightarrow{\eta_{\Pic}} \Pic(B)
\]
given by
\[
	\pi^*\eta_{\Pic}(\gamma,L^0)= \gamma^* L^0 - L^0  \quad \text{ as line bundles on }X.
\]
In other words, the line bundle $\gamma^* L^0 - L^0$ is the pullback of a line bundle $\eta_{\Pic}(\gamma,L^0)$ from $B$.
Furthermore, if we change $L^0$ by an element of $\pi^* \Pic(B)$ the resulting line bundle on $B$ doesn't change.

\subsubsection{Recollections on heights}
	\label{sssec:recollections_on_heights}
We work with projective varieties which are not necessarily smooth.
Height functions are naturally associated to elements of the Picard group, i.e. line bundles.
For instance Serre \cite[\S2.8]{Serre_Lectures-on-the-Mordell-Weil-theorem} has a discussion of heights and their functorial properties for not necessarily smooth projective varieties and line bundles.

For a projective variety $X$, denote by $\Pic(X)$ its Picard group and by $\cH(X)$ the (additive) group of height functions on $X(\ov{k})$, so that we have a natural extension
\[
	0 \to L^\infty\left(X(\ov{k});\bR\right)\into \cH(X) \onto \Pic(X)\to 0
\]
Our interest will be to find preferred choices of height functions for various subgroups of the Picard group.

\subsubsection{Variation of canonical height}
	\label{sssec:variation_of_canonical_height}
We recall briefly the theory developed by Silverman \cite{Silverman1994_Variation-of-the-canonical-height-in-algebraic-families,Silverman1994_Variation-of-the-canonical-height-on-elliptic-surfaces.-III.-Global,Silverman1994_Variation-of-the-canonical-height-on-elliptic-surfaces.-II.-Local-analyticity,Silverman1992_Variation-of-the-canonical-height-on-elliptic-surfaces.-I.-Three-examples}
 and Tate \cite{Tate} for the variation of canonical height in a family.
Assuming that the elliptic fibration has a section $\sigma\colon B\to X$, there are canonical \Neron--Tate heights on individual fibers $X_b$ and denoted $h^{can}_{X_b,\sigma}$.
Then the basic result is that for another section $\sigma'$ we have that
\[
	h^{can}_{X_b,\sigma}(\sigma'(b)) =: h^{vcan}_{\sigma,\sigma'}(b)
\]
is a height function on $B$, which is associated to the line bundle $\eta_{\Pic}(\sigma'-\sigma,\sigma'-\sigma)$.
Note that the difference of the two sections induces a parabolic automorphism of the surface, and is also a line bundle on the surface which has degree $0$ on the vertical fiber.

\subsubsection{Variation of canonical height as a pairing}
	\label{sssec:variation_of_canonical_height_as_a_pairing}
The original statements of Tate and Silverman are concerned with elliptic fibrations with a section.
But this assumption can be relaxed to view the pairing as one between parabolic automorphisms and relative degree $0$ line bundles, without the need to introduce a section.
Indeed, fix a $\bQ$-line bundle $L$ which has degree $1$ when restricted to a generic smooth fiber of the fibration (this is the replacement of the section).
Fix some basepoint $x_b\in X_b$ in each fiber, so as to obtain a canonical height $h^{can}_{L\vert_{X_b},x_b}$.
Note that the canonical height thus defined depends on the choice of $x_b$ up to a constant.
Then we simply define
\begin{align}
	\label{eqn:vcan_definition}
	h^{vcan}_{\gamma,\xi(\gamma)} (b) :=
	\lim_{n\to +\infty}
	\frac{1}{n^2} h^{can}_{L\vert_{X_b},x_b}(\gamma^n x_b)
\end{align}
Note that this limit is independent of $x_b$ (we could independently replace each occurrence of $x_b$ by another $x_b'$ in the above limit without changing it).
Furthermore, this definition feeds directly into the telescoping argument of Tate to establish the existence of the limit and its properties.

Alternatively, we can work on the Jacobian surface $J_X\to B$.
The results of Silverman and Tate apply to it directly, and we do have fiberwise isomorphisms $J_{X,b}\isom X_{b}$ after some choice of basepoint in each fiber.
In particular the definition in \autoref{eqn:vcan_definition} of $h^{vcan}$ on $X$ translates directly to the Silverman--Tate definition on $J_X$.

The next result is the analogue of the current-valued pairing constructed in \autoref{thm:current_valued_pairing} but in the setting of heights.
It follows from the previous discussion by ``polarizing'' the variation of canonical height for a single parabolic automorphism to a bilinear pairing between parabolic automorphisms.
For the proof, see \cite[III.Thm.11.1]{Silverman1994_Advanced-topics-in-the-arithmetic-of-elliptic-curves}.

\begin{theorem}[Height-valued pairing]
	\label{thm:height_valued_pairing}
	There exists a map
	\[
		\eta_h\colon
		\Aut^\circ_\pi(X)
		\times
		\Pic^{rel}_\pi(X)
		\to
		\cH(B)
	\]
	which is linear in each coordinate and with the following properties.
	\begin{enumerate}
		\item The map is compatible with the corresponding pairing on line bundles, namely $\eta_h(\gamma,L^0)$ is a height on $B$ associated to the line bundle $\eta_{\Pic}(\gamma, L^0)$.
		\item For an automorphism $\gamma\in \Aut^\circ_\pi(X)$ we have that $\eta_h(\gamma,\xi(\gamma))$ agrees with Silverman's variation of canonical height on $B$ and
		\[
			\eta_h(\gamma,\xi(\gamma)) \geq 0 \quad \text{ as a function on }B(\ov{k}).
		\]

		\item The map is symmetric, i.e. if $\xi_i=\xi(\gamma_i)$
		\[
			\eta_h(\gamma_1,\xi_2) = \eta_h(\gamma_2,\xi_1)
		\]
	\end{enumerate}
\end{theorem}

\noindent Essential for our construction will be the analogue of the existence, for a $(1,1)$-form $\alpha$ of integral $0$ on fibers, of a potential $\phi$ which makes it trivial on fibers, and in particular satisfies the identity:
	\[
		\gamma_{*}(\alpha + dd^c\phi)
		=
		(\alpha + dd^c \phi) + \pi^{*}\eta_h(\gamma,[\alpha])
	\]
as in \autoref{thm:preferred_currents}.
This analogue is the content of \autoref{thm:preferred_heights_on_elliptic_fibrations}.



\subsection{Relative dynamics in the elliptic fibration}
	\label{ssec:relative_dynamics_in_the_elliptic_fibration}

\subsubsection{Setup}
	\label{sssec:setup_relative_dynamics_in_the_elliptic_fibration}
For a smooth genus one curve $E$ over a field, we denote by $\sim$ or $\sim_E$ linear equivalence of divisors.
This will apply to both $X$ viewed as a genus one curve over the function field $k(B)$ and to the fibers $X_b$ for $b\in B(\ov{k})$ viewed as genus one curves over $k(b)$ (the residue field at $b$).

Since our genus one curves do not have a preferred basepoint, we will distinguish between addition for a group law and formal addition of divisors.
For a divisor $p^0$ of degree $0$ on $E$, and for any divisor $q$ on $E$ of degree $1$, the degree $1$ divisors $q\oplus p^0$ and $q\ominus p^0$ are well-defined and equal to the unique degree $1$ divisors linearly equivalent to $q\pm p^0$ where $\pm$ is used now as formal addition of divisors.
Throughout, the symbols $+$ and $-$ for divisors will mean formal addition of divisors, whereas $\oplus$ and $\ominus$ will mean translation by corresponding degree $0$ divisors.

\subsubsection{Hyperbolic dynamics}
	\label{sssec:hyperbolic_dynamics_number_field}
Let $D\subset X$ be an effective divisor, with no vertical components (see \autoref{sssec:vertical_divisors}), and such that $[D].[X_b]=d$ i.e. it intersects the fibers in degree $d\geq 1$.
Associated to $D$ we have the ``multiplication by $(d+1)$ map'' in the fibers, defined on the open set of smooth fibers $X^\circ$ as follows:
\begin{align}
	\begin{split}
	\label{eqn:delta_circ_def}
	\delta_D^\circ \colon X^{\circ} &\to X^{\circ}\\
	\delta_D^\circ(x) & \sim_{X_b} (d+1)\cdot x - (D\cap X_b) \qquad b = \pi(x)
	\end{split}
\end{align}
In other words $\delta_D^\circ(x)$ is the unique point in the same fiber which is linearly equivalent to the degree $1$ divisor $(d+1)\cdot x - (D\cap X_b)$.
Note that if $D\cap X_b\sim_{X_b} d\cdot r_D(b)$, where $r_D(b)\in X_b(\ov{k})$ is well-defined up to $d$-torsion, then the map can also be expressed as $\delta_D^\circ(x)=x \oplus d(x-r_D(b))$, i.e. corresponds to multiplication by $d+1$ if we took $r_D(b)$ as the origin on $X_b$.
In particular, the map $\delta_D^\circ$ has degree $(d+1)^2$ on each fiber.
We will also consider the rational map
\[
	\delta_D \colon X \dashrightarrow X
\]
induced by $\delta_D^\circ$ on the open locus.

Let us note that the map $\delta_D^\circ$ only depends on the class of $D$ in the Picard group of $X$ (hence the \Neron--Severi group for a K3 surface $X$), and furthermore does not change if we add a divisor pulled back from $B$ to the linear equivalence class of $D$ on $X$.

\subsubsection{Parabolic dynamics}
	\label{sssec:parabolic_dynamics_number_field}
Let now $\gamma\colon X\to X$ be a parabolic automorphism in the sense of \autoref{sec:elliptic_fibrations_and_currents}.
Our goal will be to understand the interaction between $\gamma$ and $\delta_D$.
On the smooth locus $X^\circ$ we have the identity
\[
	\delta_D^\circ \left(\gamma(x)\right) = \gamma^{d+1}\left(\delta_D^\circ (x)\right)
\]
which it suffices to verify on each closed fiber $X_b$.
Let $\gamma_b$ be the degree $0$ divisor on $X_b$ such that $\gamma(x)=x \oplus \gamma_b$ in the fiber, for $b=\pi(x)$.
Then we have, using the formula for $\delta_D^\circ$, that
\begin{align*}
	\delta_D^\circ \left(\gamma(x)\right)& \sim_{X_b} (d+1)(x+\gamma_b)- (D\cap X_b)\\
	 \gamma^{d+1}\left(\delta_D^\circ(x)\right)&
	 \sim_{X_b}
	 \big[(d+1)x- (D\cap X_b)\big] \oplus (d+1)\gamma_b
\end{align*}
and these divisors are clearly linearly equivalent on $X_b$.

\subsubsection{A common space for the dynamics}
	\label{sssec:a_common_space_for_the_dynamics}
In order to analyze the behavior of heights under the rational map $\delta_D$ on $X$ we need to pass to a completion on which to compare line bundles.
We therefore consider the graphs:
\begin{align*}
	\Gamma(\delta^\circ_D)& :=
	\{\left(x,\delta_D^\circ(x)\right)\colon x\in X^\circ\} \subset X^{\circ}\times X^\circ\\
	\Gamma(\delta_D)& = \ov{\Gamma(\delta^\circ_D)} \subset X\times X
\end{align*}
with projections to each factor denoted $\pi_1$ and $\pi_2$ respectively.
Consider the map
\begin{align*}
	\gamma_2\colon X\times X & \to X\times X\\
 	\gamma_2(x,y) & := \left(\gamma(x),\gamma^{d+1}(y)\right)
\end{align*}

\begin{proposition}
	\label{prop:graph_of_delta_is_preserved}
	The map $\gamma_2$ preserves $\Gamma(\delta_D)$ and induces an automorphism on it.
\end{proposition}
\begin{proof}
	We already know that $\gamma_2$ preserves $\Gamma(\delta_D^\circ)$, since this is equivalent to the identity
	\[
		\left(\gamma(x),\delta_D^\circ(\gamma(x))\right) =
		\left(\gamma(x),\gamma^{d+1}(\delta_D^\circ(x))\right)
	\]
	established in \autoref{sssec:parabolic_dynamics_number_field}.
	But $\gamma_2$ is a regular (invertible) map on $X\times X$ and hence preserves the closure $\Gamma(\delta_D)$ of $\Gamma(\delta_D^\circ)$.
\end{proof}

\subsubsection{Normalization of the graph}
	\label{sssec:normalization_of_the_graph}
Denote by $\wtilde{X}$ the normalization of the graph $\Gamma(\delta_D)=:\ov{X}$.
The automorphism $\gamma_2$ extends to $\wtilde{X}$ since the normalization is functorial and we will denote it by $\tilde{\gamma}_2$.
The same applies to the projection $\pi_1,\pi_2$ on $\ov{X}$, but we will also denote the projection $\tilde{\pi}_2\colon \wtilde{X}\to X$ by $\tilde{\delta}_D$.

\subsubsection{Singularities of $\wtilde{X}$}
	\label{sssec:singularities_of_tildeX}
The normal surface $\wtilde{X}$ has a birational morphism to the smooth surface $X$, and since the singularities of $\wtilde{X}$ can be resolved by a birational morphism to $\wtilde{X}$, the singularities of $\wtilde{X}$ are called ``sandwiched'' (see \cite{Spivakovsky1990_Sandwiched-singularities-and-desingularization-of-surfaces-by-normalized-Nash-transformations} for more on sandwiched singularities).

Sandwiched (normal) singularities are rational (see \cite{Artin1966_On-isolated-rational-singularities-of-surfaces} for the definition of rational singularities).
Indeed \cite[Prop.~1.2]{Lipman1969_Rational-singularities-with-applications-to-algebraic-surfaces-and-unique} says that singularities mapping birationally to rational ones are also rational, and smooth points are rational singularities. Furthermore, rational singularities on normal surfaces are in particular $\mathbb{Q}$-factorial \cite[Prop.~17.1]{Lipman1969_Rational-singularities-with-applications-to-algebraic-surfaces-and-unique}, so every Weil divisor on $\wtilde{X}$ has some multiple which is Cartier.

\subsubsection{Action on line bundles}
	\label{sssec:action_on_line_bundles}
We now consider the dynamics of the two maps on line bundles.
Consider $L:=\cO(D)$ the line bundle on $X$ associated to $D$.
We also denote $L^0:=\gamma^*L - L$.

\begin{proposition}[Pullbacks of line bundles]
	\label{prop:pullbacks_of_line_bundles}
	With notation as above:
	\begin{enumerate}
		\item We have the isomorphism of line bundles on $\wtilde{X}$:
		\[
			\tilde{\delta}_D^{*}L = (d+1)^2 \pi_1^* L + E^{vert}
		\]
		where $E^{vert}$ is a vertical line bundle.
		\item After passing to a finite index subgroup of $\Aut^\circ_\pi(X)$, for any $\gamma$ and associated $\tilde{\gamma}_2$, the action of $\tilde\gamma_2$ on vertical divisors is trivial.
		\item For any $\gamma$ in the finite index subgroup of $\Aut^\circ_\pi(X)$ as in the previous part, we also have the isomorphism of line bundles on $\wtilde{X}$:
		\[
			\tilde{\delta}_D^{*}\left((\gamma^{d+1})^* L\right) = (d+1)^2 \pi_1^* \left(\gamma^*L\right) + E^{vert}
		\]
		where $E^{vert}$ is the same line bundle as in part (i).
	\end{enumerate}
\end{proposition}
\noindent Subtracting the result in part (i) from that in part (iii) yields:
\begin{corollary}[Degree $0$ identity]
	\label{cor:degree_0_identity}
	There exists a finite index subgroup of $\Aut_\pi^\circ(X)$ such that any element $\gamma$ in it satisfies the identity of line bundles on $\wtilde{X}$:
	\begin{align*}
		\tilde{\delta}_D^{*}\left((\gamma^{d+1})^* L - L\right) & = (d+1)^2 \pi_1^* \left(\gamma^*L - L\right)\\
		\intertext{
		\quad or using the notation $L^0 = \gamma^*L-L$
		}
		\tilde{\delta}_D^{*}\left(L^0 + \gamma^*L^0 + \cdots + (\gamma^d)^* L^0\right)
		& = (d+1)^2 \pi_1^* L^0
	\end{align*}
\end{corollary}
\noindent Note that the second form of the expression telescopes to the first one.

\begin{proof}[Proof of \autoref{prop:pullbacks_of_line_bundles}]
	For part (i), let $X_{k(B)}$ denote the (scheme-theoretic) generic fiber of $X$ over $k(B)$, in other words the associated elliptic curve over the function field $k(B)$.
	We can restrict line bundles from $X$ to $X_{k(B)}$ and will do so without additional notation.
	Note also that we can do the same on $\wtilde{X}$ and the associated elliptic curves are isomorphic (via the induced map): $\wtilde{X}_{k(B)}\toisom X_{k(B)}$.
	Now the line bundle $\tilde\delta_D^* L - (d+1)^2\tilde\pi_1 L$ is trivial on $X_{k(B)}$, since $\tilde\delta_D$ induces a multiplication by $(d+1)$ map on the elliptic curve $X_{k(B)}$.
	In particular the bundle has a trivializing section over $k(B)$, which gives a rational section on $\wtilde{X}$ which has only vertical poles and zeros (see also \cite[8.3 Prop.~3.4]{Liu2002_Algebraic-geometry-and-arithmetic-curves}).

	For part (ii), observe that vertical divisors are permuted by the lifts of parabolic automorphisms $\tilde\gamma_2$, so they are fixed by a finite index subgroup of $\Aut^\circ_\pi(X)$.

	For part (iii), apply the automorphism $\tilde\gamma_2$ to the identity in part (i).
	This gives:
	\[
		\tilde\gamma_2^*\tilde\delta_D^{*}L = (d+1)^2 \tilde\gamma_2^*\pi_1^* L + \tilde\gamma_2^*E^{vert}
	\]
	But as maps we have the identities $\tilde\delta_D\circ \tilde\gamma_2 = \gamma^d \circ \tilde\delta_D$ (using that $\tilde\delta_D=\tilde\pi_2$ and $\tilde\gamma_2$ acts on $\wtilde{X}$) and $\pi_1\circ \tilde\gamma_2 = \gamma\circ \pi_1$.
	Combined with the triviality of the action of $\tilde{\gamma}_2$ on vertical divisors, the claimed identity follows.
\end{proof}



\subsection{Preferred heights on elliptic fibrations}
	\label{ssec:preferred_heights_on_elliptic_fibrations}

\subsubsection{Setup}
	\label{sssec:setup_preferred_heights_on_elliptic_fibrations}
The heights we construct are not quite canonical, but they have some preferred properties.
In particular, they restrict to canonical heights on smooth fibers.

For the next result, recall that canonical height functions on elliptic curves depend on the choice of basepoint.
However, changing the basepoint only changes the canonical height function by a constant.
Therefore, we can speak of ``a canonical height function up to a constant''.

\begin{theorem}[Preferred heights on elliptic fibrations]
	\label{thm:preferred_heights_on_elliptic_fibrations}
	Consider a line bundle $L^0\in \Pic^0_\pi(X)$, i.e. one that restricts to have degree $0$ on each irreducible component of the fibers of the fibration.
	Then there exists a ``preferred'' height function $h^{pf}_{L^0}$ in its class, with the property that, when restricted to any smooth fiber $X_b$, the height function agrees, up to a constant depending on $b$, with a canonical height function of $L^0$ restricted to that fiber.

	In particular, the restricted height function is affine for the ``group law'' on the genus one smooth fibers.
	On the singular fibers of $\pi$ the height $h^{pf}_{L^0}$ can be taken to be identically $0$.
\end{theorem}

\noindent Note that preferred height functions are only unique up to pullback of a bounded function from $B$.
This is in analogy to the uniqueness of the potential $\phi$ from \autoref{thm:preferred_currents} up to pullbacks of potentials from $B$.

\begin{proof}
	Fix an ample divisor $D$ and associated line bundle $L$ as in \autoref{ssec:relative_dynamics_in_the_elliptic_fibration}.
	Any line bundle $L^0$ as in the statement can be written as a linear combination of bundles of the form $\gamma^* L - L$, as $\gamma$ ranges over generators of $\Aut_\pi^\circ(X)$.
	The desired conclusion is also preserved by taking linear combinations, so it suffices to prove the statement assuming that $L^0 = \gamma^*L - L$.
	Furthermore, we can assume that $\gamma$ belongs to the finite index subgroup to which \autoref{cor:degree_0_identity} applies, so we have on $\wtilde{X}$ the identity of line bundles:
	\begin{align*}
		\tilde{\delta}_D^{*}\left(L^0 + \gamma^*L^0 + \cdots + (\gamma^d)^* L^0\right)
		& = (d+1)^2 \pi_1^* L^0.
	\end{align*}
	Recall that we also have the identity
	\[
		\gamma^* L^0 = L^0 + \eta_{\Pic}(\gamma,L^0)
	\]
	where $\eta_{\Pic}$ is the pairing valued in line bundles from \autoref{sssec:a_pairing_valued_in_line_bundles}.
	In particular $(\gamma^*)^i L^0 = L^0 + i \cdot \eta_{\Pic}(\gamma,L^0) $ so we can rewrite the above identity on $\wtilde{X}$ as
	\begin{align*}
		\tilde{\delta}_D^{*}\left((d+1)\cdot L^0
		+
		\frac{d(d+1)}{2}\eta_{\Pic}\left(\gamma,L^0\right)
		\right)
		& = (d+1)^2 \pi_1^* L^0\\
		\intertext{and after cancelling a factor of $d+1$:}
		\tilde{\delta}_D^{*}\left(L^0
		+
		\frac{d}{2}\eta_{\Pic}\left(\gamma,L^0\right)
		\right)
		& = (d+1) \pi_1^* L^0.
	\end{align*}
	Let now $h_{L^0}$ be an arbitrary height function on $X$ associated to $L^0$.
	Recall also that we had a pairing valued in heights, compatible with the pairing valued in line bundles, denoted $\eta_h(\gamma,L^0)$.
	Then on $\wtilde{X}$, using functoriality of heights, we obtain the identity of height functions
	\begin{align*}
		\tilde{\delta}_D^{*}
		\left(
		h_{L^0} +
		\frac{d}{2}
		\eta_h(\gamma,L^0)
		\right)
		& = (d+1) \pi_1^* h_{L^0} + O(1)\\
	\end{align*}
	where the error term $O(1)$ only depends on the choice of initial height function $h_{L^0}$, $\gamma$, and $D$.
	Restrict this identity on the open locus $\wtilde X^\circ$, which can then be rewritten on $X^\circ$ to yield:
	\[
		h_{L^0} = \frac{1}{d+1}\left(
		(\delta_D^{\circ})^* h_{L^0}
		+ \frac{d}{2} \eta_h\left(\gamma,L^0\right)
		\right)
		+ O(1)
	\]
	This identity shows that if we consider the transformation which changes any height function associated to $L^0$ on $X$ by the formula
	\begin{align}
		\label{eqn:fixed_point_equation}
		h_{L^0} \mapsto \frac{1}{d+1}\left(
		(\delta_D^{\circ})^* h_{L^0}
		+ \frac{d}{2} \eta_h\left(\gamma,L^0\right)
		\right)
	\end{align}
	on the open locus $X^\circ$, then the resulting function on $X(\ov{k})$ is again a height function associated to $L^0$.
	Note that the transformation does not do anything to the values of the height on the singular fibers.
	This transformation is also visibly a contraction for the $L^\infty$-norm on heights on $X^\circ(\ov{k})$, by a factor of $\frac{1}{d+1}$.
	Using the standard converging geometric series argument, it follows that if we start from some height function $h_{L^0}$, there is a corresponding height function $h^{pf}_{L^0}$ which is fixed by the above transformation.

	It remains to observe that if we restrict $h^{pf}_{L^0}$ to any smooth fiber $X_b$, it is still a fixed point of the transformation in \autoref{eqn:fixed_point_equation}.
	Note also that the value $\eta_h(\gamma,L^0)(b)$ is independent of the point on the fiber.
	Consider then the function on $X_b(\ov{k})$ defined by $h_b:= h^{pf}_{L^0} - \frac{1}{2}\eta_h(\gamma,L^0)(b)$, which differs from $h^{pf}_{L^0}$ by a constant.
	Observe that $h_b$ satisfies exactly the equation for the canonical height associated to the line bundle $L^0\vert_{X_b}$ and basepoint $r_D(b)$, where $r_D(b)\in X_b(\ov{k})$ is a point such that $d \cdot r_D(b)\sim D\cap X_b$.
	Therefore $h^{pf}_{L^0}$ is (up to a constant) a canonical height function on each smooth fiber and thus satisfies the conclusion of the theorem.

	For irreducible components of fibers where $\pi$ is not a smooth map, observe that they are all birational to $\bP^1$.
	By assumption $L^0$ restricts to a bundle of degree $0$ on each such irreducible component, hence the induced height function is a bounded distance from the trivial height function $0$.
\end{proof}

\begin{remark}[The emergence of the variation of canonical height]
	\label{rmk:the_emergence_of_the_variation_of_canonical_height}
	In the proof of \autoref{thm:preferred_heights_on_elliptic_fibrations} above, we could have used any height in the class $\eta_{\Pic}(\gamma,L^0)$, not necessarily $\eta_{h}(\gamma,L^0)$.
	This would have changed the preferred height $h^{pf}_{L^0}$ by a (bounded) function pulled back from $B$.
	But the variation of canonical height does emerge dynamically as follows.

	The statement of \autoref{thm:preferred_heights_on_elliptic_fibrations} does not involve any parabolic automorphism, but we did use one to construct the preferred height.
	Suppose now $\gamma'$ is another parabolic automorphism.
	Then we have the formula:
	\begin{align}
	\label{eqn:preferred_heights_dynamics}
		\left(\gamma'\right)^* h^{pf}_{L^{0}}	= h^{pf}_{L^{0}} + \eta_h(\gamma',L^{0})
	\end{align}
	and when $L^0 = \xi(\gamma)$ we will write $\eta_h(\gamma',\gamma)$ instead.
	Indeed, it suffices to check \autoref{eqn:preferred_heights_dynamics} pointwise, and for this we can use the fact that the preferred height is affine, for the group law on fibers.
	Then, after picking in each fiber a basepoint, the result follows from \autoref{thm:preferred_heights_on_elliptic_fibrations}.
\end{remark}




\section{Boundary heights, direct construction}
	\label{sec:boundary_heights_direct_construction}

\subsubsection*{Outline of section}

We now use the same arguments that established the existence of canonical currents to produce canonical heights for arbitrary boundary classes.
Our main result on the existence of canonical heights is stated in \autoref{thm:canonical_heights_on_the_boundary}.

In \autoref{ssec:preliminary_parabolic_estimates} we collect the needed estimates needed to handle elliptic fibrations.
Then in \autoref{ssec:heights_from_a_basis} we make everything ``explicit'' by fixing non-canonical representatives for the heights, analogously to choosing smooth differential forms to represent every cohomology class.
Then any other height function differs from the non-canonical representative by a bounded ``potential'' function on $X(\ov k)$.


\subsection{Preliminary parabolic estimates}
	\label{ssec:preliminary_parabolic_estimates}

Using the preferred representatives of height functions for relative degree $0$ lines bundles constructed in \autoref{thm:preferred_heights_on_elliptic_fibrations}, we can now establish some estimates for heights on general ample line bundles.

\subsubsection{Preparations with line bundles}
	\label{sssec:preparations_with_line_bundles}
Fix a ($\bQ$-)line bundle $L$ on $X$ of degree $1$ on the fibers of the elliptic fibration, let $h_{L}$ be a height function for $L$, and let $\gamma$ be a parabolic automorphism.
Set $L^0 = \gamma^* L - L$, let $h_{L^0}:=\gamma^*h_L- h_L$, and let $h^{pf}_{L^0}$ be a preferred height function provided by \autoref{thm:preferred_heights_on_elliptic_fibrations}.
The function $\phi\colon X\left(\ov k\right)\to \bR$ defined by $\phi:=h_{L^0}-h^{pf}_{L^0}$ is uniformly bounded.
We will make use of the following immediate formulas:
\begin{align*}
	\gamma^* h_L & = h_L + h_{L_0}		&	\gamma^* h^{pf}_{L^0} & = h^{pf}_{L^0} + \eta_h(\gamma,L^0)\\
	h_{L^0} & = h_{L^0}^{pf} + \phi 	&	(\gamma^i)^* h^{pf}_{L^0} & = h^{pf}_{L^0} + i\cdot \eta_h(\gamma,L^0)
\end{align*}
where $\eta_h$ is the height-valued pairing from \autoref{thm:height_valued_pairing}.
Note that the formula for the pullback $\gamma^* h^{pf}_{L^0}$ follows from \autoref{eqn:preferred_heights_dynamics}.
We abuse notation and view $\eta_h$ as giving heights not only on $B$, but also on $X$ by pullback along $\pi$.
The next result is a direct analogue of \autoref{prop:large_iterates_of_one_automorphism}.

\begin{proposition}[Large iterates of one automorphism]
	\label{prop:large_iterates_of_one_automorphism_heights}
	With notation as above, we have the identity of functions on $X(\ov{k})$:
	\begin{align*}
		\left(\gamma^i\right)^* h_{L^0} & = h_{L^0} + i\cdot \eta_h(\gamma,L^0) +
		(\gamma^i)^* \phi - \phi\\
		\left(\gamma^{n}\right)^* h_L & = h_L + n\cdot h_{L^0} + \tfrac{n(n-1)}{2}\eta_h(\gamma,L^0) + \\
		 & + \left[S_n(\gamma,\phi) - n\phi\right]
	\end{align*}
	where $S_n(\gamma,\phi):=\phi + \gamma^*\phi + \cdots + \gamma^{n-1}\phi$ is a Birkhoff sum associated to $\phi$ for the dynamics of $\gamma$.	
\end{proposition}
\begin{proof}
	The identity for $(\gamma^i)^*h_{L^0}$ follows by applying $(\gamma^i)^*$ to the identity $h_{L^0}=h_{L^0}^{pf}+\phi$ and using the simple iteration property for $h_{L^0}^{pf}$ established above.
	To obtain the identity for $\left(\gamma^n\right)^* h_{L}$, we simply observe that (by induction, say)
	\[
		\left(\gamma^n\right)^* h_L = h_L + h_{L^0} + \cdots + \left(\gamma^{n-1}\right)^* h_{L^0}
	\]
	and sum the identity for $(\gamma^i)^* h_{L^0}$ to conclude.
\end{proof}

\noindent We also need to establish an estimate combining several generators, in direct analogy with \autoref{prop:large_iterates_of_several_automorphisms}.

\begin{proposition}[Large iterates of several automorphisms]
	\label{prop:large_iterates_of_several_automorphisms_heights}
	Suppose that $\gamma_1,\ldots,\gamma_k$ are parabolic automorphisms, $L$ is a fixed line bundle as before, we denote by $L_i=\gamma_i^*L - L$ and associated height functions $h_L, h_{L_i}$, and preferred height functions $h^{pf}_{L_i}$ and functions $\phi_i=h_{L_i}-h_{L_i}^{pf}$.

	Then we have the identity
	\begin{align*}
		\left(\gamma_k^{n_k}\right)^*
		\cdots
		\left(\gamma_1^{n_1}\right)^* h_{L}
		& =
		h_L + \sum_{i=1}^k n_i h_{L_i} +
		\\
		& + \left[\sum_{i=1}^k S_{n_i}\left(
			\gamma_i,
			(\gamma_k^{n_k})^*
			\cdots
			(\gamma_{i+1}^{n_{i+1}})^* \phi_i
			\right)
			- n_i \phi_i\right]
			\\
		& + \left[
		\sum_{i=1}^k \tfrac{n_i(n_i-1)}{2}\eta_h(\gamma_i,\gamma_i)
		+
		\sum_{1\leq i<j\leq k} n_i n_j \cdot \eta_h(\gamma_i,\gamma_j)
		\right]
	\end{align*}
	where by abuse of notation we denoted by $\eta_h(\gamma_i,\gamma_j)$ the pullback to $X$ of the height function $\eta_h(\gamma_i,\xi(\gamma_j))$ on $B$.	
\end{proposition}
\begin{proof}
	We apply the earlier formula from \autoref{prop:large_iterates_of_one_automorphism_heights} for a single automorphism successively, but also taking into account that preferred heights satisfy the formula in \autoref{eqn:preferred_heights_dynamics}.
	In particular, for the heights $h_{L_i}$ and automorphisms $\gamma_j$, using the intermediate preferred heights (which disappear from the final statement), we see that:
	\[
		\left(\gamma_j^{n_j}\right)^* h_{L_i} =
		h_{L_i} + n_j \cdot \eta_h(\gamma_j,\gamma_i) +
		\left[
		\left(\gamma_j^{n_j}\right)^{*}\phi_i - \phi_i
		\right].
	\]
	Then the desired formula follows by induction on $k$.
	The case $k=1$ is \autoref{prop:large_iterates_of_one_automorphism_heights}.
	Now apply $(\gamma_k^{n_k})^*$ to each term in the formula for $k-1$ automorphisms to find:
	\begin{align*}
		\left(\gamma_k^{n_k}\right)^*
		h_L  & = h_L + n_k h_{L_k} + \tfrac{n_k(n_k-1)}{2}\eta_h(\gamma_k,L_k) + \left[S_n(\gamma_k,\phi_k) - n_k\phi_k\right] \\
		\left(\gamma_k^{n_k}\right)^* n_i h_{L_i} & =
		n_i h_{L_i} + n_i n_k \cdot \eta_h(\gamma_k,\gamma_i) +
		n_i \left[
		\left(\gamma_k^{n_k}\right)^{*}\phi_i - \phi_i
		\right]
		\\
		\left(\gamma_k^{n_k}\right)^* n_i \phi_i & = n_i \left(\gamma_k^{n_k}\right)^* \phi_i\\
	\end{align*}
	and the identity
	\[
		\left(\gamma_k^{n_k}\right)^* S_{n_i}\left(
			\gamma_i,
			(\gamma_{k-1}^{n_{k-1}})^*
			\cdots
			(\gamma_{i+1}^{n_{i+1}})^* \phi_i
			\right) =
			S_{n_i}\left(
			\gamma_i,
			(\gamma_{k}^{n_{k}})^*
			\cdots
			(\gamma_{i+1}^{n_{i+1}})^* \phi_i
			\right)
	\]
	which follows because the automorphisms $\gamma_\bullet$ commute with each other.
	Note also that parabolic automorphisms preserve any height functions pulled back from $B$, so the desired identity follows.	
\end{proof}

Let us set $\gamma=\gamma_k^{n_k}\cdots \gamma_1^{n_1}$ so that we can write a bit more concisely the expression in \autoref{prop:large_iterates_of_several_automorphisms_heights} as:
\begin{align}
	\label{eqn:large_iterate_grouped_terms}
	\begin{split}	
	\gamma^* h_{L}
	& =
	h_L + \sum_{i=1}^k n_i \left(h_{L_i} + \tfrac{1}{2}\eta_h(\gamma_i,\gamma_i)\right) + \frac{1}{2}\eta_h(\gamma,\gamma)+
	\\
	& + \left[\sum_{i=1}^k S_{n_i}\left(
		\gamma_i,
		(\gamma_k^{n_k})^*
		\cdots
		(\gamma_{i+1}^{n_{i+1}})^* \phi_i
		\right)
		- n_i \phi_i\right]
	\end{split}
\end{align}
where, in comparison to similar expressions in \autoref{sssec:summary_elliptic_fibrations}, we have already normalized $L$ to have degree $1$ on fibers.



\subsection{Heights from a basis}
	\label{ssec:heights_from_a_basis}

\subsubsection{Setup}
	\label{sssec:setup_heights_from_a_basis}
In order to establish the existence of canonical heights associated to every point on the boundary of the ample cone, we fix a basis for the \Neron--Severi group of $X$ and choose once and for all representatives for the height functions of the basis elements.
Extended by linearity, we have for every $\bR$-line bundle $L$ a height function $h_{L}^{lin}$.
For any other height function $h_L$ associated to the same line bundle there is a unique function $\phi(h_L)$ in $L^\infty\left(X(\ov{k})\right)$ such that we have
\begin{align}
	\label{eqn:height_noncanonical_potential}
	h_L = h_{L}^{lin} + \phi(h_L)
\end{align}
directly analogous to \autoref{eqn:omega_noncanonical_potential}.

\subsubsection{Change of height under automorphism}
	\label{sssec:change_of_height_under_automorphism}
For an element $\gamma\in \Aut(X)$ we then have the basic formula
\[
	\gamma^* h_L^{lin} = h^{lin}_{\gamma^*L} + \phi(\gamma^*h^{lin}_{L})
\]
which relates to height functions associated to the line bundle $\gamma^*L$, namely $h_{\gamma^*L}^{lin}$ and $\gamma^* h_{L}^{lin}$.
Let us now iterate this formula when applying several automorphisms:
\begin{align*}
	(\gamma_n\cdots \gamma_1)^* h^{lin}_L & =
	\gamma_1^*\cdots \gamma_n^* h^{lin}_L =
	\gamma_1^*\cdots \gamma_{n-1}^* \left(h^{lin}_{\gamma_{n}^*L} + \phi\left(\gamma_{n}^*h^{lin}_{L}\right)\right)\\
	& = \gamma_1^*\cdots \gamma_{n-2}^*
	\left(h^{lin}_{\gamma_{n-1}^*\gamma_{n}^*L} + \phi\left(\gamma_{n-1}^*h^{lin}_{\gamma_n^* L}\right) + \gamma_{n-1}^*\phi\left(\gamma_{n}^*h^{lin}_{L}\right)\right)\\
	& = \cdots\\
	& = h^{lin}_{\gamma_1^* \cdots \gamma_n^* L} + \sum_{i=1}^{n} \gamma_1^*\cdots \gamma_{i-1}^{*}\phi\left(\gamma_{i}^* h^{lin}_{\gamma_{i+1}^*\cdots\gamma_n^* L}\right)
\end{align*}
In other words, we have that
\begin{align*}
	\phi\left(\gamma_1^*\cdots \gamma_n^* h_L^{lin}\right) & =
	\sum_{i=1}^{n} \gamma_1^*\cdots \gamma_{i-1}^{*}\phi\left(\gamma_{i}^* h^{lin}_{\gamma_{i+1}^*\cdots\gamma_n^* L}\right)
\end{align*}
which is the direct analogue of \autoref{eqn:gamma1n_omega}, except that we now pull back line bundles, instead of pushing forward \Kahler metrics.

We record the estimates that we will need.
The next proposition simply replaces in \autoref{prop:parabolic_convergence_to_the_boundary} the class $[\omega]$ by $[L]$ and the representative $(1,1)$-form $A([\omega])$ by the representative height $h_{L}^{lin}$.
It uses the normed space $L^\infty(X(\ov{k}))$ instead of $C^0(X(\bC))$.

\begin{proposition}[Parabolic convergence to the boundary for heights]
	\label{prop:parabolic_convergence_to_the_boundary_heights}
	Consider a parabolic automorphism $\gamma=\gamma_k^{n_k}\cdots \gamma_1^{n_1}$ written as a product of generators and $[L]$ a parabolically normalized class.
	\begin{enumerate}
		\item
		\label{large_parabolic_brings_close_to_eta_height}
		There exists a constant $C_2=C_2(E)$ giving the estimate on potentials:
		\[
			\norm{
			\frac{\phi\left(\gamma^* h_L^{lin}\right)}{M(\gamma^* [L])
			}
			-
			\frac{\phi\left(\eta_h(\gamma,\gamma)\right)}{M\left(\eta(\gamma,\gamma)\right)}
			}_{L^{\infty}}
			\leq
			\frac{C_2}{1 + \norm{\gamma}_{\NS}}
		\]
		\item \label{key_parabolic_estimate_height}
		We also have the estimate
		\[
			\norm{
			\phi\left(\gamma^* h_L^{lin}\right)
			}_{L^{\infty}}
			\leq C_3 \left(1+\norm{\gamma}_{\NS}^2\right)\norm{[L]}
		\]
		which holds for arbitrary $[L]$.
	\end{enumerate}
\end{proposition}
\noindent The proof is word for word the same as for \autoref{prop:parabolic_convergence_to_the_boundary}, with the replacements indicated above.

\subsubsection{Adapted generators}
	\label{sssec:adapted_generators}
We will use the generators for $\Aut(X)$ fixed in \autoref{sssec:fixed_finite_set_of_generators}, and keep the notation $\leqapprox$ to denote inequalities that hold up to constants that only depend on the fixed background geometric objects.
For any such generator $\gamma$, we have the basic inequality
\begin{align}
	\label{eqn:operator_norm_bound_heights}
	\norm{\phi(\gamma^* h_L^{lin})}_{L^\infty}\leqapprox \norm{\gamma}_{op}  \cdot M(L)
\end{align}
where as before $M(L)$ is the mass of $L$ relative to a fixed once and for all ample class, see \autoref{sssec:mass_and_normalization}.
For parabolic automorphisms this is \autoref{prop:parabolic_convergence_to_the_boundary_heights}\ref{key_parabolic_estimate_height}, and for the finitely many remaining automorphisms we pick a sufficiently large constant in $\leqapprox$ so that the bound continues to hold.



\subsection{Existence of boundary heights}
	\label{ssec:existence_of_boundary_heights}

Here is the main result of this section, in analogy with \autoref{thm:continuous_family_of_boundary_currents}.
Recall that $\partial^\circ \Amp_c(X)$ is the blown-up boundary of the ample cone, as in \autoref{sssec:blown_up_boundary_of_the_ample_cone}, and $\cH(X)$ denotes the space of heights on $X$.

\begin{theorem}[Canonical heights on the boundary]
	\label{thm:canonical_heights_on_the_boundary}
	There exists an assignment of height functions
	\[
		h^{can}\colon \partial^\circ \Amp_c(X) \to \cH(X)
	\]
	compatible with the map from $\partial^\circ \Amp_c(X)$ to $\partial \Amp(X)\subset \NS(X)$, and from heights $\cH(X)$ to $\NS(X)$.
	The map has the following additional properties:
	\begin{enumerate}
		\item \label{hcan_is_equivariant}
		The map $h^{can}$ is $\Aut(X)$-equivariant.
		\item \label{hcan_is_continuous}
		For a fixed $p\in X\left(\ov k\right)$, the function
		\begin{align*}
			\partial^\circ \Amp_c(X) & \xrightarrow{\alpha\mapsto h^{can}_\alpha(p)}\bR_{\geq 0}
		\end{align*}
		is continuous.
		\item \label{hcan_silverman_par}
		On the real projective $(\rho-2)$-spaces in $\partial^\circ \Amp_c(X)$ associated to rational boundary rays corresponding to elliptic fibrations, the height function $h^{can}$ coincides with Silverman's variation of canonical height pairing:
		\[
			h^{can}(\lambda[E]\times [\xi])=\lambda\cdot \frac{\eta_{[E],h}([\xi],[\xi])}{\norm{\xi}_{\NS}^2} \text{ for }[\xi]\in \bP([E]^{\perp}/[E]).
		\]
		where $\eta_{[E],h}$ is the height-valued pairing from \autoref{thm:height_valued_pairing} associated to the genus one fibration given by $[E]$.
		\item \label{hcan_silverman_hyp}
		For boundary classes $\alpha_\pm$ expanded/contracted by hyperbolic automorphisms of $X$, the heights $h^{can}_{\alpha_\pm}$ coincide with Silverman's canonical heights \cite{Silverman_Rational-points-on-K3-surfaces:-a-new-canonical-height}.
		\item
		\label{hcan_def_prop}
		Fix an ample line bundle $L$, height function $h_L$, and mass function (on cohomology classes) $M$ as in \autoref{sssec:mass_in_cohomology}.
		For $p\in X(\ov{k})$ and a sequence of automorphisms $\gamma_i\in \Aut(X)$ as in \autoref{sssec:fixed_finite_set_of_generators}, such that $L_i:=\gamma_1^*\cdots \gamma_i^* L$ satisfies $\frac{1}{M(L_i)}L_i \to \alpha\in \partial^\circ \Amp_c(X)$, we have that
		\[
			h^{can}_\alpha(p) = \lim_{i\to \infty} \frac{1}{M(L_i)}\gamma_1^* \cdots \gamma_i^* h_L(p)
		\]
		and the limit exists.
		In particular $h^{can}_\alpha(p)\geq 0$ for any $p\in X\left(\ov k\right)$ and any $\alpha \in \partial^\circ \Amp_c(X)$.
	\end{enumerate}
	\hspace{\parindent}Furthermore, property \ref{hcan_def_prop} characterizes $h^{can}$, as does the combination of properties \ref{hcan_is_equivariant}, \ref{hcan_is_continuous}, together with either \ref{hcan_silverman_par} or \ref{hcan_silverman_hyp}.
\end{theorem}
\noindent It would be interesting and useful to have a criterion that uniquely characterizes $h^{can}$ in terms of its equivariance and positivity properties.
It is reasonable to expect that the uniqueness of positive currents in \autoref{thm:uniqueness_in_irrational_classes_verbitsky_sibony}, due to Sibony--Verbitsky in the archimedean case, also holds in the non-archimedean case.
Then the canonical height functions $h^{can}_{\alpha}$ for irrational $\alpha$ would be the unique adelic height function with subharmonic potentials at each place.

\begin{proof}
	Let us note that if a mapping $h^{can}$ satisfies the continuity property in \ref{hcan_is_continuous}, and say \ref{hcan_silverman_par} or \ref{hcan_silverman_hyp}, then it is unique and satisfies \ref{hcan_is_equivariant} automatically.
	Indeed, the expanded/contracted rays of hyperbolic automorphisms are dense on the boundary, as are the parabolic classes, so together with continuity they uniquely determine the values of other heights.
	Together with uniqueness of canonical heights for hyperbolic automorphisms, the equivariance property from \ref{hcan_is_equivariant} then follows.
	So it suffices to construct $h^{can}$ by the process described in \ref{hcan_def_prop} and check that it yields, for a fixed $p\in X(\ov{k})$, continuous functions on $\partial^\circ \Amp_c(X)$.

	The existence of the limit
	\[
		h^{can}_\alpha(p) =
		\lim_{i\to \infty} \frac{1}{M(L_i)}\gamma_1^* \cdots \gamma_i^* h_L(p) =
		\lim_{i\to \infty} \frac{1}{M(L_i)}h_L(\gamma_i\cdots \gamma_1 p)
	\]
	follows from arguments word for word identical to those in \autoref{ssec:existence_of_boundary_potentials}, with the following replacements.
	Instead of $C^0(X(\bC))$-bounds, use $L^\infty\left(X(\ov{k})\right)$, and instead of $A([\omega])$ use $h^{lin}_{L}$.
	Continuity, for fixed $p$ and varying point on the boundary, follows analogously.

	To be more specific, the bounds used in the proof of \autoref{thm:continuous_family_of_boundary_currents} for currents are listed in \autoref{sssec:the_necessary_bounds_currents}.
	Besides the ones that involve masses of cohomology classes in \autoref{eqn:useful_inequalities}, which continue to hold in the present case for the same reason, the bound in \autoref{eqn:basic_bound_potential_all_automorphisms} holds in the case of heights by \autoref{eqn:operator_norm_bound_heights}.

	The uniform boundedness of the sequence $\frac{\phi\left(h_{L}^{lin}(\gamma_i\cdots \gamma_1 p)\right)}{M(\gamma_1^*\cdots \gamma_i^* L)}$ (with bounds independent of $p\in X(\ov{k})$) follows just like in \autoref{prop:boundedness_of_potentials} for potentials.
	Similarly, the Cauchy property for such sequences is proved just like in \autoref{prop:cauchy_property_for_a_fixed_geodesic} and the same estimates imply continuity of $h_{[\alpha]}^{can}(p)$ at an irrational $[\alpha]$ for fixed $p$.

	The continuity of the height function at irrational points uses the bound in \autoref{prop:parabolic_convergence_to_the_boundary_heights}\ref{large_parabolic_brings_close_to_eta_height}, just like in \autoref{sssec:estimates_for_continuity_at_a_parabolic_boundary_point}.
\end{proof}




\section{Suspension space construction}
	\label{sec:suspension_space_construction}

\subsubsection*{Outline of section}

It is natural to wonder if the constructions of canonical currents, and heights, could be formulated in terms of the geodesic flow for an appropriate hyperbolic manifold, instead of the more direct, combinatorial construction in \autoref{sec:boundary_currents_direct_construction} and \autoref{sec:boundary_heights_direct_construction}.
We describe below the candidate space and some of its properties.
However, because of the need to blow up boundary of the ample cone, and because of the non-uniqueness of the associated currents and heights for boundary classes corresponding to elliptic fibrations, the homogeneous spaces constructed in this section cannot, as far as we can see, be used to construct the canonical currents and heights.

We establish some results from ergodic theory and formulate a number of questions.
Some of the constructions in this section will be taken up in future work, where we will make use of these homogeneous spaces to study the currents and heights.


\subsection{Suspension space construction}
	\label{ssec:suspension_space_construction}

\subsubsection{Setup}
	\label{sssec:setup_suspension_space_construction}
We return to the constructions in \autoref{sec:hyperbolic_geometry_background} and discuss how they interact with the geometry of the K3 surface.
Specifically, let $X$ be a K3 surface satisfying the standing assumptions in \autoref{sssec:standing_assumptions_main}.
Let $N:=\NS(X)$ be its \Neron--Severi group, and set $\bbG:=\SO(N)$ to be the orthogonal group preserving the intersection pairing on $N$.
We will denote by $G:=\bbG(\bR)^+$ the connected component of the identity, and by $\Gamma:=\Aut(X)\cap G$ the corresponding subgroup.
By our standing assumptions $\Gamma$ is a lattice in $G$.

Denote by $\wtilde\cM:=\Amp^1(X)$ the ample classes of volume $1$, which we view as the \Teichmuller space of Ricci-flat metrics on $X$ normalized to volume $1$ (and in the \Neron--Severi part).
The quotient $\cM:=\leftquot{\Gamma}{\wtilde{\cM}}$ is then naturally the moduli space, and its frame bundle is $\cQ := \leftquot{\Gamma}{G} $ after we choose a basepoint in $\wtilde\cM$ with a fixed frame in its tangent space.
We have a right action of $G$ on $\cQ$, a natural object in homogeneous dynamics.

\subsubsection{Suspended space}
	\label{sssec:suspended_space}
Consider now the space $\wtilde{\cX}:=G\times X$, equipped with an action of $G$ on the right by only acting on the first factor by right multiplication, and a left (diagonal) action of $\Gamma=\Aut(X)$ on both factors.
We can then form
\[
	\leftquot{\Gamma}{\wtilde{\cX}}
	=: \cX \to \cQ =
	\leftquot{\Gamma}{G}
\]
which is $G$-equivariant on the right.
This is a nonlinear analogue of constructing a flat vector bundle from a representation of the fundamental group.

\subsubsection{Basic ergodic theory}
	\label{sssec:basic_ergodic_theory}
If a group $H$ acts on a space $Y$, an $H$-invariant measure $\mu$ is \emph{ergodic} if there does not exist a decomposition $Y=Y_1\coprod Y_2$ into measurable $H$-invariant sets with $\mu(Y_1)\neq 0 \neq \mu(Y_2)$.
We will be interested in the situation when $Y$ is a manifold and carries a topology, in which case orbit closures $\ov{H\cdot y}$ for $y\in Y$ are to be understood in the ordinary, Euclidean topology (even if $Y$ also has a Zariski topology).

\subsubsection{The groups of interest}
	\label{sssec:the_groups_of_interest}
Besides the full semisimple group $G$, we will be interested in the dynamics of the following subgroups:
\begin{enumerate}
	\item A parabolic subgroup denoted $P\subset G$.
	\item A diagonal (Cartan) subgroup $A\subset P$ whose elements will be denoted $g_t$, and we view as generating the geodesic flow on $\cQ$.
	\item The maximal unipotent subgroup $U\subset P$, or subgroups $U'\subset U$.
\end{enumerate}
We fix but do not specify the necessary data (a Cartan subalgebra and Weyl chamber) needed to define these groups.
For future reference, we also a maximal compact subgroup $K\subset G$, determined by a fixed basepoint in the ample cone, and $M:=K\cap Z_K(A)$ the centralizer of $A$ inside $K$.



\subsection{\texorpdfstring{$G$}{G}-dynamics}
	\label{ssec:g_dynamics}

In this section, we begin by recalling some elementary facts from dynamics.
Then we proceed to formulate some questions that we find natural and of interest.

\subsubsection{Correspondence principle}
	\label{sssec:correspondence_principle}
It is immediate to check from the definitions that the following objects give equivalent data:
\begin{enumerate}
	\item A finite $\Gamma$-invariant measure on $X$.
	\item A finite $G$-invariant measure on $\cX$.
	\item A locally finite measure on $G\times X$ which is invariant by $G$ on the right, and by $\Gamma$ on the left.
\end{enumerate}
An analogous discussion applies to establish an equivalence between closed $\Gamma$-invariant sets on $X\times G/H$ and closed $H$-invariant sets in $\cX$, for general closed subgroups $H\subseteq G$.
The next result is an immediate corollary of Cantat's \cite[Thms.~0.2 \& 0.3]{Cantat_Dynamique-des-automorphismes-des-surfaces-K3} (see also Wang \cite{Wang_Rational-points-and-canonical-heights-on-K3-surfaces-in-bf-P1bf-P1bf}):

\begin{corollary}[$G$-invariant classification]
	\label{cor:g_invariant_classification}
	Suppose that the K3 surface, with assumptions as in \autoref{sssec:standing_assumptions_main}, has an elliptic fibration.

	Then any ergodic $G$-invariant measure on $\cX$, or closed $G$-invariant set, is obtained by suspension as in \autoref{sssec:suspended_space} from $X$ and one of the following:
	\begin{enumerate}
		\item A finite $\Aut(X)$-invariant set, with uniform measure.
		\item A totally real $\Aut(X)$-invariant submanifold in $X$, equipped with a smooth measure in the Lebesgue class of the submanifold.
		\item All of $X$ equipped with $\dVol_X$.
	\end{enumerate}
\end{corollary}

\subsubsection{Consequences of Howe--Moore theorem}
	\label{sssec:consequences_of_howe_moore_theorem}
Recall that the Howe--Moore theorem \cite[\S2]{Zimmer1984_Ergodic-theory-and-semisimple-groups} says that if the semisimple Lie group $G$ has a unitary representation on a Hilbert space, and there are no $G$-invariant vectors, then there are no $H$-invariant vectors for any unbounded (i.e. with noncompact closure) subgroup of $G$.
As a consequence, if $G$ acts ergodically on a measure space, then so does any of its unbounded subgroups.
Since the measures in \autoref{cor:g_invariant_classification} are $G$-ergodic, we obtain:

\begin{corollary}[Orbit closures for a.e. point]
	\label{cor:orbit_closures_for_a_e_point}
	Fix the assumptions as in \autoref{cor:g_invariant_classification} and let $\mu$ be one of the ergodic $G$-invariant measures.
	
	Then for any noncompact Lie subgroup $H\subset G$ (e.g. $U,A$ or $P$) and $\mu$-a.e. $x\in \cX$ the orbit $H\cdot x$ is dense in the support of $\mu$.
\end{corollary}

\subsubsection{Relation to $P$-dynamics}
	\label{sssec:relation_to_p_dynamics}
We recall here some results of Fur\-sten\-berg \cite[Thm.~2.1]{Furstenberg_Noncommuting-random-products}.
Fix a probability measure $\nu$ on $G$ which is ``admissible'', i.e. some convolution power of it is absolutely continuous for Haar measure on $G$.
Suppose $Y$ is locally compact and second countable, with a continuous $G$-action.
Then results of Furstenberg, see \cite[Thm.~1.4]{NevoZimmer1999_Homogenous-projective-factors-for-actions-of-semi-simple-Lie-groups} for precise statements, imply that there is a one-to-one correspondence between $\nu$-stationary measures on $Y$ and $P$-invariant measures on $Y$.
Recall that a $\nu$-stationary measure on $Y$ is a (probability) measure $\mu$ satisfying $\int_{G}\left(g_*\mu\right) d\nu(g)=\mu$.
One says that the action of $G$, or of $\Gamma$, on $Y$ is \emph{stiff} if any stationary measure is actually invariant, i.e. $g_*\mu = \mu,\forall g\in G$, or for all $g\in \Gamma$ respectively.
Stiffness for $\Gamma$ is expected to hold in situations such as $\Aut(X)$ acting on $X$, see e.g. conjectures of Furstenberg \cite{Furstenberg1998_Stiffness-of-group-actions}.
In the homogeneous setting these were resolved by Benoist--Quint \cite{BenoistQuint_Mesures-stationnaires-et-fermes-invariants-des-espaces-homogenes_annals} quite generally and by Bourgain--Furman--Lindenstrauss--Mozes in the case of tori \cite{BourgainFurmanLindenstrauss2011_Stationary-measures-and-equidistribution-for-orbits-of-nonabelian-semigroups-on-the-torus}.
Analogous results in the nonlinear setting, closer to the one considered here, have been obtained by Brown--Rodriguez-Hertz \cite{BrownRodriguez-Hertz_Measure-rigidity-for-random-dynamics-on-surfaces-and-related} for group actions on smooth real surfaces, and by Eskin, Mirzakhani, and Mohammadi in the setting of \Teichmuller dynamics \cite{EskinMirzakhaniMohammadi_Isolation-equidistribution-and-orbit-closures-for-the-rm-SL2Bbb-R-action-on-moduli,EskinMirzakhani_Invariant-and-stationary-measures-for-the-rm-SL2Bbb-R-action-on-moduli-space}.
It is therefore natural to pose:

\begin{question}[$P$-invariant implies $G$-invariant]
	\label{eg:p_invariant_implies_g_invariant}
	If a measure $\mu$ on $\cX$ is $P$-invariant, must it be also $G$-invariant?
\end{question}
This is close to, but not equivalent to stiffness for the action of $\Gamma=\Aut(X)$ on $X$.
This stiffness has been established under some assumptions by Cantat--Dujardin \cite{CantatDujardin2020_Stationary-measures-on-complex-surfaces}.
Let us also note that an affirmative answer to \autoref{eg:p_invariant_implies_g_invariant} implies stiffness for AC measures in the sense of \cite[Def.~4.1]{Furstenberg1998_Stiffness-of-group-actions}, by the same method as the proof of Thm.~4.2 in loc. cit. which identifies the boundaries of $\Gamma$ and $G$ when the measure on $\Gamma$ is AC.

\subsubsection{$U$-dynamics}
	\label{sssec:u_dynamics}
The discussion in \autoref{sssec:relation_to_p_dynamics} illustrates that there is little essential difference between random walks on $\Aut(X)$ and the action of $P$ on $\cX$.
In the setting of the suspended space $\cX$, it is also meaningful to ask about the action of unipotent subgroups.
This cannot be, as far as we know, directly translated into properties of the random walk.
The most interesting challenge to us seems:
\begin{question}[$U$-invariant implies $P$-invariant]
	\label{eg:u_invariant_implies_p_invariant}
	Can one classify all $U'$-invariant measures, for $U'\subset G$ a unipotent subgroup?

	Specifically, suppose that $\mu$ is a measure on $\cX$, invariant under $U$, and that projects to Haar measure $\dVol_\cQ$ on $\cQ$.
	Is it true that $\mu$ must also be $P$-invariant?
\end{question}

\begin{remark}[On unipotent classification]
	\label{rmk:on_unipotent_classification}
	On the homogeneous space $\cQ$, the classification of orbit closures and invariant measures for $U$, or one of its subgroups $U'$, follows from Ratner's work \cite{Ratner_On-Raghunathans-measure-conjecture,Ratner_On-measure-rigidity-of-unipotent-subgroups-of-semisimple-groups}, see also earlier work of Margulis \cite{Margulis1987_Formes-quadratriques-indefinies-et-flots-unipotents-sur-les-espaces-homogenes}.
	Using this classification of invariant measures and sets for $U'$ on $\cQ$,
	\autoref{eg:u_invariant_implies_p_invariant} is most interesting when the invariant measure, resp. set on $\cX$, when projected to $\cQ$ is $\dVol_{\cQ}$, resp. all of $\cQ$.
	The possibilities for $U'$-invariant measures on $\cQ$ are homogeneous ones supported on closed $U''$-orbits, for $U''$ an intermediate unipotent subgroup between $U'$ and $U$, or homogeneous measures on $\leftquot{\Gamma\cap H}{H}$ where $H\subset G$ is a semisimple subgroup (yielding a totally geodesic hyperbolic submanifold in $\cM$).

	The ergodic invariant measures on $\cX$ projecting to homogeneous closed $U''$-orbits on $\cQ$, for $U''$ a unipotent subgroup, are immediately classified.
	They must be homogeneous measures on tori of dimension $0$, $1$, or $2$, contained in a single fiber of the corresponding elliptic fibration.
\end{remark}



\subsection{\texorpdfstring{$A$}{A}-dynamics}
	\label{ssec:a_dynamics}

\subsubsection{Setup}
	\label{sssec:setup_a_dynamics}
We now look in more detail at the dynamics of the subgroup $A$ generating the geodesic flow on $\cM$.
To introduce some notation, let $\pi\colon \cX\to \cQ$ denote the projection.
Given a probability measure $\mu$ on $\cQ$, we will be interested in probability measures on $\cX$ that project to $\mu$.
The space of all such measures will be denoted $\pi_{*}^{-1}(\mu)$, it is compact in the weak topology.

\subsubsection{The centralizer of $A$}
	\label{sssec:the_centralizer_of_a}
Recall from \autoref{sssec:the_groups_of_interest} that $K$ is a maximal compact and $M$ is the centralizer in $K$ of the $A$-action.
For this section only, all measures on $\cX$ and $\cQ$ will be assumed to be $M$-invariant, or equivalently we work on the spaces $\cX/M$ and $\cQ/M$ and only consider the right action of $A$.
For example, $\cQ/M=\leftrightquot{\Gamma}{G}{M}$ is the unit tangent bundle of $\cM$, denoted from now on $T^1\cM$, and the right action of $A$ gives the geodesic flow.

\subsubsection{The space of geodesics}
	\label{sssec:the_space_of_geodesics}
The stabilizer of an oriented hyperbolic geodesic in $\wtilde{\cM}$ is the group $MA$, with $A$ acting as translation along the geodesic and $M$ acting as a rotation along the axis of the geodesic.
Therefore $G/MA$ is naturally identified with the space of geodesics in $\wtilde{\cM}$.
Furthermore, by the correspondence principle in \autoref{sssec:correspondence_principle}, a $\Gamma$-invariant locally finite measure on $G/MA$ is equivalent to an $A$-invariant locally finite measure on $\leftrightquot{\Gamma}{G}{M}=:T^1\cM$.
We will only work with finite $A$-invariant measures on $T^1\cM$.

\subsubsection{Bruhat cell and boundaries}
	\label{sssec:bruhat_cell_and_boundaries}
If we denote by $U^{\pm}$ the stable and unstable unipotent subgroups associated to $A$, then the embedding
\[
	\rightquot{G}{MA} \into \rightquot{G}{MAU^+}\times \rightquot{G}{MAU^-}
\]
identifies the space of oriented geodesics with the space of points on the boundary of $\wtilde{\cM}$ that are transverse.
Equivalently, this is the open Bruhat cell on the right-hand side \cite[VII.4]{Knapp2002_Lie-groups-beyond-an-introduction}.

Let us denote the image by $\partial^{(2)}\wtilde{\cM}$ and observe that it consists of pairs of transverse points on the boundary:
\[
	\partial^{(2)}\wtilde{\cM} =
	\Big\lbrace
	(l_+, l_-)\in \bP(N)^2 \colon
	l_+^2 = 0 =  l_-^2 \text{ and } l_+\wedge l_-\neq 0
	\Big\rbrace
\]
where $l_{\pm}$ are isotropic lines in the \Neron--Severi group.
We can also make the identification with pairs of actual cohomology classes $([\eta_+],[\eta_-])\in N\times N$ satisfying the condition $[\eta_+]\wedge[\eta_-]=1$ modulo the equivalence $([\eta_+],[\eta_-])\sim \left(e^\lambda[\eta_+],e^{-\lambda}[\eta_-]\right)$ for $\lambda\in \bR$ and $[\eta_\pm]$ in the boundary of the ample cone.

Finally, let $\partial^{(2)}_{irr}\wtilde{\cM}$ denote the subset of $\partial^{(2)}\wtilde{\cM}$ consisting of pairs of rays both of which are irrational.

\begin{proposition}[Finite measures are supported on irrational directions]
	\label{prop:finite_measures_are_supported_on_irrational_directions}
	Suppose that $\mu$ is a finite $A$-invariant measure $T^1\cM$.
	Then $\Gamma$-invariant measure on $G/MA$ associated to $\mu$ by the correspondence principle and denoted $\mu_{G/MA}$, satisfies:
	\[
		\mu_{G/MA}\left(\partial^{(2)}\wtilde{\cM}\setminus \partial^{(2)}_{irr}\wtilde{\cM}\right) = 0
	\]
	or in other words, it is carried by the irrational directions.
\end{proposition}
\begin{proof}
	The set excluded from $\partial^{(2)}_{irr}\wtilde{\cM}$ consists of pairs of the form $[E]\times (\bP(N)\setminus [E])$ and those obtained by flipping the two factors, where $E$ is an integral isotropic vector.
	Since there are countably many such sets, arguing by contradiction it would follow that one such set has positive $\mu_{G/MA}$-measure.

	But then going back through the correspondence principle construction, it would follow that $\mu$ has infinite mass in the cusp associated to $[E]$.
\end{proof}

Note that there are locally finite $A$-invariant but infinite measures on $T^1\cM$, such as those supported on a geodesic that goes between two cusps of $\cM$.

\subsubsection{Canonical lifts}
	\label{sssec:canonical_lifts_of_maximal_entropy}
Let now $\mu$ be an ergodic, $A$-invariant probability measure on $T^1\cM$.
Using the equivariant family of currents on the boundary of the ample cone, we construct a lift $\mu_{\cX}$ of $\mu$ to an $A$-invariant measure on $\cX$.
First, observe that as a consequence of \autoref{thm:continuous_family_of_boundary_currents}, which equivariantly assigns closed positive currents to each boundary class, we have:

\begin{corollary}[Equivariant measures]
	\label{cor:equivariant_measures}
	There exists a $\Gamma$-equivariant measurable map
	\[
		\mu_{mme}\colon \partial^{(2)}_{irr}\wtilde{\cM}
		\to
		\cP(X)
	\]
	were $\cP(X)$ denotes the probability measures on $X$.
\end{corollary}
\begin{proof}
	We can associate to a pair of cohomology classes $[\eta_\pm]$ the currents $\eta_{\pm}$ provided by \autoref{thm:continuous_family_of_boundary_currents}, and define
	\[
		\mu_{mme}([\eta_+],[\eta_-]):=\eta_+\wedge \eta_-
	\]
	where the wedge-product is well-defined since each of $\eta_\pm$ has continuous potentials and is positive, while the normalization $[\eta_+]\wedge[\eta_-]=1$ ensures we get a probability measure, and the scaling equivariance of the canonical currents implies the measure is independent of the choice of representatives for the pair of cohomology classes.

	The map $\mu_{mee}$ is measurable since it is the restriction to a measurable subset of a continuous map.
\end{proof}

\begin{theorem}[Canonical lifts of measures]
	\label{thm:canonical_lifts_of_measures}
	Let $\mu$ be an $A$-invariant probability measure on $T^1\cM$.
	Then there exists a measure $\mu_\cX$ on $\cX/M$ which pushes forward to $\mu$ and which is also $A$-invariant (and $M$-invariant), with fiberwise conditional measures given by \autoref{cor:equivariant_measures}.
\end{theorem}
\begin{proof}
	Indeed, $\mu$ gives rise to the $\Gamma$-invariant measure $\mu_{G/MA}$ on $G/MA$, and this space is identified with $\partial^{(2)}\wtilde{\cM}$.
	Furthermore, by \autoref{prop:finite_measures_are_supported_on_irrational_directions} the measure $\mu_{G/MA}$ is supported on the irrational points, and we can form
	\[
		\mu_{G/MA}\otimes \mu_{mme} \text{ on }\partial^{(2)}\wtilde{\cM}\times X
	\]
	which is $\Gamma$-equivariant, and by the correspondence principle again yields an $A$-invariant measure on $\cX/M$.
\end{proof}

\begin{remark}[Maximal entropy]
	\label{rmk:maximal_entropy}
	A subsequent work will establish that $\mu_{\cX}$ has relative entropy $1$ for $A$, i.e. that the difference between the entropy of $\mu$ and that of $\mu_{\cX}$ is $1$.
	Furthermore $1$ is the maximum relative entropy among lifts of $\mu$, and $\mu_{\cX}$ is unique with this property.

	This puts together and generalizes the Gromov--Yomdin theorem \cite{Gro,Yom} which says that the entropy of a hyperbolic automorphism $\gamma$ is equal to the spectral radius of $\gamma^*$ acting on $H^{1,1}$.
	Indeed, we can take $\mu$ to be the probability measure supported by the closed geodesic corresponding to $\gamma$ and observe that for a flow such as $g_t$ (generating $A$) the relation between entropies of different elements is $h(g_t)=|t|h(g_1)$, and the spectral radius of $\gamma^*$ is precisely the length of the associated geodesic.
\end{remark}



\subsection{A consequence of homogeneous dynamics}
	\label{ssec:a_consequence_of_homogeneous_dynamics}

\subsubsection{Setup}
	\label{sssec:setup_a_consequence_of_homogeneous_dynamics}
We keep the notation for Lie groups from the previous sections.
Our goal is to explain how rigidity theorems in homogeneous dynamics imply the following statement:

\begin{theorem}[Orbit closure dichotomy]
	\label{thm:orbit_closure_dichotomy}
	Let $[\alpha]\in \partial \Amp(X)$ be a point in the boundary of the ample cone.
	Then either its orbit $\Aut(X)\cdot[\alpha]$ is dense in the ample cone, or $[\alpha]$ is proportional to an integral vector $[E]$ and its orbit is discrete.
\end{theorem}
\begin{proof}
	By the correspondence principle \autoref{sssec:correspondence_principle} and passing to finite-index subgroups, it suffices to show the same dichotomy for $H$-orbits on $\leftquot{\Gamma}{G}$, where $H$ is the stabilizer of an isotropic vector $[\alpha]\in \partial \Amp(X)$.
	In our earlier notation, the subgroup is $H=MU^+$, where $M\subset K$ is the centralizer of $A$ and $U^+$ is a horospherical subgroup.
	Note that $M$ also normalizes $U^+$, in particular it acts on $U^+$-orbits.

	Now $U^+$ is a unipotent group, in fact a maximal horospherical subgroup, so Dani's theorem \cite[Thm.~A]{Dani1986_Orbits-of-horospherical-flows} (see also \cite[Thm.~8.2]{Dani1981_Invariant-measures-and-minimal-sets-of-horospherical-flows} where incidentally Raghunathan's conjectures are stated, that eventually became Ratner's theorems) applies to show that $U^+$-orbit closures are homogeneous under some connected Lie subgroup $L$ with $U^+\subseteq L\subseteq G$.
	Furthermore
	there exists $U'\subset U^+$ which is normal in $L$ such that $L/U'$ is reductive, and
	$L$ is unimodular (so the $L$-orbit can admit a finite invariant measure).

	If $L\subset MU^+$ with the closed orbit $\Gamma gL\subset \leftquot{\Gamma}{G}$, then since $gLg^{-1}\cap\Gamma$ in a lattice in $gLg^{-1}$, and by the structure of lattices in compact-by-unipotent groups, it follows that also $gU^+g^{-1}\cap\Gamma$ is a lattice in $gU^+g^{-1}$.
	Therefore $gU^+g^{-1}$ stabilizes an integral isotropic vector $[E]$.
	It follows that $gMg^{-1}$ also stabilizes $[E]$ and hence the orbit $\Gamma gH=\Gamma gU^+M$ is closed, and $\Gamma gL$ is contained in this closed orbit.

	Because $G$ has real rank $1$, we now check that the only other possibility is $L=G$.
	It suffices to look at the Lie algebra decomposition
	\[
		\frakg = \fraku^+ \oplus (\frakm \oplus \fraka) \oplus \fraku^-
	\]
	Suppose that the Lie algebra of $L$, denoted $\frakl$, is not entirely contained in $\fraku^+ \oplus \frakm$, but contains $\fraku^+$.
	Considering the Lie algebra $\fraku'\subset \fraku^+$ such that $\frakl/\fraku$ is reductive, we see that we must have $\frakl/\fraku'=\fraku^{+}{}''\oplus \frakm'' \oplus \fraka''\oplus \fraku^{-}{}''$, where the $''$ parts come from the corresponding subspaces in the decomposition.
	By the assumption that $\frakl$ is not contained in $\fraku^+\oplus \frakm$ it follows that $\fraka''\neq 0$ (if $\frakl$ only had something in $\fraku^-$, we could still obtain elements in $\fraka$ by commutators with $\fraku^+$).
	Thus in fact $\fraka\subset \frakl$, since $\dim \fraka = 1$.
	Furthermore $\fraku^{-}{}''\neq 0$ since otherwise $\frakl$ would not be unimodular (it would be almost a parabolic, up to some compact parts in $\frakm$).
	Finally, an explicit calculation with commutators of matrices shows that by taking commutators of a vector in $\fraku^{-}$ with arbitrary vectors in $\fraku^+$, and then subsequent commutators, we can generate first all of $\frakm$ and then all of $\fraku^-$.
	So it must be that $\frakl=\frakg$ as claimed.
\end{proof}




\section{Applications of heights}
	\label{sec:applications_of_heights}

\subsubsection*{Outline of section}
We now consider some elementary applications of the canonical heights constructed in \autoref{thm:canonical_heights_on_the_boundary}.
In \autoref{ssec:invariants_associated_to_canonical_heights} we construct an invariant of $\Aut(X)$-orbits of rational points.
This can be viewed as analogous to the invariant constructed by Silverman \cite[Thm.1.2]{Silverman_Rational-points-on-K3-surfaces:-a-new-canonical-height} for a single hyperbolic automorphism, which in his case is the \emph{product} $\hat{h}^+(p)\cdot \hat{h}^-(p)$ of the canonical heights.

In \autoref{ssec:small_points_and_bounded_orbits} we analyze the situation when a point has vanishing canonical height, for an irrational direction on the boundary.
It is established in \autoref{thm:zero_height_implies_bounded_orbit_irrational_case} that this is equivalent to the point having finite orbit under an appropriate sequence of automorphisms.
This can be seen as analogous to results of Silverman for a single automorphism, and to results about torsion points in elliptic fibrations for parabolic automorphisms.


\subsection{Invariants associated to canonical heights}
	\label{ssec:invariants_associated_to_canonical_heights}

\subsubsection{Setup}
	\label{sssec:setup_invariants_associated_to_canonical_heights}
We maintain the conventions from \autoref{sec:boundary_heights_direct_construction}.
Let $h^{can}\colon \partial^\circ \Amp_c(X)\to \cH(X)$ be the assignment of canonical heights constructed in \autoref{thm:canonical_heights_on_the_boundary}.
We will denote by $[\alpha]$ an element of $\partial^\circ \Amp_c(X)$ and by $p$ an element of $X(\ov{k})$, and we will denote by $h^{can}([\alpha],p)$ the associated height.
This is $\Aut(X)$-equivariant, and also scales as $h^{can}(\lambda[\alpha],p)=\lambda\cdot h^{can}([\alpha],p)$.

We will also keep the Lie-theoretic notions introduced in \autoref{sec:suspension_space_construction} and particularly the groups defined in \autoref{sssec:the_groups_of_interest}.

\subsubsection{The space of horocycles}
	\label{sssec:the_space_of_horocycles}
The stabilizer in $G$ of a null-vector in $\NS(X)$ is the group $MN$, therefore we can identify in a $\Gamma$-equivariant manner $\partial \Amp(X)$ (the usual boundary of the ample cone) with $G/MN$.
Note that there is a well-defined right action of $A$, since $A$ normalizes $MN$, which commutes with the left action of $\Gamma$.
Under the identification with null-vectors, this is nothing but the scaling action of $\bR_+^\times$.
Our goal will be to associate invariants (for the left $\Gamma$-action) to functions that are homogeneous for the $A$-action.

Both $G$ and $MN$ are unimodular groups, therefore there is a $G$-invariant measure on $G/MN$, unique up to scaling.
Here is an explicit description in coordinates.
Take the quadratic form to be $x_0^2-x_1^2-\cdots - x_{\rho-1}^{2}$ and identify the boundary of the ample cone with $\bR^{\rho-1}\setminus 0$ via
\[
	(x_1,\ldots,x_{\rho-1})\mapsto \left(\left(x_1^2+\cdots + x_{\rho-1}^2\right)^{\tfrac 12},x_1,\ldots,x_{\rho-1}\right)
\]
and the square root is the positive one.
Then the differential form
\begin{align}
	\label{eqn:invariant_measure}
	\frac{dx_1\wedge\cdots\wedge dx_{\rho-1}}{\left(x_1^2+\cdots +x_{\rho-1}^2\right)^{\tfrac 12}}	
\end{align}
is clearly $\SO_{\rho-1}(\bR)$-invariant, and if we check that it is invariant also under $A$ which only acts on the first two coordinates by
\[
	\begin{bmatrix}
		x_0\\
		x_1
	\end{bmatrix}
	\mapsto
	\begin{bmatrix}
		\cosh(t)x_0 + \sinh(t)x_1\\
		\sinh(t)x_0 + \cosh(t)x_1
	\end{bmatrix}
\]
the invariance under the full group $G$ follows (since $G$ admits the Iwasawa, or $KAK$ decomposition).
However, identifying the denominator in \autoref{eqn:invariant_measure} with $x_0$, an immediate calculation gives that $\frac{dx_1\wedge \cdots \wedge dx_{\rho-1}}{x_0}$ is invariant by the above transformation and the result follows.
To check the invariance one needs to use that $x_0 dx_0 = \sum_{i=1}^{\rho-1} x_i dx_i$ and the above differential form can also be obtained as the residue of $\frac{dx_0\wedge\cdots\wedge dx_{\rho-1}}{x_0^2-x_1^2 - \cdots - x_{\rho-1}^2}$ along the null hypersurface.

\subsubsection{The null quadric}
	\label{sssec:the_null_quadric}
We can also view the boundary of $\Amp^1(X)$ as the null quadric $Q\subset \bP(\NS(X))$ in the projectivized \Neron--Severi group.
Now the projective space $\bP:=\bP\left(\NS(X)\right)$ carries the tautological bundle $\cO(-1)$.
Then the null vectors are the total space of $\cO(-1)\vert_Q \to Q$, and the boundary of the positive cone is one of its two connected components (all over $\bR$).
The quadric itself is given as the vanishing locus of a section of $\cO_{\bP}(2)$ so its canonical bundle, by the adjunction formula, is $\cO_{\bP}(-\rho+2)\vert_{Q}$.

\subsubsection{An invariant of $\Gamma$-orbits}
	\label{sssec:an_invariant_of_gamma_orbits}
Let us now identify the irrational rays on $Q$ with the irrational part of the boundary $\partial^{\circ}_{irr}\Amp(X)$ and denote this set, and its rays, by $Q_{irr}$ and $\cO(-1)\vert_{Q_{irr}}$.
The case $\rho=3$ is special and we don't need to restrict to irrational rays, since then the boundaries of interest agree, see the discussion in \autoref{sssec:some_examples}.
In general, for the constructions we make below, ignoring a set of Lebesgue measure zero, such as that of irrational directions, will be irrelevant.

For a point $p\in X(\ov{k})$ the height function $h^{can}(-,p)$ gives a $1$-homogeneous function on $\cO(-1)\vert_{Q_{irr}}$, i.e. it provides us with a measurable section of $\cO(1)\vert_Q$, let us call this section $s_p$.
Because the canonical bundle of $Q$ is $\cO(-\rho+2)$, if we raise $s_p$ to the power $-\rho+2$ we can then integrate it.
Recall that by our standing assumption $\rho\geq 3$.

\begin{theorem}[Invariance of total height]
	\label{thm:invariance_of_total_height}
	For a point $p\in X\left(\ov k\right)$, the integral
	\[
		\int_{Q}s_p^{-\rho+2}=:h^{tot}(p)
	\]
	is independent of the choice of representative in the $\Gamma$-orbit, i.e. we have
	\[
		h^{tot}(p)=h^{tot}(\gamma p) \quad \forall \gamma \in \Gamma.
	\]
	In fact this holds for any $\gamma\in \Aut(X)$.

	Furthermore, the invariant can be alternatively described as follows.
	Consider the subset
	\[
		S_{p}:=\{[\alpha]\in \cO(-1)\vert_{Q_{irr}}\colon h^{can}([\alpha],p)\leq 1\}
	\]
	Then up to a constant $c(\rho)$ that only depends on $\rho$, we have that
	\[
		\Vol(S_p) = c(\rho)\cdot \frac{1}{h^{tot}(p)}
	\]
	where $\Vol$ is a fixed $G$-invariant volume on the boundary of the ample cone, e.g. the one in \autoref{eqn:invariant_measure}.
\end{theorem}

The boundaries of the sets $S_p$ are depicted in \autoref{fig:star_shaped_sets} in blue, using computer simulations.

\begin{proof}
	The $\Gamma$-invariance of $h^{tot}(p)$ follows from the $G$-invariance of the operation of integration of differential forms, i.e. $\int_{Q}s_p^{-\rho+2}$.
	In fact this is further invariant under the larger group of orthogonal transformations preserving the positive cone (this group has two connected components) but it contains the entire automorphism group $\Aut(X)$.

	For the relation between the volume of $S_p$ and $h^{tot}(p)$, it suffices to remember only the measurable section $s_p$.
	Then if we change $s_p$ to $\lambda s_p$, then $\left(\int_Q s_p^{-\rho+2}\right)$ rescales by $\lambda^{-\rho+2}$.
	Similarly, the set $S_p$ changes to $\frac{1}{\lambda}S_p$ and its volume changes by $\frac{1}{\lambda^{-\rho+2}}$.
	Both assignments are $G$-invariant, and clearly given locally on $Q$ by integration, so they must be the same up to some fixed constant that depends on the $G$-invariant measure fixed on $G/MN$.
\end{proof}

\begin{remark}[Relation to Silverman's invariant]
	\label{rmk:relation_to_silverman_s_invariant}
	Recall that in \cite{Silverman_Rational-points-on-K3-surfaces:-a-new-canonical-height}, starting from a hyperbolic automorphism $\gamma$ with expanded/contracted directions $[\alpha_{\pm}]$, and given a point $p\in X(\ov{k})$, Silverman defined the quantity $h^{can}([\alpha_+],p)\cdot h^{can}([\alpha_-],p)$ and verified that it is an invariant associated to the $\gamma$-orbit of $p$.
	His construction corresponds to the critical case $\rho=2$ in the above constructions.
	The space of sections of $\cO(1)\vert_{Q}$ is $2$-dimensional and identified naturally with $\NS(X)$, and it admits a $G$-invariant quadratic form of signature $(1,1)$.
	Silverman's invariant is the quadratic form applied to the section $s_p$.
\end{remark}



\subsection{Small points and bounded orbits}
	\label{ssec:small_points_and_bounded_orbits}

\subsubsection{Setup}
	\label{sssec:setup_small_points_and_bounded_orbits}
Suppose that $\gamma\in \Aut(X)$ is a hyperbolic automorphism and $[\alpha_+]$ is the expanded direction.
Silverman in \cite{Silverman_Rational-points-on-K3-surfaces:-a-new-canonical-height} proved that a point $p\in X\left(\ov k\right)$ satisfies $h^{can}([\alpha_+],p)=0$ if and only if $p$ is periodic for $\gamma$ (and hence also $h^{can}([\alpha_-],p)=0$ if $[\alpha_-]$ denotes the contracted direction).
We will now develop some analogous corollaries, both for the irrational boundary points, and for the ones contained in the rational directions associated to elliptic fibrations.

\subsubsection{The irrational boundary points}
	\label{sssec:the_irrational_boundary_points}
Suppose that $[\alpha]\in \partial^\circ\Amp_c(X)$ is an irrational direction.
Using the fixed generating set in \autoref{sssec:fixed_finite_set_of_generators}, there is therefore a sequence of elements $\gamma_1,\gamma_2,\ldots$ among them such that the appropriately normalized basepoint ample class $[L]$ converges to $[\alpha]$ under the iterations $\gamma_1^*\cdots \gamma_n^*[L]$.
To a point $p\in X(\ov{k})$ we can associate the sequence of points $p_n:=\gamma_n\cdots \gamma_1 p$.
Observe that this sequence is, up to finite ambiguity, independent of the choice of generators.
Namely, if we picked a different set of generators $\gamma_{\bullet}'$, considered the new sequence $\gamma_i'$ associated to $[\alpha]$, and hence sequence of points $p_n':=\gamma_n'\cdots \gamma_1'p$, then there exists a \emph{finite set} of automorphisms $\beta_1,\ldots,\beta_K$, depending only on the two fixed sets of generators, such that $p_n'=\beta_{\psi(n)}p_{\phi(n)},\forall n$ for appropriate $\phi(n),\psi(n)$.
Indeed this already holds at the level of group elements, i.e. $\gamma_n'\cdots \gamma_1' = \beta_{\psi(n)}\cdot \gamma_{\phi(n)}\cdots \gamma_1$.

\begin{theorem}[Zero height implies bounded orbit]
	\label{thm:zero_height_implies_bounded_orbit_irrational_case}
	Suppose that $[\alpha]\in \partial^\circ\Amp_c(X)$ is an irrational class and $p\in X\left(\ov k\right)$ is a point.
	Then the canonical height satisfies $h^{can}([\alpha],p)=0$ if and only if the set of points $\{\gamma_n\cdots \gamma_1p\}_{n\geq 1}$ is finite, where $\gamma_i$ is the sequence of generators associated to $[\alpha]$ by \autoref{sssec:fixed_finite_set_of_generators}.
\end{theorem}

\begin{proof}
	Fix an ample class $[L]$ with $[L]^2=1$ and a height function $h_{[L]}$ in this class.
	Then the converse direction is immediate: if the set of points is finite, then the set of heights $h_{[L]}(p_n)$ is finite, so the limit of $\frac{1}{M(\gamma_1^*\cdots \gamma_n^*[L])}h_L(p_n)$, which defines the canonical height, is clearly zero.
	To show that the vanishing of the height implies the set of points is finite, we will show that the set of heights $h_{[L]}(p_n)$ is bounded, with a bound independent of $n$, so the result will follow by the Northcott property of heights associated to ample line bundles.

	\noindent \textbf{$[L]$-normalized representatives.}
	Suppose that $[\alpha_+]$ satisfies $[\alpha_+]^2=0$ and $[\alpha_+].[L]=\tfrac 12$; call such classes \emph{$[L]$-normalized isotropic}.
	Then there exists a unique $[L]$-normalized and isotropic class $[\alpha_-]$, which we will also denote by $op_{[L]}([\alpha_+])$ such that $[L]=[\alpha_+]+[\alpha_-]$.
	We can view $[\alpha_-]$ as the ``opposite'' point on the boundary to $[\alpha_+]$, with respect to the interior basepoint $[L]$.
	We then have the following family of heights associated to $[L]$, namely:
	\[
		h^{can}_{[\alpha_+]} + h^{can}_{op_{[L]}([\alpha_+])} \text{ as }[\alpha_+] \text{ ranges over $[L]$-normalized isotropic classes.}
	\]
	When $[\alpha_+]$ or its opposite land at a rational point, we can take an arbitrary height $h^{can}$ corresponding to that elliptic fibration.
	Then by the analogue for heights of \autoref{prop:boundedness_of_potentials}, it follows that all of these heights are within a uniform constant of the fixed height $h_{[L]}$.
	So it suffices to find for each $p_n$ an $[L]$-normalized isotropic class $[\alpha_{+,n}]$ (with $[\alpha_{-,n}]:=op_{[L]}([\alpha_{+,n}])$) such that $h^{can}([\alpha_{+,n}],p_n)+h^{can}([\alpha_{-,n}],p_n)$ is bounded by a constant independently of $n$.

	\noindent \textbf{Dynamics and choice of classes.}
	Set $g_n:=\gamma_1^*\cdots \gamma_n^*$ viewed as an orthogonal transformation in $\SO(\NS(X))$.
	Then we have, using the equivariance property of canonical heights:
	\begin{multline*}
		h^{can}([\alpha_{+,n}],p_n)+h^{can}([\alpha_{-,n}],p_n)  = \\
		= h^{can}([\alpha_{+,n}],\gamma_n\cdots\gamma_1p)+h^{can}([\alpha_{-,n}],\gamma_n\cdots \gamma_1 p)\\
		= h^{can}(g_n[\alpha_{+,n}],p)+h^{can}(g_n[\alpha_{-,n}],p)
	\end{multline*}
	Let us now choose $[\alpha_{+,n}]$ such that $g_n[\alpha_{+,n}]$ is proportional to the class $[\alpha]$ for which we know that $h^{can}([\alpha];p)=0$.
	Note that the constant of proportionality is going to be large (increasing with $n$), but the vanishing of the height still holds.
	The boundedness of heights will then follow if we show that $g_n[\alpha_{-,n}]$ is a vector whose norm (relative the fixed basepoint $[L]$) is bounded independently of $n$.
	The uniform boundedness of $g_n[\alpha_{-,n}]$ and ampleness of $[L]$ will imply $h_{L}\geqapprox h_{g_n[\alpha_{+,n}]}^{can}+O(1)$ and the result will follow, since $h(L;p)$ is a fixed constant.

	\noindent \textbf{Estimating the opposite vector.}
	The above argument has reduced the claim to an estimate in linear algebra.
	Specifically, we have that
	\begin{align}
	\label{eqn:L_gnL}
	\begin{split}
		[L] & = [\alpha_{+,n}] + [\alpha_{-,n}]\\
		g_n[L] & = e^{A_n}[\alpha] + g_n[\alpha_{-,n}] \quad \text{ so }g_n[\alpha_{n,+}]=e^{A_n}[\alpha]
	\end{split}
	\end{align}
	by choice of classes $[\alpha_{\pm,n}]$.
	The construction of the group element $g_n$ is such that, in hyperbolic space, the point $g_n[L]$ is a uniformly bounded distance away from the geodesic connecting $[L]$ with the boundary point determined by $[\alpha]$.
	Let us call $[\alpha_+]:=[\alpha]$ and $[\alpha_-]:=op_{[L]}([\alpha])$.
	Then the last geometric statement in hyperbolic space translates as follows.

	Let $m_n:=\norm{g_n}$, for the norm relative to the fixed basepoint $[L]$.
	Let $a_n\in \SO(\NS(X))$ be the element which acts by $[\alpha_\pm]\mapsto e^{\pm m_n}[\alpha_{\pm}]$ and the identity on the orthogonal complement.
	Then the property that $g_n[L]$ is a bounded distance from the geodesic connecting $[L]$ and $[\alpha_+]$ is equivalent to the existence of $h_n\in \SO(\NS(X))$, of uniformly bounded norm, such that
	\[
		g_n = a_n h_n \quad \text{ as elements of }\SO(\NS(X)).
	\]
	Indeed, there exists an $h_n'$ of uniformly bounded size such that $g_nh_n'[L]=a_n[L]$ for the action on hyperbolic space (this is the statement about uniform distance to the geodesic) and so $a_n^{-1}g_nh_n'$ stabilizes $[L]$, hence must be in the corresponding compact subgroup:
	\[
		a_n^{-1} g_n h_n' = k_n \in \Stab_{[L]}(\SO(\NS(X)))
	\]
	Thus $g_n = a_n k_n \left(h_n'\right)^{-1}$ as claimed.

	Consider now the action of $h_n$ on $[L]=[\alpha_+]+[\alpha_-]$:
	\[
		h_n[L] = c_n[\alpha_+] + d_n[\alpha_-] + v_n
	\]
	where, by boundedness of $h_n$, we see that $c_n, d_n$ and $\norm{v_n}$ are uniformly bounded above, and $v_n$ is orthogonal to the span of $[\alpha_{\pm}]$.
	It follows that:
	\[
		g_n[L]=a_n h_n[L] = c_ne^{m_n}[\alpha_+] + d_n e^{-m_n}[\alpha_-]+v_n
	\]
	so it must be that $g_n[\alpha_{-,n}] = d_n e^{-m_{n}}[\alpha_-]+v_n$, which is of bounded norm, as required.
\end{proof}

\subsubsection{Heights along elliptic fibrations}
	\label{sssec:heights_along_elliptic_fibrations}
Let now $[E]$ be the class of a fiber in a genus one fibration $X\xrightarrow{\pi}B$ and $p\in X\left(\ov k\right)$.
Set also $V({[E]}):=[E]^{\perp}/[E]$ viewed as a real vector space, with a $\bZ$-lattice inside, which we can identify with a finite-index subgroup of $\Aut_\pi(X)$ as we now do.
By \autoref{thm:height_valued_pairing}, we have a quadratic form:
\[
	h^{vcan}(-,p)\colon V([E])\to \bR
\]
which only depends on the projection $b:=\pi(p)\in B$, so we will write it as $h^{vcan}(-,b)$.
In fact, if $\Jac(X)\to B$ denotes the Jacobian fibration associated to the original genus one fibration, there is a natural specialization map $\Aut_{\pi}^{\circ}(X)\to \Jac(X)_{b}(k)$ (recall that we assume all automorphisms are already defined over $k$) and the above quadratic form is just the canonical height on the elliptic curve $\Jac(X)_{b}$ restricted to the image of $\Aut_{\pi}^{\circ}(X)$.
The next result collects several known facts, due to Silverman and Tate \cite{Tate,Silverman1994_Variation-of-the-canonical-height-in-algebraic-families}.

\begin{theorem}[Kernel of quadratic form and torsion]
	\label{thm:kernel_of_quadratic_form_and_torsion}
	\leavevmode
	\begin{enumerate}
		\item A parabolic automorphism $\gamma\in \Aut_\pi^{\circ}(X)\toisom V([E])$ has $h^{vcan}(\xi(\gamma);b)=0$ if and only if it is of finite order when restricted to the fiber $X_b$.
		\item The kernel of the quadratic form $h^{vcan}(-;b)$ is a rational subspace inside $V([E])$.
		\item \label{item:kernel_rank1_uniformity}
		If $B\isom \bP^1$ then the kernel of the quadratic form can have dimension $\geq 1$ for only finitely many values of $b\in X(k)$.
	\end{enumerate}
\end{theorem}
\begin{proof}
	Part (i) is immediate from the classical properties of canonical heights on the elliptic curve $\Jac(X)_b$, and part (ii) similarly identifies the kernel of the quadratic form with the kernel of the specialization map $\Aut_\pi^\circ(X)\to \Jac(X)_b(k)$.

	Part (iii) follows from the results of Silverman and Tate, see \autoref{ssec:the_pairing_from_the_variation_of_canonical_height}.
	Indeed, the height functions $h^{vcan}(v;-)\colon B\left(\ov{k}\right)\to \bR$, as $v$ varies in $V([E])$ subject to the normalization $[v]^2=-1$, are within uniform constants of each other, and in particular for a fixed height $h_{L_B}$ on $\cO_{B}(1)$, we will have that $h^{vcan}(v;b)+C\geq h_{L_B}(b)$ for a constant $C$ independent of $b$.
	Therefore, the set of points $b\in B\left(\ov k\right)$ where $h^{vcan}(v;b)$ vanishes \emph{for some $v\in V([E])\setminus 0$} is of bounded height, and hence there are only finitely many such in $B(k)$.
\end{proof}

\begin{remark}[On the results of Masser and Zannier]
	\label{rmk:on_the_results_of_masser_and_zannier}
	Based on the results of Masser--Zannier \cite{MasserZannier2012_Torsion-points-on-families-of-squares-of-elliptic-curves,MasserZannier2008_Torsion-anomalous-points-and-families-of-elliptic-curves}, it seems plausible also that the kernel of the quadratic form can have dimension $\geq 2$ for only finitely many values of $b\in X\left(\ov k\right)$.
	Indeed, what they proved is that any specific $2$-dimensional subspace inside $V([E])$ can be in the kernel at only finitely many $b\in B\left(\ov k\right)$.

	Note also that the result in \ref{item:kernel_rank1_uniformity} of \autoref{thm:kernel_of_quadratic_form_and_torsion} above is frequently only stated for a single section or automorphism of the elliptic fibration, but the proof of the slightly stronger statement, allowing for any linear combinations, is analogous.
\end{remark}







\bibliographystyle{sfilip_bibstyle}
\bibliography{canonical_currents_heights}

\end{document}